\documentclass[11pt,reqno,a4paper]{amsart}
\usepackage{latexsym,bm}
\usepackage{amsmath,amssymb,cases}
\usepackage{amsthm}
\usepackage{xcolor}
\usepackage{enumerate}
\usepackage{caption}
\usepackage{graphicx,subfig}
\usepackage[numbers,sort&compress]{natbib}
\usepackage{accents}

\usepackage{geometry}
\usepackage{ulem}
\usepackage[thicklines]{cancel}

\newtheorem{theorem}{Theorem}[section]
\newtheorem{lemma}{Lemma}[section]

\newtheorem{definition}{Definition}[section]

\newtheorem{remark}{Remark}[section]

\graphicspath{{fig/}}

\numberwithin{equation}{section}
\numberwithin{figure}{section}

\begin{document}
\captionsetup[figure]{labelfont={default},labelformat={default},labelsep=period,name={Fig.}}

\title[Entropy Conditions with Non-Lipschitz Source Term]{Well-Posedness 
of the Cauchy Problem \\
for First-order Quasilinear Equations \\
with Non-Lipschitz Source Terms \\ 
and Its Applications}

\author{Gaowei Cao}
\address{Gaowei Cao: Wuhan Institute of Physics and Mathematics, 
	Innovation Academy for Precision Measurement Science and Technology, 
	Chinese Academy of Sciences, Wuhan 430071, China}
\email{\tt gwcao@apm.ac.cn}

\author{Gui-Qiang G. Chen$^{\dag}$}
\address{Gui-Qiang G. Chen: Oxford Centre for Nonlinear Partial Differential Equations,
        Mathematical Institute, University of Oxford, Oxford, OX2 6GG, UK}
\email{\tt gui-qiang.chen@maths.ox.ac.uk}

\author{Wei Xiang}
\address{Wei Xiang, Department of Mathematics, City University of Hong Kong, Kowloon, Hong Kong, China}
\email{\tt weixiang@cityu.edu.hk}

\author{Xiaozhou Yang}
\address{Xiaozhou Yang, Wuhan Institute of Physics and Mathematics, 
	Innovation Academy for Precision Measurement Science and Technology, 
	Chinese Academy of Sciences, Wuhan 430071, China}
\email{\tt xzyang@apm.ac.cn}

\date{\today}

\begin{abstract}
We are concerned with the well-posedness 
of the Cauchy problem for the first-order quasilinear equations with non-Lipschitz source terms and the global structures of the multi-dimensional Riemann solutions. 
For such 
quasilinear 
equations with initial data in $L^\infty$,
when the source term $g(t, x, u)$ is only right-Lipschitz (not necessarily left-Lipschitz) in $u$,
we first prove that the Kruzkov
entropy condition is sufficient to guarantee the well-posdeness of entropy solutions.
Next, we analyze  
the structures of global multi-dimensional Riemann solutions
for scalar conservation laws with non-Lipschitz source terms, 
where the Riemann-type initial data consist of two different constant states separated by a smooth hypersurface.
More precisely, we construct the global multidimensional Riemann solutions with non-selfsimilar structures, 
including nonselfsimilar shock waves and nonselfsimilar rarefaction waves, 
and prove that these two kinds of basic waves can be 
expressed via the implicit functions of functional equation determined by the initial discontinuity, 
the flux functions, and the non-Lipschitz source term. 
Moreover, 
we discover two new phenomena:
(i) such kinds of basic waves can disappear in a finite time 
and (ii) the new-type rarefaction wave can contain some weak discontinuities in its interior. 
These behaviors are in stark contrast to the case of the Lipschitz source term,
where these two kinds of basic waves persist globally and the rarefaction waves remain smooth in the interior.
Finally, we provide two examples to respectively demonstrate the uniqueness 
of Riemann solutions
in the case of the source term being non-Lipschitz from left (necessarily right-Lipschitz),
and the non-uniqueness of Riemann solutions in the case of the source term being non-Lipschitz from right.
\end{abstract}

\keywords{First-order quasilinear equations, Non-Lipschitz source terms, Cauchy problem,
Kruzkov entropy condition, Stability and uniqueness, Existence, 
Conservation laws, Balance laws, Well-posedness of characteristic ODEs, 
Riemann problem, Multi-dimensional Riemann solutions with non-selfsimilar structures, 
Multi-dimensional shock waves, Multi-dimensional rarefaction waves, 
Non-uniqueness.\\
$^\dag$Corresponding author.}

\subjclass[2020]{Primary: 35L65, 35L03, 35L60, 35A02, 35A01, 35L67, 35A21, 35B35;
Secondary: 35B30, 35D30, 58K30}
\maketitle

\tableofcontents

\section{Introduction}
We are concerned with the well-posedness
of the Cauchy problem for the first-order quasilinear equations 
with non-Lipschitz source terms for $u=u(t,x) \in \mathbb{R}$:
\begin{eqnarray}
&&u_t+\sum\limits^n_{i=1}f_i(t,x,u)_{x_i}=g(t,x,u)\qquad\,\, \mbox{for $(t,x)\in [0,T]\times \mathbb{R}^n$},\label{1.1}\\[1mm]
&&u|_{t=0}=u_0(x), \label{ID}
\end{eqnarray}
where $u_0(x)\in L^{\infty}(\mathbb{R}^n)$ and $T$ is either a positive number or $\infty$. 

It is well known that, due to the nonlinearity of the flux functions $f_i(t,x,u), i=1, \cdots, n$,
no matter how smooth the initial data function $u_0(x)$ is,
the solution may form shocks generically in a finite time.
Thus, it is necessary to understand the solution in a weak sense, 
namely, the solution as a function with suitable integrability solves 
the Cauchy problem $\eqref{1.1}$--$\eqref{ID}$ in the distributional sense.
In general, since the weak solutions are not unique,
some entropy condition should be employed to 
single out the unique entropy solution among the weak solutions.

For equation \eqref{1.1} in one-dimension with uniformly convex flux function $f(u)\in C^2(\mathbb{R})$, 
the entropy condition, introduced by Oleinik \cite{[OOA]}, 
is sufficient to single out the unique weak solution (physically relevant) among all possible weak solutions. 
In Dafermos \cite{[DCM]} and Lax \cite{[Lax]} (also see the references cited therein),  
such a unique weak solution has been shown to satisfy the regularizing effect that the initial data function in $L^\infty$ 
is regularized to be in $BV_{\rm loc}$ 
instantaneously for the solution, 
and to possess many fine properties such as the regularity,
the decay,
as well as the convergence of approximation schemes, among others. 
For the multi-dimensional case, 
by following the earlier results for entropy solutions in $BV_{\rm loc}$ by Conway--Smoller \cite{[CS]} and Vol$'$pert \cite{[Va]}, 
the existence and stability 
of entropy solutions in $L^\infty$ was established by Kruzkov \cite{[Ksn2]}, 
where the uniqueness was enforced by the so-called Kruzkov entropy condition:
A weak solution $u\in L^\infty$ is an entropy solution in the sense of Kruzkov if,  
for any $k\in \mathbb{R}$, 
\begin{align}\label{eta0}
\big(|u-k|\big)_t&+
\sum_{i=1}^n\big({\rm sgn}(u-k)\big(f_i(t,x,u)-f_i(t,x,k)\big)\big)_{x_i}\nonumber\\
&
-{\rm sgn}(u-k)\Big({-}\sum_{i=1}^nf_{ix_i}(t,x,k)+g(t,x,u)\Big)\leq 0
\qquad \,\, \mbox{in $\mathcal{D}^\prime$}.
\end{align}

For the Riemann solutions of equation $\eqref{1.1}$ with selfsimilar structures, 
by using the generalized characteristic method, Zhang-Zheng \cite{[ZZ2]} (also see \cite{[CH]}) 
constructed the selfsimilar solutions for $2$-D scalar conservation law with the Riemann initial data 
consisting of four pieces of constants separated by the coordinate axes, and Chen-Li-Tan \cite{[CLT]}, Guckenheimer \cite{[Gj1]}, 
and Sheng \cite{[Sw]} constructed the selfsimilar solutions for the case of the Riemann initial data
consisting of two or three pieces of constants, separated by two or three rays; 
see \cite{[Gj2],[Lwb],[Wdh],[YZ],[ZZ1]} for more related results on the selfsimilar solutions.

For the Riemann solutions of equation $\eqref{1.1}$ with non-selfsimilar structures, 
Yang \cite{[Yxz]} firstly constructed the non-selfsimilar solutions for the $n$-D scalar conservation law 
with the Riemann-type initial data consisting of two pieces of constants separated by a hypersurface. 
For the $n$-D scalar conservation law with a Lipschitz source term, 
Cao-Xiang-Yang \cite{[CXY]} obtained the global structures of non-selfsimilar shock waves and non-selfsimilar 
rarefaction waves for the Riemann-type initial data consisting of two pieces of constants separated by a hypersurface.
See \cite{[CHY],[CKXY],[QCYZ]} for more related results on the non-selfsimilar solutions.

\vspace{2pt}
The first motivation of this paper is that some of the first-order scalar equations, 
reduced from 
hyperbolic systems of conservation laws, may not possess a Lipschitz source term, 
such as the multi-dimensional compressible Euler equations 
with spherical symmetry ({\it cf.} \cite{[CSM],[CW],[DWX],[FGXZ],[HLY]}).
Thus, the well-posedness of the Cauchy problem $\eqref{1.1}$--$\eqref{ID}$ with a non-Lipschitz source term  
is an important topic for the mathematical theory of hyperbolic conservation laws. 
The second motivation is that the Riemann solutions for equation $\eqref{1.1}$ with a non-Lipschitz source term 
can behave differently from the Riemann solutions for the case of the Lipschitz source term, 
and there are no related results found in the previous literature. 
The third motivation is that the theory of differential equations out of the smooth context arouses 
a big interest not only in the theoretical importance of such an extension but also in the study of 
various physical models of the mechanics of fluids. 
As examples, these especially include: 
the linear transport equations and the ordinary differential equation (ODE) flows 
(see DiPerna-Lions \cite{[DL]}, Ambrosio \cite{[AL]}, Alberti-Bianchini-Crippa \cite{[ABC]}, 
Ambrosio-Crippa \cite{[AC]}, Ambrosio-Colombo-Figalli \cite{[ACF]}, Bianchini-Bonicatto \cite{[BB]}), 
the divergence-measure fields (see Chen-Frid \cite{[CF99],[CF03]}, Chen-Torres \cite{[CT05]}, 
Chen-Ziemer-Torres \cite{[CZT]}), among others; 
also see Filippov \cite{[FAF]}, Dafermos \cite{[DCMG]}, Ambrosio-Crippa-De Lellis-Otto-Westdickenberg \cite{[ACD]}, 
Bouchut-James-Mancini \cite{[BJM]}, Lions-Seeger \cite{[LS]}, and the references therein.

\smallskip
In this paper, we establish the well-posedness 
of the Cauchy problem $\eqref{1.1}$--$\eqref{ID}$ in $L^{\infty}$ with a non-Lipschitz source 
term and then construct global multidimensional Riemann solutions 
with non-selfsimilar structures. In particular, these include the following results:

\begin{enumerate}
\item[(i)] For the source term $g(t,x,u)$ 
being only right-Lipschitz (not necessarily left-Lipschitz) in $u$, the Kruzkov entropy condition $\eqref{eta0}$ 
is still able to enforce the stability and uniqueness of the solutions 
of the Cauchy problem $\eqref{1.1}$--$\eqref{ID}$, and then the existence of the solutions is also established.
As far as we know, 
This is the first well-posedness result for scalar conservation laws with non-Lipschitz source terms 
in the multi-dimensional case since Kruzhkov’s work in 1970.

\item[(ii)] For the Riemann problem of equation $\eqref{1.1}$ with 
the source term $g(u)$ being only right-Lipschitz (not necessarily left-Lipschitz), 
based on the well-posedness of the related characteristic ODEs established in this paper, 
the Riemann solutions with non-selfsimilar structures, including shock waves and rarefaction waves, 
are constructed by giving their expressions determined by the implicit functions. 
We discover two new phenomena: The shock waves and rarefaction waves can disappear in a finite time, 
and the rarefaction wave can contain some weak discontinuities in its interior (see Remarks \ref{rem:6.1}--\ref{rem:6.3} below).

\item[(iii)]  Two examples for the case of non-Lipschitz source terms are provided. 
One example is to construct the unique Riemann solution
when the source term is non-Lipschitz from left (but Lipschitz from right). 
Another example is to show that the Riemann solutions are not unique when the source term is not right-Lipschitz (but Lipschitz from left).
\end{enumerate}
\noindent

\medskip
A function $h(t,x,u)\in C([0,T]\times\mathbb{R}^n\times\mathbb{R})$ 
is called to be Lipschitz in $u$ on a compact set $\Omega\times[a,b]$ if there exists a constant $C>0$ such that, 
for any $(t,x)\in \Omega$, 
\begin{equation}\label{gnL}
\mathop{ \sup}\limits_{u,v \in [a,b]}\frac{|h(t,x,u)-h(t,x,v)|}{|u-v|}\leq C.
\end{equation}
A function $h(t,x,u)\in C([0,T]\times\mathbb{R}^n\times\mathbb{R})$ 
is called to be right-Lipschitz in $u$ on a compact set $\Omega\times[a,b]$ if there exists a constant $L$ such that, for any $(t,x)\in \Omega$, 
\begin{equation}\label{gnRL}
L(t,x;[a,b]):=\mathop{ \sup}\limits_{u,v \in [a,b]}\frac{h(t,x,u)-h(t,x,v)}{u-v}\leq L.
\end{equation}
Notice that the right-Lipschitz condition $\eqref{gnRL}$ for the source term $g(t,x,u)$ 
is equivalent to saying that $g(t,x,u)\in C([0,T]\times\mathbb{R}^n\times\mathbb{R})$ 
and satisfies that, for any $(t,x)\in \Omega$,
the first-order difference of $g(t,x,\cdot)$ on the bounded interval $[a,b]$ 
possesses a finite supremum. 
In general, such a source term $g(t,x,u)$ allows itself 
to be non-Lipschitz from left in $u$, {\it i.e.}, 
the infimum of the first-order difference of $g(t,x,\cdot)$ 
may equal ${-}\infty$.

Our proofs for the well-posedness of the entropy solutions for the case of the non-Lipschitz source term are motivated by the arguments in Kruzkov \cite{[Ksn2]}, 
in which the stability and uniqueness were proved by the method so-called {\it double variables} and the existence was established by the vanishing viscosity method. 
The proof of the well-posedness of the characteristic ODEs for equation $\eqref{1.1}$ with a non-Lipschitz source term in the form of $g(u)$ is established 
by the subtle and careful analysis of the characteristic ODEs in the various cases classified by the zero point set of $g(u)$. 
The proof of the Riemann solutions with non-selfsimilar structures is based on both {\it Condition} $(\mathcal{H})$ 
as the multi-dimensional version of the convex condition on the flux function in some sense and 
the expressions of multi-dimensional shock waves and rarefaction waves determined by the implicit functions.

\smallskip
This paper is organized as follows:

In \S 2, we first introduce some basic concepts, 
including the weak solutions, the Kruzkov entropy condition,
the geometric entropy condition, and {\it Condition} $(\mathcal{H})$, 
and then present the main theorems of this paper.

In \S 3, we complete the proofs of Theorem $\ref{th:5.1}$
on the stability and Theorem $\ref{th:5.2}$ on the uniqueness of the case that 
the source term $g(t,x,u)$ is only right-Lipschitz (not necessarily left-Lipschitz) in $u$. 
In \S 4, we complete the proof of Theorem $\ref{th:ex}$ 
on the existence for the case that the source term $g(t,x,u)$ is 
only right-Lipschitz (not necessarily left-Lipschitz) in $u$. 

Then, in \S 5, we prove Lemma $\ref{lem:3.1}$ on the well-posedness of the characteristic ODEs for equation $\eqref{1.1}$ 
with non-Lipschitz source terms in the form of $g(u)$ for the subsequent development.
In \S 6, we complete the proof of Theorem $\ref{th:4.4}$ 
on the global multidimensional Riemann solutions with non-selfsimilar structures, 
including the expressions of shock waves and rarefaction waves. 

Finally, in \S 7, we present two examples with the flux functions $(f_1(u), f_2(u))=(\frac{1}{2}u^2, \frac{1}{4}u^4)$ 
and the non-Lipschitz source term $-u^{\frac{1}{3}}$ or $u^{\frac{1}{3}}$, respectively. 
The example with non-Lipschitz source term $-u^{\frac{1}{3}}$ is used to demonstrate the uniqueness of basic waves 
and the new phenomenon that the shock waves and rarefaction waves disappear in a finite time. 
The example with non-Lipschitz source term $u^{\frac{1}{3}}$ is used to demonstrate that the right-Lipschitz 
condition on the source term is necessary to guarantee the uniqueness via showing that, for the source term $u^{\frac{1}{3}}$ 
being non-Lipschitz from right, 
there exist more than one Riemann solutions for both cases of shock waves and rarefaction waves.

\section{Basic Notions and Main Theorems}
In this section, we first introduce some basic concepts and then present the main theorems of this paper. 
Let
\begin{align*}
    \Pi_{_T}:=[0,T]\times \mathbb{R}^n,\quad 
    \mathring{\Pi}_{_T}:=(0,T)\times \mathbb{R}^n
    \qquad\,\, {\rm for}\ T>0.
\end{align*}

Similar to the Einstein summation convention, we use the summation convention that, 
if one of the index variables $i,j,k$ appears twice in a monomial, 
it implies that the summation is taken from $1$ to $n$.

\subsection{Stability and uniqueness}
We now present the notion of weak solutions and the Kruzkov entropy condition,
and then state the main theorems on the stability and uniqueness of the entropy 
solutions, determined by the Kruzkov entropy condition, 
for the case of non-Lipschitz source terms.

\begin{definition}[Weak Solutions]\label{def:2.1}
A bounded measurable function $u=u(t,x)$ is called a weak solution of 
the Cauchy problem $\eqref{1.1}$--$\eqref{ID}$ in $\Pi_{_T}$ if,
for any test function $\phi=\phi(t,x)\in C^\infty_{c}([0,T) \times \mathbb{R}^n)$, 
\begin{equation}\label{2.1}
\iint_{\Pi_{_T}}
\big\{u\phi_t+f_i(t,x,u)\phi_{x_i}+g(t,x,u)\phi\big\}
\,{\rm d}t{\rm d}x
+\int_{\mathbb{R}^n}u_0(x)\phi(0,x)\,{\rm d}x=0.
\end{equation}
\end{definition}

\begin{definition}[Entropy Solutions]\label{def:2.2} 
A bounded measurable function $u=u(t,x)$ 
is called an 
entropy solution of the Cauchy problem $\eqref{1.1}$--$\eqref{ID}$ in $\Pi_{_T}$  
if it satisfies the Kruzkov entropy condition{\rm :} For any constant $k \in \mathbb{R}$ and any test function $\phi\in C^\infty_{c}(\mathring{\Pi}_{_T})$ with $\phi\geq 0$, 
\begin{align}\label{2.2}
\iint_{\Pi_{_T}}
\big\{&|u-k|\phi_t 
+{\rm sgn}(u-k)(f_i(t,x,u)-f_i(t,x,k))\phi_{x_i}\nonumber\\
&
+{\rm sgn}(u-k)\big({-}f_{ix_i}(t,x,k)+g(t,x,u)\big)\phi\big\}
\,{\rm d}t{\rm d}x\geq 0,
\end{align}
and, in addition, there exists a measure zero set $\mathcal{Z}\subset [0,T]$ such that $u(t,x)$ 
is defined almost everywhere in $\mathbb{R}^n$ for any $t\in [0,T]\setminus\mathcal{Z}$ and 
\begin{align}\label{2.2a}
\lim_{\genfrac{}{}{0pt}{2}{t\rightarrow 0}{t\in[0,T]\setminus\mathcal{Z}}}
\int_{B_R(0)}|u(t,x)-u_0(x)|\,{\rm d}x=0
\end{align}	
for any closed ball $B_R(0):=\{x\in\mathbb{R}^n\,:\, |x|\leq R\}$.
\end{definition}

\begin{remark}
Since the test function $\phi$ in $\eqref{2.2}$ is nonnegative, 
by taking $k=\pm\sup|u(t,x)|$ into $\eqref{2.2}$, 
it is direct to see that the 
entropy solution defined by {\rm Definition} $\ref{def:2.2}$ 
satisfies the integral identity $\eqref{2.1}$ and hence is a weak solution in the sense of {\rm Definition} $\ref{def:2.1}$.    
\end{remark}

\smallskip
The characteristic cone $\mathcal{K}$ is defined by
\begin{equation}\label{5.7}
\mathcal{K}:=\big\{(t,x)\,:\, |x| \leq R-Nt \quad{\rm for}\,\, t\in [0,T_0]\big\},
\end{equation}
where $T_0:=\min\{T,RN^{-1}\}$ with $N:=N_M(R)$ given by
\begin{equation}\label{5.8}
N= N_M(R):=\mathop{\max}\limits_{\genfrac{}{}{0pt}{3}{(t,x)\in [0,T]\times B_R(0)}{|u|\leq M}}
\Big(\sum^n_{i=1} \big(f_{iu}(t,x,u)\big)^2 \Big)^{\frac{1}{2}}.
\end{equation}

Our first two main theorems are on the stability and uniqueness.
\begin{theorem}\label{th:5.1}
Let $f_i(t,x,u)\in C^1([0,T]\times\mathbb{R}^n\times\mathbb{R})$ be the flux functions of $\eqref{1.1}$ 
with continuous derivatives $f_{ix_ju}(t,x,u)$ and $f_{itu}(t,x,u)$, 
and let the source term $g(t,x,u)\in C([0,T]\times\mathbb{R}^n\times\mathbb{R})$ 
satisfy the right-Lipschitz condition in $u$ on any compact set as in $\eqref{gnRL}$.
Assume that $u(t,x)$ and $v(t,x)$ are two 
entropy solutions of 
the Cauchy problem $\eqref{1.1}$--$\eqref{ID}$ in $\Pi_{_T}$ 
with initial data functions $u_0(x)$ and $v_0(x)$, respectively, 
where $|u(t,x)| \leq M$ and $|v(t,x)| \leq M$ 
almost everywhere in the cylinder $[0,T]\times B_R(0)$. 
Then, for almost all $t\in [0,T_0]$,
\begin{equation}\label{5.9}
\int_{S_t}|u(t,x)-v(t,x)|\,{\rm d}x \leq
e^{Lt}\int_{S_0}|u_0(x)-v_0(x)|\,{\rm d}x,
\end{equation}
where $L:=\sup_{(t,x)\in \mathcal{K}}L(t,x;[-M,M])<\infty$, 
and $S_{t}:=B_{R-Nt}(0)$ is the cross-section of the characteristic cone $ \mathcal{K}$ in $\eqref{5.7}$ 
at time $t\in [0,T_0]$.
\end{theorem}

\begin{theorem}\label{th:5.2} 
If $N_M(R)$ in $\eqref{5.8}$ satisfies 
\begin{equation}\label{nmr}
    \lim_{R\rightarrow \infty} R^{-1}N_M(R)=0\qquad {\rm for\ any}\ M>0,
\end{equation}
then, under the assumptions 
in {\rm Theorem $\ref{th:5.1}$}, 
the 
entropy solution of the Cauchy problem $\eqref{1.1}$--$\eqref{ID}$ in $\Pi_{_T}$ is unique.
\end{theorem}

\smallskip
\subsection{Existence}
We now present the theorem on the existence of the Kruzkov entropy solutions.
Firstly, we give the basic assumptions on the smoothness and boundedness of 
the flux functions $f_i(t,x,u)$ and the source term $g(t,x,u)$ and their derivatives: 

\medskip
\noindent
(i) The smoothness of $f_i(t,x,u)$ and $g(t,x,u)$:
\begin{enumerate}
    \item [(a)]  $f_i(t,x,u)\in C^1(\Pi_{_T}\times\mathbb{R}), i=1,2,\cdots, n$, have continuous 
    derivatives $f_{iux_j}(t,x,u)$ and $f_{ix_ix_j}(t,x,u)$; 

 \vspace{2pt}
    \item [(b)]  $g(t,x,u)\in C(\Pi_{_T}\times\mathbb{R})$ is continuously differentiable in $(t,x)\in\Pi_{_T}$ 
    for any $u\in\mathbb{R}$, and satisfies the right-Lipschitz condition in $u$ on any compact set as in $\eqref{gnRL}$.
\end{enumerate}

\noindent
(ii) The boundedness of $f_i(t,x,u)$ and $g(t,x,u)$ with their derivatives:
\begin{enumerate}
    \item [(a)]   $f_{iu}(t,x,u)$, $i=1,2,\cdots,n$, are uniformly bounded for $(t,x,u)\in D_M:= \Pi_{_T}\times [-M,M]$, ${\it i.e.}$,
    \begin{equation}\label{ex0}
\bar{N}:= N_M(\infty)=\mathop{\max}\limits_{(t,x,u)\in D_M}
\Big(\sum^n_{i=1} \big(f_{iu}(t,x,u)\big)^2 \Big)^{\frac{1}{2}}<\infty.
\end{equation}
    \item [(b)]  $f_{ix_i}(t,x,u)$ and $g(t,x,u)$ satisfy that 
    \begin{eqnarray}
&&\sup_{\Pi_{_T}}|f_{ix_i}(t,x,0)-g(t,x,0)|\leq c_0,\qquad 
\sup_{\Pi_{_T}\times \mathbb{R}}\big({-}f_{ix_iu}(t,x,u)\big)\leq c_1,\label{ex1}\\[1mm]
&&\sup_{(t,x)\in\Pi_{_T}}L(t,x;\mathbb{R})=
\sup_{(t,x)\in\Pi_{_T};\, u,v\in\mathbb{R}}
\frac{g(t,x,u)-g(t,x,v)}{u-v}\leq c_2. \label{ex2}
\end{eqnarray}
\end{enumerate}

We now present the main theorem on the existence of the 
entropy solutions.
\begin{theorem}\label{th:ex}
Under the assumptions 
as in ${\rm (i)}$--${\rm (ii)}$ with $M:=(M_0+c_0T)e^{1+|c_1+c_2|T}$ 
for $M_0:=\|u_0(x)\|_{L^{\infty}(\mathbb{R}^n)}$, 
there exists an 
entropy solution of the Cauchy problem $\eqref{1.1}$--$\eqref{ID}$ 
in $\Pi_{_T}$. 
\end{theorem}

\subsection{Riemann solutions with non-selfsimilar structures}
We now introduce {\it Condition} $(\mathcal{H})$ and the main theorem on 
the Riemann solutions with non-selfsimilar structures. 

If the solution $u=u(t,x)$ is a piecewise smooth function in the neighborhood of the points on discontinuity $\mathcal{S}$, 
let $u_l$ and $u_r$ be the traces of $u=u(t,x)$ on the each side of discontinuity $\mathcal{S}$, 
and let $\vec{n}$ be the unit normal vector at point $p:=(t,x)\in\mathcal{S}$ pointing from $u_r$ to $u_l$, {\it i.e.},
\begin{align*}
u_l=\lim_{\varepsilon \rightarrow 0{+}}
u(p+\varepsilon \vec{n}),\qquad 
u_r=\lim_{\varepsilon \rightarrow 0{+}}u(p-\varepsilon \vec{n}).
\end{align*}

From Definitions $\ref{def:2.1}$--$\ref{def:2.2}$, 
we have

\begin{enumerate}
\item[(i)] {\it The Rankine-Hugoniot Condition}:
A piecewise continuous function $u=u(t,x)$ is a weak solution of \eqref{1.1} if and only if
$u$ satisfies the Rankine-Hugoniot condition along any discontinuity $\mathcal{S}${\rm :}
For any point $p=(t,x)\in \mathcal{S}$, 
\begin{equation}\label{2.3}
\vec{n}\cdot ([u],[f_1],[f_2],\cdots,[f_n])=0,
\end{equation}
where $[u]=u_l-u_r$ and $[f_i]=f_i(t,x,u_l)-f_i(t,x,u_r)$ for $i=1,2,\cdots,n$.

\smallskip
\item[(ii)]
{\it Geometric Entropy Conditions}:
A piecewise continuous function $u=u(t,x)$ is a Kruzkov entropy solution 
if and only if $u$ satisfies the geometric entropy conditions along any 
discontinuity $\mathcal{S}$ with the traces $u_l$ and $u_r$, $u_l>u_r$,
at $(t,x)\in \mathcal{S}${\rm :} 
For any constant $k\in [u_r,u_l]$, 
\begin{equation}\label{2.4}
\vec{n}\cdot (k-u_l,f_1(t,x,k)-f_1(t,x,u_l),\cdots,f_n(t,x,k)-f_n(t,x,u_l))\geq 0,
\end{equation}
or equivalently
\begin{equation}\label{2.5}
\vec{n}\cdot (k-u_r,f_1(t,x,k)-f_1(t,x,u_r),\cdots,f_n(t,x,k)-f_n(t,x,u_r))\geq 0.
\end{equation}
\end{enumerate}

\smallskip
To explicitly establish the expressions of Riemann solutions, 
we need to solve its characteristic ODEs.
For this purpose, we particularly consider the Riemann problem of the following $n$-D scalar conservation law 
with a non-Lipschitz source term:
\begin{eqnarray}
&&u_t+\sum\limits^n_{i=1} 
f_i(u)_{x_i}=g(u)\qquad\,\, 
{\rm for} \ (t,x)\in \mathbb{R}^+\times\mathbb{R}^n,\label{eq1.1}\\
&&u|_{t=0}=u_0(x)=
\begin{cases}
u_- \quad \mbox{if}\  M(x)<0,
\\[1mm]
u_+\quad\mbox{if}\  M(x)>0,
\end{cases} \label{eq1.2}
\end{eqnarray}
where $u=u(t,x)\in \mathbb{R}$, $\mathbb{R}^+:=[0,\infty)$, and the flux functions $f_i(u) \in C^2(\mathbb{R})$.
The Riemann initial data 
consist two constant states $u_-$ and $u_+$, 
which are separated by a curved initial discontinuity determined by $M(x)=0$, 
where $M(x)\in C^1(\mathbb{R}^n)$ and $M(x)=0$ is an $(n-1)$-dimensional manifold separating $\mathbb{R}^n$ into two unbounded regions.

For the source term $g(u)$, assume that  
$g(u)\in C(\mathbb{R})$ satisfies the right-Lipschitz condition 
as in $\eqref{gnRL}$ with $h(t,x,u)\equiv g(u)$ and
\begin{align}\label{conhg}
   \mathop{\overline{\lim}}\limits_{\eta \rightarrow \pm\infty}
   \frac{g(\eta)}{\eta}<\infty.
\end{align}

\begin{remark}
Condition \eqref{conhg} on the source term $g(u)$ is used to exclude the {\it ODE}-type blow up so that we can construct the Riemann solution globally in time.
\end{remark}

The characteristic ODEs of equation $\eqref{eq1.1}$ are given by
\begin{equation}\label{3.1}
 \frac{{\rm d}x_i}{{\rm d}t}=f'_i(\bar{u}(t,s)),\qquad x_i(0)=x^0_i,
\end{equation}
for $i=1,\,2,\,\cdots,\,n$, and $\bar{u}=\bar{u}(t,s)$ satisfies
\begin{equation}\label{3.3}
 \frac{{\rm d}\bar{u}}{{\rm d}t}=g(\bar{u}),\qquad \bar{u}(0,s)=s,
\end{equation}
where $x^0:=(x^0_1,x^0_2,\cdots,x^0_n)\in \mathbb{R}^n$ and $s\in\mathbb{R}$.

Denote the point set $A$ by
\begin{align*}
   A:=\{u\in\mathbb{R}\,:\, g(u)=0\}.
\end{align*}
Since $g(u) \in C(\mathbb{R})$, the set $A$ is closed. 
Then the set $\partial A$, as the boundary of $A$, is also closed; 
and $\mathbb{R}\setminus A$ and the set ${\rm int} A$, as the interior of $A$, are both open. 
Therefore, $\mathbb{R}$ can be expressed by the following countable disjoint union:
\begin{equation}\label{3.7}
\mathbb{R}=(\mathbb{R}\setminus A)\cup {\rm int} A \cup \partial A=:
(\cup_iI_i )\cup (\cup_j J_j) \cup \partial A,
\end{equation}
where $g(a_i)=0$ if $a_i$ is finite, and $g(b_i)=0$ if $b_i$ is finite; 
and the open intervals $I_i:=(a_i,b_i)$ and $J_j:=(c_j,d_j)$ satisfy 
\begin{equation}\label{3.7a}
g(u) \neq 0 \quad {\rm for}\ u\in (a_i,b_i),\qquad\,\, g(u) = 0 \quad {\rm for}\ u\in (c_j,d_j).
\end{equation}

Now, we give the lemma on the well-posedness of the characteristic ODE $\eqref{3.3}$.
\begin{lemma}\label{lem:3.1}
Assume that the source term of equation $\eqref{eq1.1}$ satisfies $g(u)\in C(\mathbb{R})$, the right-Lipschitz condition 
in $\eqref{gnRL}$ with $h(t,x,u)\equiv g(u)$, and $\eqref{conhg}$ holds.
Then there exists a unique continuous function $\bar{u}=\bar{u}(t,s)$ on $\mathbb{R}^+ \times \mathbb{R}$
that satisfies the following properties{\rm :}

\begin{itemize}
\item[(i)] For any fixed $s \in \mathbb{R}$, $\bar{u}=\bar{u}(\cdot,s)$ is a global solution of {\rm ODE} $\eqref{3.3}$ on $\mathbb{R}^+${\rm ;}
\item[(ii)] For any fixed $t \in \mathbb{R}^+$, $\bar{u}(t,\cdot)$ is a non-decreasing Lipschitz continuous function{\rm ;} 
\item[(iii)] $\bar{u}(t,s)$ is $C^1$ 
on $\mathbb{R}^+\times(\mathbb{R}\setminus(\partial A{-}\mathring {\partial}A))$ and satisfies
\begin{subnumcases}{\label{usts}\bar{u}_s(t,s)=}
\frac{g(\bar{u}(t,s))}{g(s)} &\quad {\rm if} $s \in \mathbb{R}\setminus A,\label{3.8}$\\
1 &\quad {\rm if} $ s \in {\rm int} A,\label{3.9}$\\
e^{g'(s)t} &\quad {\rm if} $s \in \mathring {\partial}A \  {\rm or}\ g'(s)={-}\infty,\label{3.10}$
\end{subnumcases} 
where
$\mathring {\partial}A:=\big \{s\in \partial A\, :\, g'(s)\ {\rm exists} \big \}$.
\end{itemize}
\end{lemma}

For any $s\in\mathbb{R}$, by taking $\bar{u}=\bar{u}(t,s)$ into $\eqref{3.1}$ 
and integrating it from $0$ to $t$, 
we obtain the unique characteristic $x(t)=(x_1(t),x_2(t),\cdots,x_n(t))$ 
emitting from $x^0\in \mathbb{R}^n$:
\begin{equation}\label{4.1}
x_i(t)=x^0_i+\int^t_0f'_i(\bar{u}(\tau,s))\, {\rm d}\tau=:x_i^0+\chi_i(t,s)
\qquad\,\, {\rm for}\,\,t\in \mathbb{R}^+.
\end{equation}

Given $u_-,u_+\in \mathbb{R}$, we denote
\begin{equation}\label{4.3}
[\chi_i](t):=\int^t_0 \frac {[f_i]}{[u]}_{\pm}(\tau)\, {\rm d}\tau
\qquad\,\, {\rm for}\ t \in \mathbb{R}^+,
\end{equation}
where $\frac {[f_i]}{[u]}_{\pm}(t)$ is defined by
\begin{equation}\label{4.4}
\frac {[f_i]}{[u]}_{\pm}(t):=
\begin{cases}
\dfrac{f_i(\bar{u}(t,u_-))-f_i(\bar{u}(t,u_+))}
{\bar{u}(t,u_-)-\bar{u}(t,u_+)}\quad &{\rm if}\ \bar{u}(t,u_-) \neq \bar{u}(t,u_+) ,\\[2mm]
f'_i(\bar{u}(t,u_+))\quad &{\rm if}\ \bar{u}(t,u_-) = \bar{u}(t,u_+).
\end{cases}
\end{equation}

We now present the multi-dimensional version of the convex condition 
on the flux functions and the initial discontinuity, namely,
{\it Condition} $(\mathcal{H})$, 
which is similar to the convexity of the flux function in the one-dimensional case.

\begin{definition}[Condition $(\mathcal{H})$]
The Riemann problem $\eqref{eq1.1}$--$\eqref{eq1.2}$ satisfies 
{\it Condition} $(\mathcal{H})$ 
if, for any $x\in\mathbb{R}^n$ with $M(x)=0$, 
\begin{equation}\label{convexh}
H(x,u):=M_{x_i}(x)f''_i(u)>0 \qquad\,\,\, {\rm for}\ u\in (a,b),
\end{equation}
where the open interval $(a,b)$ is determined by
\begin{equation}\label{2.7}
a \leq \mathop{\inf}\limits_{t\geq 0}\left\{\bar{u}(t,u_+), \bar{u}(t,u_-)\right\},
\qquad b \geq \mathop{\sup}\limits_{t\geq 0}\left\{\bar{u}(t,u_+), \bar{u}(t,u_-)\right\}.
\end{equation}
\end{definition}

Now we can state the main theorem on the global Riemann solutions with non-selfsimilar structures 
for the Riemann problem $\eqref{eq1.1}$--$\eqref{eq1.2}$ with non-Lipschitz source terms.

\begin{theorem}\label{th:4.4}
Given $u_-,u_+\in \mathbb{R}$. 
Suppose that the source term $g(u)\in C(\mathbb{R})$ satisfies 
the right-Lipschitz condition as in $\eqref{gnRL}$ with $h(t,x,u)\equiv g(u)$ 
and condition $\eqref{conhg}$.
Suppose that the Riemann problem $\eqref{eq1.1}$--$\eqref{eq1.2}$ satisfies 
{\it Condition} $(\mathcal{H})$. 
Then
\begin{itemize}
\item[(i)] If $u_->u_+$, the solution of the Riemann problem $\eqref{eq1.1}$--$\eqref{eq1.2}$ 
is a shock wave with non-selfsimilar structure{\rm :}
\begin{equation}\label{4.13}
\qquad u(t,x)=
\begin{cases}
 \bar{u}(t,u_-) &{\rm if}\ M(x_1-[\chi_1](t),x_2-[\chi_2](t),\cdots,x_n-[\chi_n](t))<0,
\\[1mm]
 \bar{u}(t,u_+) &{\rm if}\ M(x_1-[\chi_1](t),x_2-[\chi_2](t),\cdots,x_n-[\chi_n](t))>0,
\end{cases}
\end{equation}
and the shock surface $S(t,x)=0$ is given by
\begin{equation}\label{4.14}
M(x_1-[\chi_1](t),x_2-[\chi_2](t),\cdots,x_n-[\chi_n](t))=0,
\end{equation}
where $[\chi_i](t)$ is defined by $\eqref{4.3}${\rm ;}
 
\item[(ii)] If $u_-<u_+$, the solution of the Riemann problem $\eqref{eq1.1}$--$\eqref{eq1.2}$ 
is a rarefaction wave with non-selfsimilar structure{\rm :}
\begin{equation}\label{4.17}
\qquad u(t,x)=
\begin{cases}
\bar{u}(t,u_-) &{\rm if}\ M(x_1-\chi_1(t,u_-),\cdots,x_n-\chi_n(t,u_-))<0,\,\, t\geq 0;\\
\bar{u}(t,c(t,x)) &{\rm if}\ M(x_1-\chi_1(t,u_-),\cdots,x_n-\chi_n(t,u_-))\geq 0,\\
&\quad M(x_1-\chi_1(t,u_+),\cdots,x_n-\chi_n(t,u_+))\leq 0,\,\, t>0;\\
\bar{u}(t,u_+) &{\rm if}\ M(x_1-\chi_1(t,u_+),\cdots,x_n-\chi_n(t,u_+))>0,\,\, t\geq 0,
\end{cases}
\end{equation}
where $c(t,x)$ is a function implicitly but uniquely determined by 
\begin{equation}\label{4.18}
F(t,x,c):=M(x_1-\chi_1(t,c),x_2-\chi_2(t,c),\cdots,x_n-\chi_n(t,c))=0.
\end{equation}
\end{itemize}
\end{theorem}

\begin{remark}
If $H(x,u)<0$ on $M(x)=0$ for any $u\in(a,b)$, 
by letting $\tilde{M}(x)=-M(x)$, 
then for any $x\in\mathbb{R}^n$ with $\tilde{M}(x)=0$, 
\begin{equation*}
\tilde{H}(x,u):=\tilde{M}_{x_i}(x)f''_i(u)>0 \qquad\,\,\, {\rm for}\ u\in (a,b),
\end{equation*}
where the open interval $(a,b)$ is given by $\eqref{2.7}$.
Thus, {\rm Theorem} $\ref{th:4.4}$ holds for the Riemann problem of equation $\eqref{eq1.1}$ 
with the initial data given by
\begin{align*}
u_0(x)=\begin{cases}
    u_+\quad &{\rm if}\  \tilde{M}(x)<0,\\[1mm]
u_-\quad &{\rm if}\  \tilde{M}(x)>0.
\end{cases}
\end{align*}
\end{remark}

\medskip
\section{Uniqueness and Stability of 
Entropy Solutions: Proof of Theorems $\ref{th:5.1}$--$\ref{th:5.2}$  }
In this section, we prove Theorems $\ref{th:5.1}$--$\ref{th:5.2}$ for the uniqueness and stability of 
entropy solutions. 

Let  $\delta(\sigma)\in C^{\infty}(\mathbb{R})$ satisfy 
\begin{equation}\label{5.1}
\delta(\sigma)\geq 0,\qquad\delta(\sigma)\equiv 0\quad{\rm for}\,\, |\sigma| \geq 1,\qquad\int^{\infty}_{{-}\infty} \delta(\sigma)\, {\rm d}\sigma =1.
\end{equation}
For any $h>0$, we let
\begin{equation}\label{5.2}
\delta_h(\sigma)=h^{-1}\delta(h^{-1}\sigma).
\end{equation}
Then $\delta_h(\sigma)\in C^{\infty}(\mathbb{R})$ satisfies
\begin{equation}\label{5.3}
\delta_h(\sigma)\geq 0,
\qquad\delta_h(\sigma)\equiv 0\quad{\rm for}\,\, |\sigma| \geq h, 
\qquad|\delta_h(\sigma)|\lesssim h^{-1},
\qquad \int^{\infty}_{{-}\infty} \delta_h(\sigma)\, {\rm d}\sigma =1.
\end{equation}

Denote $\lambda_h:=\lambda_h(t,x,\tau,y)$ by
\begin{equation}\label{5.4}
\lambda_h=\lambda_h(t,x,\tau,y):=\delta_h(\frac{t-\tau}{2})\, \prod^n_{i=1}\delta_h(\frac{x_i-y_i}{2}).
\end{equation}
For any fixed $(t,x)$ with $t<T$, 
if $h<\min \big\{\frac{t}{2},\frac{T-t}{2}\big\}$, then
\begin{equation}\label{5.5}
\iint_{\Pi_{_T}}\frac{1}{2^{n+1}}\lambda_h(t,x,\tau,y)\, {\rm d}\tau {\rm d}y=\int_{-h}^h\int_{\mathbb{R}^n}\delta_h(\beta)\,
\textstyle\prod^n_{i=1}\delta_h(\xi_i)\, {\rm d}\beta {\rm d}\xi=1,
\end{equation}
where 
$\beta=\frac{t-\tau}{2}$ and $\xi=(\xi_1,\cdots,\xi_n)$ with $\xi_i=\frac{x_i-y_i}{2}$.

\medskip
Now we are ready to prove Theorem \ref{th:5.1}.
\begin{proof}[Proof of {\rm Theorem} $\ref{th:5.1}$]
The proof is divided into four steps.

\smallskip
\noindent
{\bf 1.} Let $w\in C_{c}^{\infty}(\Pi_{_T} \times \Pi_{_T})$ with $w=w(t,x,\tau,y)\geq 0$. 
For each fixed $(\tau,y)\in \Pi_{_T}$, 
by taking $k=v(\tau,y)$ and $\phi(t,x)=w(t,x,\tau,y)$ in $\eqref{2.2}$, 
and integrating it with respect to $(\tau,y) $ over $\Pi_{_T}$, we obtain
\begin{align}\label{5.10}
\iiiint_{\Pi_{_T} \times \Pi_{_T}} {\rm sgn}\big(u(t,x)-v(\tau,y)\big)
\Big\{&\big(u(t,x)-v(\tau,y)\big)w_t \nonumber\\
&+\big(f_i(t,x,u(t,x))-f_i(t,x,v(\tau,y))\big)w_{x_i}\nonumber\\[1mm]
&
-f_{ix_i}(t,x,v(\tau,y))w+g(t,x,u(t,x))w\Big\}\,{\rm d}t{\rm d}x{\rm d}\tau{\rm d}y\ge 0.
\end{align}	
Similarly, for each fixed $(t,x)\in \Pi_{_T}$, 
by taking $k=u(t,x)$ and $\phi(\tau,y)=w(t,x,\tau,y)$ in $\eqref{2.2}$, 
and integrating it with respect to $(t,x)$ over $\Pi_{_T}$, we obtain
\begin{align}\label{5.11}
\iiiint_{\Pi_{_T} \times \Pi_{_T}} 
{\rm sgn}\big(v(\tau,y)-u(t,x)\big)
 \Big\{&\big(v(\tau,y)-u(t,x)\big)w_{\tau} \nonumber\\
&+\big(f_i(\tau,y,v(\tau,y))-f_i(\tau,y,u(t,x))\big)w_{y_i}\nonumber\\[1mm] 
&-f_{iy_i}(\tau,y,u(t,x))w+g(\tau,y,v(\tau,y))w\Big\}\,{\rm d}\tau{\rm d}y {\rm d}t{\rm d}x\ge 0.
\end{align}
Adding $\eqref{5.10}$ and $\eqref{5.11}$ together, we conclude
\begin{align}\label{5.12}
0\leq\iiiint_{\Pi_{_T} \times \Pi_{_T}} {\rm sgn}&\big(u(t,x)-v(\tau,y)\big)\nonumber\\
\times \Big\{&\big(u(t,x)-v(\tau,y)\big)(w_t+w_{\tau}) \nonumber\\
&+\big(f_i(t,x,u(t,x))-f_i(\tau,y,v(\tau,y))\big)(w_{x_i}+w_{y_i}) \nonumber\\
&+\big(\big(f_i(\tau,y,v(\tau,y))-f_i(t,x,v(\tau,y))\big)w_{x_i}-f_{ix_i}(t,x,v(\tau,y))w\nonumber\\
&\quad\,\,\,\, +\big(f_i(\tau,y,u(t,x))-f_i(t,x,u(t,x))\big)w_{y_i}+f_{iy_i}(\tau,y,u(t,x))w\big)\nonumber\\
&+\big(g(t,x,u(t,x))-g(\tau,y,v(\tau,y))\big)w\Big\}\,{\rm d}\tau{\rm d}y{\rm d}t{\rm d}x\nonumber\\
=:\iiiint_{\Pi_{_T} \times \Pi_{_T}}(I_1&+I_2+I_3+I_4)\,
{\rm d}\tau{\rm d}y{\rm d}t{\rm d}x.
\end{align}

\smallskip
\noindent
{\bf 2.} Let $\phi \in C_{\rm c}^{\infty}(\Pi_{_T} )$ with $\phi=\phi(t,x)\geq 0$ and 
\begin{align*}
\phi(t,x)\equiv 0 \qquad {\rm if}\ (t,x)\notin[\rho,T-2\rho]\times B_{r-2\rho}(0),
\end{align*}
where $2\rho \leq \min\{T,r\}$.
For $h\in(0,\rho]$, choose the test function $w$ in $\eqref{5.12}$ given by 
\begin{equation}\label{5.13}
 w^h:=\phi(\frac{t+\tau}{2},\frac{x+y}{2})\,\delta_h(\frac{t-\tau}{2}) 
 \prod^n_{i=1}\delta_h(\frac{x_i-y_i}{2})
 =\phi(\frac{t+\tau}{2},\frac{x+y}{2})\,\lambda_h(t,x,\tau,y),
\end{equation}
where $\delta_h$ and $\lambda_h$ are given by $\eqref{5.2}$ and $\eqref{5.4}$, respectively. 
Then it is direct to see that
\begin{equation*}
 w^h_t+w^h_{\tau}=\phi_t(\frac{t+\tau}{2},\frac{x+y}{2})\, \lambda_h, \qquad 
 w^h_{x_i}+w^h_{y_i}=\phi_{x_i}(\frac{t+\tau}{2},\frac{x+y}{2})\, \lambda_h.
\end{equation*}

Let $G_k:=G_k(t,x,\tau,y,u(t,x),v(\tau,y))$, $k=1,2,4$, be determined by
\begin{equation}\label{gn1}
I_k=:G_k(t,x,\tau,y,u(t,x),v(\tau,y))\lambda_h(t,x,\tau,y)=G_k\lambda_h.
\end{equation}
Since $f_i(t,x,u)\in C^1([0,T]\times\mathbb{R}^n\times\mathbb{R})$, 
by the choice of $w^h$ in $\eqref{5.13}$, it follows from $\eqref{5.12}$ and $\eqref{gn1}$ that 
$G_1$ and $G_2$ are both Lipschitz in all the variables. 
Furthermore, since $f_{it}(t,x,u)$ and $f_{ix_j}(t,x,u)$ satisfy the Lipschitz condition in $u$ 
on $[-M,M]$ as in $\eqref{gnL}$, 
and $I_3$ is independent of the source term $g(t,x,u)$, 
by the same arguments in \S 3 of Kruzkov \cite{[Ksn2]}, we obtain
{\small 
\begin{align}\label{gn2}
&\lim_{h\rightarrow 0}\!\iiiint_{\Pi_{_T} \times \Pi_{_T}}\!\!\!(I_1+I_2+I_3)\,{\rm d}\tau{\rm d}y{\rm d}t{\rm d}x\nonumber\\
&=2^{n+1}\!\!\iint_{\Pi_{_T}}\!\!
{\rm sgn}\big(u(t,x){-}v(t,x)\big)\Big\{\big(u(t,x){-}v(t,x)\big)\phi_t
+\big(f_i(t,x,u(t,x)){-}f_i(t,x,v(t,x))\big)\phi_{x_i}\Big\}
\,{\rm d}t{\rm d}x.
\end{align}
}
For $I_4$, it follows from $\eqref{5.12}$ and $\eqref{gn1}$ that
\begin{align}\label{gn3}
G_4&={\rm sgn}\big(u(t,x)-v(\tau,y)\big)\big(g(t,x,u(t,x))-g(\tau,y,v(\tau,y))\big)\,\phi(\frac{t+\tau}{2},\frac{x+y}{2})\nonumber\\
&={\rm sgn}\big(u(t,x)-v(\tau,y)\big)\big(g(t,y,v(\tau,y))-g(\tau,y,v(\tau,y))\big)\,\phi(\frac{t+\tau}{2},\frac{x+y}{2})\nonumber\\
&\quad\, +{\rm sgn}\big(u(t,x)-v(\tau,y)\big)\big(g(t,x,v(\tau,y))-g(t,y,v(\tau,y))\big)\,\phi(\frac{t+\tau}{2},\frac{x+y}{2})\nonumber\\
&\quad\, +{\rm sgn}\big(u(t,x)-v(\tau,y)\big)\big(g(t,x,u(t,x))-g(t,x,v(\tau,y))\big)\,\phi(\frac{t+\tau}{2},\frac{x+y}{2})\nonumber\\
&=:G_{4,1}+G_{4,2}+G_{4,3}.
\end{align}
Since the solutions $u(t,x)$ and $v(t,x)$ are both bounded by $M>0$, 
and the source term $g(t,x,u)$ 
is continuous in $(t,x)\in[0,T]\times\mathbb{R}^n$, 
we have 
\begin{align}\label{gn4}
\lim_{h\rightarrow 0}\iiiint_{\Pi_{_T} \times \Pi_{_T}}(G_{4,1}+G_{4.2})\lambda_h\,{\rm d}\tau{\rm d}y{\rm d}t{\rm d}x=0.
\end{align}
For $G_{4,3}$, since $\phi\geq 0$ and $\lambda_h\geq 0$, it follows from 
$\eqref{gnRL}$, $\eqref{5.5}$, and $\eqref{gn3}$--$\eqref{gn4}$ that
\begin{align}\label{gn5}
\lim_{h\rightarrow 0}\iiiint_{\Pi_{_T} \times \Pi_{_T}}I_4\,{\rm d}\tau{\rm d}y{\rm d}t{\rm d}x
&=\lim_{h\rightarrow 0}\iiiint_{\Pi_{_T} \times \Pi_{_T}}G_{4,3}\,\lambda_h\,{\rm d}\tau{\rm d}y{\rm d}t{\rm d}x\nonumber\\
&\leq\lim_{h\rightarrow 0}\iiiint_{\Pi_{_T} \times \Pi_{_T}}
L|u(t,x)-v(\tau,y)|\,\phi\,\lambda_h\,{\rm d}\tau{\rm d}y{\rm d}t{\rm d}x\nonumber\\
&=2^{n+1}\iint_{\Pi_{_T}}L|u(t,x)-v(t,x)|\,\phi\,{\rm d}t{\rm d}x.
\end{align}

By combining $\eqref{gn2}$ with $\eqref{gn5}$, it follows from $\eqref{5.12}$ that
\begin{align}\label{5.18}
\iint_{\Pi_{_T}} {\rm sgn}\big(u(t,x)-v(t,x)\big)
\big\{&\big(u(t,x)-v(t,x)\big)\phi_t 
+\big(f_i(t,x,u(t,x))-f_i(t,x,v(t,x))\big)\phi_{x_i} 
\nonumber\\[-0.5mm]
&+L\big(u(t,x)-v(t,x)\big)\phi\big\}\, {\rm d}t{\rm d}x \geq 0.
\end{align}

\smallskip
\noindent
{\bf 3.} 
We now employ inequality $\eqref{5.18}$ to establish the following inequality:
\begin{equation}\label{5.20}
\mu(\tau)\leq \mu(\rho)+L\int_{\rho}^{\tau}\mu(t)\, {\rm d}t \qquad{\rm a.e.}\,\,\, 0<\rho<\tau<T_0,
\end{equation}
where $T_0=\min\{T,RN^{-1}\}$ as in $\eqref{5.7}$, and $\mu(t)$ is defined by
\begin{equation}\label{5.19}
\mu(t):=\int_{S_t}|u(t,x)-v(t,x)|\, {\rm d}x\qquad {\rm for}\,\, t\geq 0.
\end{equation}

Denote $\alpha_h(\sigma):=\int_{{-}\infty}^{\sigma}\delta_h(\beta)\,{\rm d}\beta$ so that
$\alpha'_h(\sigma)=\delta_h(\sigma)\geq 0$. 	
Choose $\rho,\tau \in (0,T_0)$ with $\rho<\tau$
and $\phi=\phi(t,x)$ in $\eqref{5.18}$ given by
$$\phi(t,x)=\big(\alpha_h(t-\rho)-\alpha_h(t-\tau)\big)
\,\xi^{\varepsilon}(t,x)\qquad\, {\rm for}\,\, h<\min\{\rho,T_0-\tau\},$$
where $\xi^{\varepsilon}:=\xi^{\varepsilon}(t,x)$ is given by
$$\xi^{\varepsilon}(t,x):= 1-\alpha_{\varepsilon}(|x|+Nt-R+\varepsilon)\qquad {\rm for}\,\,\varepsilon>0.$$
Then it is direct to see that $\phi(t,x) \in C_{\rm c}^{\infty}(\Pi_{_{T_0}})$ with $\phi(t,x)\geq 0$. 
Notice that, for any $\varepsilon>0$, 
$\xi^{\varepsilon}(t,x)$ $\equiv 0$ outside  cone $\mathcal{K}$, 
and 
$$ 0= \xi^{\varepsilon}_t+N|\nabla_x\xi^{\varepsilon}|\geq \xi^{\varepsilon}_t+
\frac{f_i(t,x,u)-f_i(t,x,v)}{u-v}\,\xi^{\varepsilon}_{x_i}
\qquad {\rm for\ any}\ (t,x)\in \mathcal{K}.$$
Thus, it follows from $\eqref{5.18}$ that
\begin{align*}
\iint_{\Pi_{_{T_0}}} \big\{&\big(\delta_h(t-\rho)-\delta_h(t-\tau)\big)\xi^{\varepsilon}(t,x)|u(t,x)-v(t,x)|\nonumber\\
&+L\big(\alpha_h(t-\rho)-\alpha_h(t-\tau)\big)\xi^{\varepsilon}(t,x)|u(t,x)-v(t,x)|\big\}
\, {\rm d}t{\rm d}x \geq 0,
\end{align*}	
which, by letting $\varepsilon\rightarrow 0$ and using $\eqref{5.19}$, implies
\begin{equation}\label{5.23}
 \int_0^{T_0} \big\{\big(\delta_h(t-\rho)-\delta_h(t-\tau)\big )\mu(t)
 +L\big(\alpha_h(t-\rho)-\alpha_h(t-\tau)\big)\mu(t)\big\}\, {\rm d}t \geq 0.
\end{equation}

From the properties of $\delta_h(\sigma)$ in $\eqref{5.3}$, 
for $h\leq \min\{\rho,T_0-\tau\}$, 
\begin{align*}
\Big|\int_0^{T_0} \delta_h(t-\rho)\mu(t)\, {\rm d}t-\mu(\rho)\Big|
&=\Big|\int_0^{T_0}\delta_h(t-\rho)(\mu(t)-\mu(\rho))\, {\rm d}t\Big|\\
&\lesssim h^{-1}\int_{\rho-h}^{\rho+h}|\mu(t)-\mu(\rho)|\, {\rm d}t,
\end{align*}
which, by letting $h\rightarrow 0$, implies
\begin{equation*}
\lim_{h\rightarrow 0}\int_0^{T_0} \delta_h(t-\rho)\mu(t)\, {\rm d}t=\mu(\rho)\qquad{\rm a.e.}\,\,0<\rho<\tau<T_0.
\end{equation*}
Similarly, we have
\begin{equation*}
\lim_{h\rightarrow 0}\int_0^{T_0} \delta_h(t-\tau)\mu(t)\, {\rm d}t=\mu(\tau)\qquad\, {\rm a.e.}\,\,\,0<\rho<\tau<T_0.
\end{equation*}
Then we obtain
\begin{align*}
\int_0^{T_0} L\big(\alpha_h(t-\rho)-\alpha_h(t-\tau)\big)\mu(t)\, {\rm d}t
&=\int_0^{T_0} L\mu(t)\int_{t-\tau}^{t-\rho}\delta_h(\sigma)\, {\rm d}\sigma {\rm d}t\nonumber\\
&=\int_0^{T_0} L\mu(t)\int_{\rho}^{\tau}\delta_h(t-\beta)\, {\rm d}\beta {\rm d}t\nonumber\\
&= L\int_{\rho}^{\tau}\Big(\int_0^{T_0} \mu(t)\delta_h(t-\beta)\, {\rm d}t\Big){\rm d}\beta\nonumber \\
&\longrightarrow L\int_{\rho}^{\tau}\mu(\beta)\,{\rm d}\beta\qquad{\rm as}\,\, h\rightarrow 0.
\end{align*}
Therefore, 
by letting $h\rightarrow 0$ in $\eqref{5.23}$, 
we conclude $\eqref{5.20}$.

\medskip
\noindent
{\bf 4.} By letting $\rho\rightarrow 0$ in $\eqref{5.20}$, 
\begin{equation}\label{5.24x}
\mu(\tau)\leq \mu(0)+L\int_{0}^{\tau}\mu(t)\, {\rm d}t \qquad{\rm a.e.}\,\,\,0<\tau<T_0.
\end{equation}
Denote $\eta(\tau):=\int_{0}^{\tau}\mu(t)\, {\rm d}t$. It follows from $\eqref{5.24x}$ that 
\begin{equation*}
\eta'(\tau)=\mu(\tau)\leq L\eta(\tau)+\mu(0)\qquad {\rm a.e.}\,\,\,0<\tau<T_0,
\end{equation*}
which implies
$$\big (\eta(\tau) e^{-L\tau}\big )'=\big(\eta'(\tau)-L\eta(\tau)\big) 
e^{-L\tau}\leq \mu(0)e^{-L\tau} \qquad {\rm a.e.}\,\,\,0<\tau<T_0. $$
This means that, for each $t\in(0,T_0)$, 
\begin{equation}\label{5.26x}
\eta(t) e^{-Lt}\leq \eta(0)+\int_0^t\mu(0)e^{-L\tau}\, {\rm d}\tau.
\end{equation}
Since $\eta(0)=0$, it follows from $\eqref{5.24x}$--$\eqref{5.26x}$ that, 
for {\it a.e.} $t\in (0,T_0)$,
\begin{align*}
\mu(t)\leq \mu(0)+L\eta(t)\leq \mu(0)+Le^{Lt}\int_0^t\mu(0)e^{-L\tau}\, {\rm d}\tau 
=\mu(0)e^{Lt}.
\end{align*}
By the definition of $\mu(t)$ in $\eqref{5.19}$, we conclude $\eqref{5.9}$.
This completes the proof.
\end{proof}

\begin{proof}[Proof of {\rm Theorem} $\ref{th:5.2}$]
From $\eqref{nmr}$, for any $(t,x)\in \Pi_{_{T}}$,
there exists a characteristic cone containing point $(t,x)$. 
Then Theorem $\ref{th:5.2}$ follows directly from Theorem $\ref{th:5.1}$.
\end{proof}

\section{Existence of 
Entropy Solutions: Proof of Theorem $\ref{th:ex}$}

In Kruzkov \cite{[Ksn2]}, the source term was assumed to be Lipschitz in $u$, and the existence was proved by 
the vanishing viscosity method, in which the key point was to establish the {\it a priori} estimate of the $L^1$--continuity modulus 
of the solutions $u^\varepsilon$ of the second-order quasilinear parabolic equations. 
For Theorem $\ref{th:ex}$,
the regularity of the continuous source term $g(t,x,u)$ with respect to $u$ 
is assumed to be only right-Lipschitz (not necessarily left-Lipschitz). 
The existence is proved based on the observation that the {\it a priori} estimate of the $L^1$--continuity modulus depends only on 
the supremum of the first-order difference of the source term $g(t,x,u)$ with respect to $u$, 
as well as some derivatives of the flux functions $f_i(t,x,u)$ and the first-order derivatives of the source term $g(t,x,u)$ 
with respect to $(t,x)$. 

\begin{proof}[Proof of {\rm Theorem} $\ref{th:ex}$]
We now construct the 
entropy solutions of the Cauchy problem 
$\eqref{1.1}$--$\eqref{ID}$ in the sense of Kruzkov, by the vanishing viscosity method. 
The proof is divided into four steps. 

\smallskip
\noindent
{\bf 1.} We first consider the Cauchy problem of second-order quasilinear parabolic equations:
\begin{eqnarray}
&&u^{\Lambda}_t+f_i^{m,l}(t,x,u^{\Lambda})_{x_i}
-g^{m,l}(t,x,u^{\Lambda})=\varepsilon\Delta u^{\Lambda},
\label{pe}\\[1mm]
&&u^{\Lambda}|_{t=0}=u^h_0(x), \label{peID}
\end{eqnarray}
where the
parameters $\Lambda:=(h,l,\varepsilon,m)$ satisfy $h,l,\varepsilon\in (0,1]$ and $m\in \mathbb{N}$.
The functions $f_i^{m,l}(t,x,u)$ and $g^{m,l}(t,x,u)$ are defined by
\begin{align}\label{fgml}
    f_i^{m,l}(t,x,u):=\eta^m(|x|)f_i^l(t,x,u), \qquad 
    g^{m,l}(t,x,u):=\eta^m(|x|)g^l(t,x,u),
\end{align}
where $\eta^m(\sigma):=1-\int^\sigma_{-\infty}\delta(\varsigma-m)\,{\rm d}\varsigma$, and the functions $f_i^l(t,x,u)$ and $g^l(t,x,u)$ are given by
\begin{eqnarray}
&& f_i^l(t,x,u):=\iiint_{\Pi_{_{T}}\times \mathbb{R}}
\delta_l(t-\tau)\, \textstyle\prod^n_{i=1}\delta_l(x_i-y_i)\delta_l(u-v)f_i(\tau,y,v)\, 
{\rm d}\tau {\rm d}y {\rm d}v,
\label{fgl1}\\[1mm]
&&g^l(t,x,u):=\iiint_{\Pi_{_{T}}\times \mathbb{R}}
\delta_l(t-\tau)\, \textstyle\prod^n_{i=1}\delta_l(x_i-y_i)\delta_l(u-v)g(\tau,y,v)\, 
{\rm d}\tau {\rm d}y {\rm d}v.\label{fgl2}
\end{eqnarray}
The function $u^h_0(x)$ with $h\in(0,1]$ is
the modifying function of $u_0(x)\in L^\infty(\mathbb{R}^n)$, {\it i.e.}, 
\begin{align}\label{peID0}
u_0^h(x):=\int_{\mathbb{R}^n} \textstyle\prod^n_{i=1}\delta_h(x_i-y_i)u_0(y)\, {\rm d}y.
\end{align}

It follows from $\eqref{fgml}$--$\eqref{fgl2}$ that the bounds of the functions $f_i^l(t,x,u)$ and $g^l(t,x,u)$ with their derivatives 
are inherited from the bounds of the functions $f_i(t,x,u)$ and $g(t,x,u)$ with their derivatives, 
and hence so are the bounds of the functions $f_i^{m,l}(t,x,u)$ and $g^{m,l}(t,x,u)$ with their derivatives. 
Specially, from $\eqref{ex2}$, $\eqref{5.1}$, and $\eqref{fgml}$--$\eqref{fgl2}$, 
the first-order difference of $g^{m,l}$ with respect to $u$ satisfies
\begin{align*}
g^{m,l}[u,v]&=\frac{g^{m,l}(t,x,u)-g^{m,l}(t,x,v)}{u-v}\nonumber\\
&=\eta^m(|x|)\!\!\iiint_{\Pi_{_{T}}\times \mathbb{R}}
\delta_l(t{-}\tau)\textstyle\prod^n_{i=1}\delta_l(x_i{-}y_i)
\delta(w)\dfrac{g(\tau,y,u{-}lw)-g(\tau,y,v{-}lw)}{u-v}
\,{\rm d}\tau{\rm d}y{\rm d}w\nonumber\\
&\leq\iiint_{\Pi_{_{T}}\times \mathbb{R}}\!\!
\delta_l(t{-}\tau)\textstyle\prod^n_{i=1}\delta_l(x_i{-}y_i)
\displaystyle\delta(w)L(\tau,y;\mathbb{R})
\,{\rm d}\tau{\rm d}y{\rm d}w
\leq\sup_{(\tau,y)\in\Pi_{_{T}}}L(\tau,y;\mathbb{R})\leq c_2,
\end{align*}
which implies that, for any $l\in(0,1]$ and $m\in \mathbb{N}$,
\begin{align}\label{gmlu}
g^{m,l}_u(t,x,u)\leq c_2\qquad {\rm for\ any}\ (t,x,u)\in \Pi_{_{T}}\times \mathbb{R}.
\end{align}

Since $\|u_0(x)\|_{L^{\infty}(\mathbb{R}^n)}=M_0$, the function $u^h_0(x)$ with $h\in(0,1]$ is bounded by $M_0$, 
and at least possesses bounded derivatives up to the third order. 
Therefore, by the theory of second-order parabolic equations ({\it cf.} \cite{[OK],[FA]}), 
the Cauchy problem $\eqref{pe}$--$\eqref{peID}$ admits a unique classical solution $u^\Lambda(t,x)$, 
which is bounded in $\Pi_{_{T}}$, 
and possesses bounded and uniformly H\"{o}lder continuous derivatives in equation $\eqref{pe}$.

Let $\omega_R(\sigma)$ be the $L^1$--continuity modulus of $u_0(x)\in L^\infty(\mathbb{R}^n)$: 
\begin{align}\label{JRu00}
J_R(u_0(x),\Delta x):=\int_{B_R(0)}|u_0(x+\Delta x)-u_0(x)|\,{\rm d}x\leq \omega_R(|\Delta x|)
\qquad\,\, \mbox{for any $R>0$}.
\end{align}
Then it follows from $\eqref{peID0}$ and $\eqref{JRu00}$ that, 
for any $R\geq 1$,
\begin{align}\label{uhu0}
  \int_{B_R(0)}\big|u^h_0(x)-u_0(x)\big|\,{\rm d}x
&\leq\int_{\mathbb{R}^n}\textstyle\prod^n_{i=1}\delta(z_i)
\displaystyle\int_{B_R(0)}|u_0(x-hz)-u_0(x)|
\,{\rm d}x{\rm d}z\nonumber\\
&\leq\int_{\mathbb{R}^n}\textstyle\prod^n_{i=1}\delta(z_i)\,
\omega_R(h|z|)\,{\rm d}z\nonumber\\
&\leq \omega_R(\sqrt{n}h).
\end{align}

\smallskip
\noindent
{\bf 2.} We first show the uniform bound of $u^\Lambda(t,x)$ 
with respect to all the parameters in $\Lambda$, {\it i.e.},
\begin{align}\label{buL}
    |u^\Lambda(t,x)|\leq (M_0+c_0T)e^{1+|c_1+c_2|T}=M.
\end{align}

In fact, from $\eqref{pe}$, 
\begin{align}\label{pe1}
  \mathcal{L}u^{\Lambda}:=\varepsilon\Delta u^{\Lambda}- u^{\Lambda}_t
  =f_i^{m,l}(t,x,u^{\Lambda})_{x_i}-g^{m,l}(t,x,u^{\Lambda})
  =:f(t,x,u^{\Lambda},\nabla_xu^{\Lambda}).
\end{align}
For constants $c_0$, $c_1$, and $c_2$ in $\eqref{ex1}$--$\eqref{ex2}$, 
by setting $\rho:=b-|c_1+c_2|>0$, we define
\begin{align*}
 v(t,x):=\big(\frac{c_0}{\rho}+M_0\big)e^{bt}.
\end{align*}
Then $v(0,x)\geq u_0^h(x) \geq -v(0,x)$. 
Noticing that $v\geq 0$, it follows from $\eqref{pe1}$ that
\begin{align*}
\mp f(t,x,\pm v,\nabla_x(\pm v))
&=\mp f_{ix_i}^{m,l}(t,x,\pm v)\pm g^{m,l}(t,x,\pm v)\nonumber\\
&=-f_{ix_iu}^{m,l}(t,x,\tilde{v})v\mp f_{ix_i}^{m,l}(t,x,0)
+\tfrac{g^{m,l}(t,x,\pm v)-g^{m,l}(t,x,0)}{\pm v-0}v\pm g^{m,l}(t,x,0)
\nonumber\\[1mm]
&\leq  (c_1+c_2)v+c_0 
\leq  |c_1+c_2|v+c_0
\nonumber\\[1mm]
&
=bv+c_0-\rho v
\leq bv 
=\mp \mathcal{L}(\pm v).
\end{align*}
This means that $\mathcal{L}v\leq f(t,x,v,\nabla_xv)$ and 
$\mathcal{L}(-v)\geq f(t,x,-v,\nabla_x(-v))$. 
Thus, by the maximum principle, $|u^\Lambda(t,x)|\leq v(t,x)$. 
Therefore, by taking $v(t,x)$ with $\rho=1/T$, we obtain
\begin{align*}
|u^\Lambda(t,x)|\leq v(t,x)
=\big(c_0T+M_0\big)e^{|c_1+c_2|t+t/T}
\leq \big(M_0+c_0T\big)e^{1+|c_1+c_2|T}.
\end{align*}

\smallskip
\noindent
{\bf 3.} We now prove the uniform estimate 
of the $L^1$-continuity modulus in $x$ and $t$ of the solutions $u^\Lambda(t,x)$ with respect 
to parameters $(l,\varepsilon)$ 
in $\Lambda$: 
When $t\in [0,T]$,
\begin{eqnarray}
&&J_r(u^{\Lambda}(t,x),\Delta x):=
   \int_{B_r(0)}|u^{\Lambda}(t,x+\Delta x)-u^{\Lambda}(t,x)|\,{\rm d}x
   \leq \omega_r^{x;h,m}(|\Delta x|),\label{Jrux}\\[1mm]
&&I_r(u^{\Lambda}(t,x),\Delta t):=
\int_{B_r(0)}|u^{\Lambda}(t+\Delta t,x)-u^{\Lambda}(t,x)|\,{\rm d}x
\leq \omega_r^{t;h,m}(|\Delta t|), \label{Irut0}
\end{eqnarray}
for some functions $\omega_r^{x;h,m}(\sigma)$ and $\omega_r^{t;h,m}(\sigma)$ independent of 
$l,\varepsilon\in (0,1]$. 

In what follows, we denote $C_j(m)$ for $j\in \mathbb{N}$ the constants that 
depend only on parameter $m$ in $\Lambda$ and constants $(n, M)$.

\smallskip
\noindent
{\bf 3.1.} We first prove $\eqref{Jrux}$. 
For $k=1,2,\cdots,n$, let $v^{\Lambda}_k=v^{\Lambda}_k(t,x):=u^{\Lambda}_{x_k}(t,x)$. 
By taking the derivatives $\partial_{x_k}$ on $\eqref{pe}$, we obtain
\begin{align}\label{pew}
    (v^{\Lambda}_k)_t
    +\big(f_{iu}^{m,l}\, v^{\Lambda}_k\big)_{x_i}
    +\big(f_{iux_k}^{m,l}-\delta_{ik}\,g^{m,l}_u\big) v^{\Lambda}_i
    +f_{ix_ix_k}^{m,l}-g^{m,l}_{x_k}
    =\varepsilon\Delta v^{\Lambda}_k,
\end{align}
where the partial derivatives of $f^{m,l}_i$ and $g^{m,l}$ are functions of $(t,x,u^{\Lambda}(t,x))$, 
and $\delta_{ik}=1$ if $i=k$ and $\delta_{ik}=0$ if $i\neq k$.

For $k=1,2,\cdots,n$, multiplying $\eqref{pew}$ by a compactly supported function 
$\phi^k=\phi^k(t,x)$ in $\Pi_\tau$ for $\tau\in(0,T]$ with continuous derivatives $\phi^k_t$, 
$\phi^k_{x_i}$, and $\phi^k_{x_ix_j}$,
integrating it over $\Pi_{\tau}$, summing over $k$ from $1$ to $n$, and integrating by part, we obtain
\begin{align}\label{pewtau}
\int_{\mathbb{R}^n}v^{\Lambda}_k\,\phi^k\big|_{t=0}^{t=\tau}\,{\rm d}x=
\iint_{\Pi_{\tau}}v^{\Lambda}_k\,\mathcal{L}_k(\phi)\,{\rm d}t{\rm d}x
+\iint_{\Pi_{\tau}}\big({-}f_{ix_ix_k}^{m,l}+g^{m,l}_{x_k}\big)\phi^k\,{\rm d}t{\rm d}x,
\end{align}
where $\phi:=(\phi^1,\cdots,\phi^n)$, and the adjoint linear parabolic operator $\mathcal{L}_k$ is given by
\begin{align}\label{pewL}
\mathcal{L}_k(\phi):=\phi^k_t+f^{m,l}_{iu}\,\phi^k_{x_i}
+\big({-}f_{kux_i}^{m,l}+\delta_{ki}\,g^{m,l}_u\big)\phi^i
+\varepsilon \Delta \phi^k.
\end{align}

Now, fix $r>1$. Let 
$q_{h'}(t,x):=(q_{h'}^1(t,x),\cdots,q_{h'}^n(t,x))$ solves 
the Cauchy problem of the following linear parabolic system in $\Pi_{\tau}$:
\begin{align}\label{Lqh}
\mathcal{L}_k(q_{h'}(t,x))=0,\qquad q_{h'}^k(\tau,x)=\beta_{h'}^k(x),
\end{align}
where $\beta_{h'}^k(x)$ with $h'\in (0,1]$ is the modifying function of $\beta^k(x)$ given by 
$$
\beta^k(x)={\rm sgn}\,v^{\Lambda}_k(\tau,x)\quad {\rm if}\ |x|\leq r-h',
\qquad\,\, \beta^k(x)=0 \quad {\rm if}\ |x|>r-h'.
$$
Then it follows from $\eqref{gmlu}$ and $\eqref{pewL}$--$\eqref{Lqh}$ that
\begin{align}\label{pewL2}
0=2q_{h'}^k\mathcal{L}_k(q_{h'})
&=((q_{h'})^2)_t+f^{m,l}_{iu}\,((q_{h'})^2)_{x_i}
-2f_{kux_i}^{m,l}\,q^k_{h'}\,q^i_{h'}
+2\,g^{m,l}_u\,(q_{h'})^2
+2\varepsilon q_{h'}^k \Delta q_{h'}^k
\nonumber\\
&\leq ((q_{h'})^2)_t+f^{m,l}_{iu}\,((q_{h'})^2)_{x_i}
+2|f_{kux_i}^{m,l}|\,|q^k_{h'}|\,|q^i_{h'}|
+2c_2\,(q_{h'})^2+\varepsilon \Delta (q_{h'})^2
\nonumber\\
&\leq((q_{h'})^2)_t+f^{m,l}_{iu}\,((q_{h'})^2)_{x_i}+C_{2}(m) (q_{h'})^2
+\varepsilon \Delta (q_{h'})^2
=:\mathcal{L}((q_{h'})^2),
\end{align}
where $(q_{h'})^2:=q_{h'}^k\,q_{h'}^k$ and constant $C_{2}(m):=2nC_1(m)+2c_2$
with $C_1(m)$ being the uniform bound of the continuous functions $\{f^{m,l}_{kux_i}\}_{k,i=1}^n$ on $[0,T]\times B_m(0)\times [-M,M]$.
By the maximum principle, the function $(q_{h'})^2$ with $q_{h'}(\tau,x)=0$ for $|x|\geq r$ satisfies 
\begin{align}\label{qQh}
(q_{h'})^2\leq C_3(m) \qquad \mbox{for $(t,x)\in \Pi_{\tau}$}.
\end{align}

Denote $Q_{\varepsilon,m}:=Q_{\varepsilon,m}(t,x)$ by
\begin{align*}
Q_{\varepsilon,m}(t,x):=C_3(m)\exp\big\{\varepsilon^{-1}\big( H(\tau-t)+r-|x|\big)
 +\tau|C_{2}(m)|+(\tau-t)C_{2}(m)\big\},
\end{align*}
where $H:=1+\bar{N}$, and $\bar{N}$ is given by $\eqref{ex0}$. Let $\Omega$ be given by 
\begin{align*}
\Omega:=\big\{(t,x)\,:\, |x|\geq r+H(\tau-t),\, t\in [0,\tau]\big\}.
\end{align*}
Then we see that 
$$
Q_{\varepsilon,m}(t,x)\big|_{|x|=r+H(\tau-t)}\geq C_3(m), 
\qquad Q_{\varepsilon,m}(\tau,x)\big|_{|x|\geq r}\geq 0,
$$
and, for any $(t,x)\in \Omega$,
$$\mathcal{L}(Q_{\varepsilon,m})=Q_{\varepsilon,m}\big\{{-}\varepsilon^{-1}\big(\bar{N}+f^{m,l}_{iu}x_i|x|^{-1}\big)
-(n{-}1)|x|^{-1}\big\}\leq 0.
$$
By $\eqref{pewL2}$--$\eqref{qQh}$, it implies that
\begin{align}\label{qQ}
(q_{h'})^2=(q_{h'}(t,x))^2\leq Q_{\varepsilon,m}(t,x) \qquad {\rm for\ any}\ (t,x)\in \Omega.
\end{align}
Furthermore, $\eqref{qQ}$ implies that,
for any $\varepsilon\in (0,1]$, 
if $|x|\geq \frac{1}{1-\sqrt{\varepsilon}/2}(H\tau+r)$, then
\begin{align}\label{qhx2}
(q_{h'})^2
\leq C_3(m)\exp\big\{{-}\tfrac{1}{2\sqrt{\varepsilon}}|x|+2|C_{2}(m)|T\big\}
=:C_4(m) \exp\big\{{-}\tfrac{1}{2\sqrt{\varepsilon}}|x|\big\}.
\end{align}
Thus, for $R\geq \bar{r}:=2r+2(1+\bar{N})T>1$, 
\begin{align}\label{qhx0}
\int_{\mathbb{R}^n/B_R(0)}|q_{h'}^k(0,x)|\,{\rm d}x
&\leq \sqrt{C_4(m)} \int_{\mathbb{R}^n/B_R(0)}\exp\big\{{-}\tfrac{1}{4\sqrt{\varepsilon}}|x|\big\}\,{\rm d}x
\nonumber\\[1mm]
&
\leq C_5(m)\, R^{n-1}\exp\big\{{-}\tfrac{1}{4\sqrt{\varepsilon}}R\big\}.
\end{align}

According to $\eqref{Lqh}$, 
by taking $\phi^k=\phi^k(t,x):=q_{h'}^k(t,x)\eta^{m'}(|x|)$ into $\eqref{pewtau}$ 
and using integration by parts, we obtain
\begin{align}\label{pewtau1}
\int_{\mathbb{R}^n}v^{\Lambda}_k\,q_{h'}^k\,\eta^{m'}\big|_{t=\tau}\,{\rm d}x
&=\int_{\mathbb{R}^n}v^{\Lambda}_k\,q_{h'}^k\,\eta^{m'}\big|_{t=0}\,{\rm d}x
+\iint_{\Pi_{\tau}}\big({-}f_{ix_ix_k}^{m,l}+g^{m,l}_{x_k}\big)q_{h'}^k\,\eta^{m'}\,{\rm d}t{\rm d}x\nonumber\\
&\quad\,\,\,+\iint_{\Pi_{\tau}}
\Big\{\big(f^{m,l}_{iu}\,v^{\Lambda}_k-2\varepsilon (v^{\Lambda}_k)_{x_i}\big)
\frac{x_i}{|x|}\delta(|x|-m')
-\varepsilon v^{\Lambda}_k\Delta \eta^{m'}\Big\}q_{h'}^k
\,{\rm d}t{\rm d}x\nonumber\\
&=:A_1+A_2.
\end{align}
Since $\delta(|x|-m')\equiv 0$ for $|x|<m'-1$, $\lim_{m'\rightarrow \infty}A_2=0$.
Letting $m'\rightarrow \infty$ in $\eqref{pewtau1}$,
we have
\begin{align*}
\int_{\mathbb{R}^n}v^{\Lambda}_k\,q_{h'}^k\big|_{t=\tau}\,{\rm d}x
&=\int_{\mathbb{R}^n}v^{\Lambda}_k\,q_{h'}^k\big|_{t=0}\,{\rm d}x
+\iint_{\Pi_{\tau}}\big({-}f_{ix_ix_k}^{m,l}+g^{m,l}_{x_k}\big)q_{h'}^k\,{\rm d}t{\rm d}x.
\end{align*}
By $\eqref{qQh}$ and $\eqref{qhx2}$--$\eqref{qhx0}$, 
for any $R\geq \bar{r}$,
\begin{align}\label{pewtau3}
\int_{\mathbb{R}^n}v^{\Lambda}_k(\tau,x)q_{h'}^k(\tau,x)\,{\rm d}x
&\leq \int_{\mathbb{R}^n}|v^{\Lambda}_k(0,x)||q_{h'}^k(0,x)|\,{\rm d}x
+\iint_{[0,\tau]\times B_m(0)}\big|{-}f_{ix_ix_k}^{m,l}+g^{m,l}_{x_k}\big||q_{h'}^k|\,{\rm d}t{\rm d}x\nonumber\\
&\leq n\sqrt{C_3(m)}\frac{M_0}{h}|B_R(0)| 
+\frac{M_0}{h}\int_{\mathbb{R}^n/B_R(0)}\sum_{k=1}^n|q_{h'}^k(0,x)|\,{\rm d}x
+C_6(m)\nonumber\\[1mm]
&\leq C_7(m)\Big(\frac{R^n}{h}+\frac{R^{n-1}}{h}\exp\big\{{-}\tfrac{1}{4\sqrt{\varepsilon}}R\big\}+1\Big)
\nonumber\\[1mm]
&
\leq C_8(m)\frac{R^n}{h},
\end{align}
where we have used the fact that 
$|v^{\Lambda}_k(0,x)|=|\partial_{x_k}(u^h_0(x))|\leq \frac{M_0}{h}$.
Then, 
by letting $h'\rightarrow 0$,
\begin{align*}
\int_{B_r(0)}\sum^n_{k=1}|v^{\Lambda}_k(\tau,x)|\,{\rm d}x
\leq C_8(m) \frac{R^n}{h} 
\qquad {\rm for\ any}\ R\geq \bar{r}.
\end{align*}
By the arbitrariness of $\tau\in (0,T]$ and taking $R=\bar{r}$, we conclude that, for any $t\in [0,T]$,
\begin{align}\label{Jrux1}
J_r(u^{\Lambda}(t,x),\Delta x)
\leq C_9(m)\big(r+(1+\bar{N})T\big)^n\, \frac{|\Delta x|}{h}
=:\omega_r^{x;h,m}(|\Delta x|).
\end{align}

\smallskip
\noindent
{\bf 3.2.} For the uniform estimate 
of the $L^1$-continuity modulus in $t$ of $u^\Lambda(t,x)$ 
with respect to parameters $(l,\varepsilon)$ in $\Lambda$ as in $\eqref{Irut0}$, 
by the same arguments in \S 4 of Kruzkov \cite{[Ksn2]}, 
it follows from $\eqref{Jrux1}$ that, 
for any $t\in [0,T]$,
\begin{align}\label{Irut}
I_r(u^{\Lambda}(t,x),\Delta t)&=
\int_{B_r(0)}|u^{\Lambda}(t+\Delta t,x)-u^{\Lambda}(t,x)|\,{\rm d}x\nonumber\\
&\leq C_{10}(m)C_1(r) \min_{0<\sigma\leq 1} \Big\{\omega_r^{x;h,m}(\sigma)+\sigma+\sigma^{-2}|\Delta t| \Big\}
\nonumber\\
&\leq C_9(m)C_{10}(m)C_1(r)\big(r+(1+\bar{N})T\big)^n\, 
\frac{|\Delta t|^{1/3}}{h^{2/3}}
\nonumber\\
&
=:\omega_r^{t;h,m}(|\Delta t|),
\end{align}
where constant $C_1(r)$ depends only on $r$ 
and constants $(n, M)$.

\smallskip
\noindent
{\bf 4.} We now pass the limit of $u^{\Lambda}(t,x)$ with respect to all the parameters in $\Lambda$ to 
find the Kruzkov entropy solution $u(t,x)$ of 
the Cauchy problem $\eqref{1.1}$--$\eqref{ID}$ 
in the sense of Definition $\ref{def:2.2}$.

Let $\Phi(u)\in C^2(\mathbb{R})$ be a convex function and $\phi\in C^\infty_{\rm c}(\mathring{\Pi}_{_T})$ with $\phi\geq 0$. 
Multiplying $\eqref{pe}$ by $\Phi'(u^{\Lambda})\phi$, integrating it over $\Pi_{_{T}}$, and integrating by part, we obtain
\begin{align}\label{Phiphi}
0\leq \iint_{\Pi_{_{T}}}&
\varepsilon \phi\Phi''(u^{\Lambda})u^{\Lambda}_{x_i}u^{\Lambda}_{x_i}
\,{\rm d}t{\rm d}x\nonumber\\
=\iint_{\Pi_{_{T}}}
\Big\{&\varepsilon\Phi(u^{\Lambda})\Delta \phi
-\phi\Phi(u^{\Lambda})_t-\phi \Phi'(u^{\Lambda})f_{iu}^{m,l}(t,x,u^{\Lambda})u^{\Lambda}_{x_i}
\nonumber\\
&-\phi \Phi'(u^{\Lambda})f_{ix_i}^{m,l}(t,x,u^{\Lambda})
+\phi \Phi'(u^{\Lambda})g^{m,l}(t,x,u^{\Lambda})\Big\}
\,{\rm d}t{\rm d}x
\nonumber\\
=\iint_{\Pi_{_{T}}}
\Big\{&\Phi(u^{\Lambda}) (\phi_t+\varepsilon\Delta \phi)
+\int_k^{u^{\Lambda}}\Phi'(\zeta)f_{iu}^{m,l}(t,x,\zeta)\,{\rm d}\zeta\, \phi_{x_i}
-\Phi'(u^{\Lambda})f_{ix_i}^{m,l}(t,x,u^{\Lambda})\phi
\nonumber\\
&+\int_k^{u^{\Lambda}}\Phi'(\zeta)f_{iux_i}^{m,l}(t,x,\zeta)\,{\rm d}\zeta\,\phi
+\Phi'(u^{\Lambda})g^{m,l}(t,x,u^{\Lambda})\phi \Big\}
\,{\rm d}t{\rm d}x.
\end{align}

\smallskip
\noindent
{\bf 4.1.} 
We first fix parameters $(l,\varepsilon,m)$ in $\Lambda$. Denote $\Lambda_j:=(h_j,l,\varepsilon,m)$ for $j=1,2$. 
Multiplying $\eqref{pe}$ by a compactly supported function $\phi^k=\phi^k(t,x)$ with continuous derivatives
 $\phi^k_t$, $\phi^k_{x_i}$, and $\phi^k_{x_ix_j}$ in $\Pi_\tau$ for $\tau\in(0,T]$, 
 integrating it over $\Pi_\tau$, and then integrating by part, 
 we obtain that, for $j=1,2$,
\begin{align*}
\int_{\mathbb{R}^n}u^{\Lambda_j}\phi\big|_{t=0}^{t=\tau}\,{\rm d}x
=\iint_{\Pi_{\tau}}
\big\{u^{\Lambda_j}(\phi_t+\varepsilon \Delta \phi)+f^{m,l}_i(t,x,u^{\Lambda_j})\phi_{x_i}
+g^{m,l}(t,x,u^{\Lambda_j})\phi\big\}\,{\rm d}t{\rm d}x.
\end{align*}
Let $w=w(t,x):=u^{\Lambda_2}(t,x)-u^{\Lambda_1}(t,x)$.
Then
\begin{align}\label{pewuu1}
\int_{\mathbb{R}^n}w\phi\big|_{t=0}^{t=\tau}\,{\rm d}x=
\iint_{\Pi_{\tau}}w\big(\phi_t+a_i\phi_{x_i}
+b\phi+\varepsilon \Delta \phi\big)\,{\rm d}t{\rm d}x, 
\end{align} 
where $a_i$ and $b$ are respectively given by
\begin{align*}
a_i&:=\int_0^1f_{iu}^{m,l}(t,x,(1-\theta)u^{\Lambda_1}(t,x)
+\theta u^{\Lambda_2}(t,x))\,{\rm d}\theta,
\nonumber\\
b&:=\int_0^1 g_{u}^{m,l}(t,x,(1-\theta)u^{\Lambda_1}(t,x)
+\theta u^{\Lambda_2}(t,x))\,{\rm d}\theta.
\end{align*}
By the same arguments done in Step {1.2}, it follows from $\eqref{uhu0}$ and $\eqref{pewuu1}$ that 
\begin{align}\label{pewuu10}
\int_{B_r(0)}\big|w(t,x)\big|\,{\rm d}x
&\leq C_{11}(m) \int_{\mathbb{R}^n}
\exp\big\{{-}\tfrac{1}{2\sqrt{\varepsilon}}|x|\big\}
\big|u^{h_2}_0(x)-u^{h_1}_0(x)\big|\,{\rm d}x
\nonumber\\
&\leq C_{11}(m) \sum_{j=1}^2\int_{B_R(0)}\big|u^{h_j}_0(x)-u_0(x)\big|\,{\rm d}x
+2M_0\int_{\mathbb{R}^n/B_R(0)}
\exp\big\{{-}\tfrac{1}{2\sqrt{\varepsilon}}|x|\big\}\,{\rm d}x
\nonumber\\[1mm]
&\leq C_{12}(m)\Big(\omega_R(\sqrt{n}h_1)+\omega_R(\sqrt{n}h_2)
+R^{n-1}\exp\big\{{-}\tfrac{1}{2\sqrt{\varepsilon}}R\big\}\Big)
\nonumber\\[1mm]
&\leq C_{12}(m)\Big(\omega_R(\sqrt{n}h_1)+\omega_R(\sqrt{n}h_2)
+R^{n-1}\exp\big\{{-}\tfrac{1}{2}R\big\}\Big)
\end{align}
for any $R\geq 1$ and $t\in [0,T]$. 
Then, integrating over $[0,T]$ and letting $h_1,h_2\rightarrow 0$, we obtain that, for any $R\geq 1$,
\begin{align*}
\lim_{h_1,h_2\rightarrow 0}
\int_{[0,T]\times B_r(0)}\big|u^{\Lambda_2}(t,x)-u^{\Lambda_1}(t,x)\big|\,{\rm d}t{\rm d}x
\leq C_{12}(m)\, R^{n-1}\exp\big\{{-}\tfrac{1}{2}R\big\}T.
\end{align*}
Letting $R\rightarrow \infty$ yields
\begin{align}\label{pewu1u2}
\lim_{h_1,h_2\rightarrow 0}
\int_{[0,T]\times B_r(0)}\big|u^{\Lambda_2}(t,x)-u^{\Lambda_1}(t,x)\big|\,{\rm d}t{\rm d}x=0.
\end{align}

Therefore, for any fixed parameters $(l,\varepsilon,m)$ in $\Lambda$, 
$\eqref{pewu1u2}$ implies that 
$\{u^{\Lambda}(t,x)\}_{h\in(0,1]}$ is a Cauchy sequence in $L^1([0,T]\times B_r(0))$,
so that there exists a subsequence that converges almost everywhere to 
a bounded measurable function $u^{\Lambda_0}(t,x)$ in $[0,T]\times B_r(0)$, 
where $\Lambda_0:=(l,\varepsilon,m)$. 
Furthermore, by using the diagonal process with respect to $r=2,3,\cdots$, 
we can find a subsequence of $\{u^{\Lambda}(t,x)\}$ 
that converges almost everywhere in $\Pi_{_{T}}$ to 
a bounded function $u^{\Lambda_0}(t,x)$. 
Then it follows from $\eqref{Phiphi}$ that, 
for any convex function $\Phi(u)\in C^2(\mathbb{R})$ and 
$\phi\in C^\infty_{c}(\mathring{\Pi}_{_T})$ with $\phi\geq 0$, 
\begin{align}\label{Phiphi0}
0\leq \iint_{\Pi_{_{T}}}
\Big\{&\Phi(u^{\Lambda_0}) (\phi_t+\varepsilon\Delta \phi)
+\int_k^{u^{\Lambda_0}}\Phi'(\zeta)f_{iu}^{m,l}(t,x,\zeta)\,{\rm d}\zeta\, \phi_{x_i}
-\Phi'(u^{\Lambda_0})f_{ix_i}^{m,l}(t,x,u^{\Lambda_0})\phi
\nonumber\\
&+\int_k^{u^{\Lambda_0}}\Phi'(\zeta)f_{iux_i}^{m,l}(t,x,\zeta)\,{\rm d}\zeta\,\phi
+\Phi'(u^{\Lambda_0})g^{m,l}(t,x,u^{\Lambda_0})\phi \Big\}
\,{\rm d}t{\rm d}x.
\end{align}

\smallskip
\noindent
{\bf 4.2.} We now fix parameter $m$ in $\Lambda_0$. 
Clearly, 
$|u^{\Lambda_0}(t,x)|\leq M$. 
Letting $h_2\rightarrow 0$ and replacing $h_1$ by $h$ in $\eqref{pewuu10}$, we obtain that, 
for any $t\in [0,T]$ and $R\geq 1$,
\begin{align*}
\int_{B_r(0)}\big|u^{\Lambda}(t,x)-u^{\Lambda_0}(t,x)\big|\,{\rm d}x
\leq C_{12}(m)\Big(\omega_R(\sqrt{n}h)
+R^{n-1}\exp\big\{{-}\tfrac{1}{2}R\big\}\Big).
\end{align*}
Thus, by $\eqref{Jrux1}$--$\eqref{Irut}$, 
for any $t\in [0,T]$, 
\begin{align}
J_r(u^{\Lambda_0},\Delta x)&\leq C_{13}(m)C_2(r) 
\min_{\genfrac{}{}{0pt}{3}{h\in (0,1]}{R\geq 1}}
\Big\{\frac{|\Delta x|}{h}
+\omega_R(\sqrt{n}h)+R^{n-1}\exp\big\{{-}\tfrac{1}{2}R\big\}\Big\}\nonumber\\
&=:\omega_r^{x;m}(|\Delta x|),\label{JIux}\\
I_r(u^{\Lambda_0},\Delta t)&\leq C_{14}(m)C_3(r) 
\min_{\genfrac{}{}{0pt}{3}{h\in (0,1]}{R\geq 1}}
\Big\{\frac{|\Delta t|^{1/3}}{h^{2/3}}
+\omega_R(\sqrt{n}h)+R^{n-1}\exp\big\{{-}\tfrac{1}{2}R\big\}\Big\}\nonumber\\
&=:\omega_r^{t;m}(|\Delta t|),\label{JIut}
\end{align}
where $\omega_r^{x;m}(\sigma)$ and $\omega_r^{t;m}(\sigma)$ are independent of  parameters $(l,\varepsilon)$ in $\Lambda_0$.

For any fixed $m\in \mathbb{N}$, $\eqref{buL}$ and $\eqref{JIux}$--$\eqref{JIut}$ imply that the family of 
functions $\{u^{\Lambda_0}(t,x)\}$ with $l,\varepsilon\in (0,1]$ is compact in the $L^1$--norm 
in any cylinder $[0,T]\times B_r(0)$ for $r=2,3,\cdots$. 
Thus, by passing the limit with respect to parameters $(l,\varepsilon)$, 
we can find a subsequence of $\{u^{\Lambda_0}(t,x)\}$ 
that converges almost everywhere in any cylinder $[0,T]\times B_r(0)$ 
to a bounded function $u^m(t,x)$.
Furthermore, using the diagonal process with respect to $r$, 
we can find a subsequence of $\{u^{\Lambda_0}(t,x)\}$ that converges almost everywhere in $\Pi_{_{T}}$ 
to a bounded measurable function $u^m(t,x)$. 
Restricting to this subsequence, for any fixed $m\in \mathbb{N}$, 
passing the limit in $\eqref{Phiphi0}$ with respect to parameters $(l,\varepsilon)$, 
we obtain that, for any convex function $\Phi(u)\in C^2(\mathbb{R})$ 
and $\phi\in C^\infty_{c}(\mathring{\Pi}_{_T})$ with $\phi\geq 0$, 
\begin{align*}
0\leq &\iint_{\Pi_{_{T}}}
\Big\{\Phi(u^{m}) \phi_t
+\eta^m(|x|)\Big(\int_k^{u^{m}}\Phi'(\zeta)f_{iu}(t,x,\zeta)\,{\rm d}\zeta\,\phi_{x_i}
-\Phi'(u^{m})f_{ix_i}(t,x,u^{m})\,\phi
\nonumber\\
&\qquad\quad\qquad\qquad\qquad\qquad\quad\,\,
+\int_k^{u^{m}}\Phi'(\zeta)f_{iux_i}(t,x,\zeta)\,{\rm d}\zeta\,\phi
+\Phi'(u^{m})g(t,x,u^{m})\,\phi \Big)\Big\}
\,{\rm d}t{\rm d}x
\nonumber\\
&+\iint_{\Pi_{_{T}}}\Big\{{-}\Phi'(u^{m})f_{i}(t,x,u^{m})+\int_k^{u^{m}}\Phi'(\zeta)f_{iu}(t,x,\zeta)\,{\rm d}\zeta\Big\}
\frac{x_i}{|x|}\,\delta(|x|-m)\,\phi 
\,{\rm d}t{\rm d}x.
\end{align*}
Therefore, approximating the function $|u-k|$ via convex functions $\Phi(u)\in C^2(\mathbb{R})$, 
we see that, for any $k\in \mathbb{R}$ and
$\phi\in C^\infty_{\rm c}(\mathring{\Pi}_{_T})$ with $\phi\geq 0$, 
\begin{align}\label{Phiphim2}
0\leq &\iint_{\Pi_{_{T}}}
\Big\{|u^m-k| \phi_t
+\eta^m(|x|)\Big({\rm sgn}(u^m-k)\big(f_i(t,x,u^m)-f_i(t,x,k)\big)\,\phi_{x_i}
\nonumber\\
&\qquad \quad\,\,
+{\rm sgn}(u^m-k)\big({-}f_{ix_i}(t,x,k)+g(t,x,u^m)\big)\,\phi \Big)\Big\}
\,{\rm d}t{\rm d}x
\nonumber\\
&-\iint_{\Pi_{_{T}}}{\rm sgn}(u^m-k)f_i(t,x,k)
\frac{x_i}{|x|}\,\delta(|x|-m)\,\phi 
\,{\rm d}t{\rm d}x.
\end{align}
Taking $t=0$ in $\eqref{JIut}$, we obtain that, for any $r>0$ and $\rho\in [0,T]$,
\begin{align*}
\int_{B_r(0)}|u^{\Lambda_0}(\rho,x)-u_0(x)|\,{\rm d}x
\leq\omega_r^{t;m}(\rho).
\end{align*}
Thus, letting $(l,\varepsilon)\rightarrow (0,0)$, we conclude that, for any $r>0$,
\begin{align}\label{2.2am}
\lim_{\genfrac{}{}{0pt}{2}{\rho\rightarrow 0}{\rho\in[0,T]\setminus\mathcal{Z}}}
\int_{B_r(0)}|u^m(\rho,x)-u_0(x)|\,{\rm d}x=0.
\end{align}	

\smallskip
\noindent
{\bf 4.3.} 
Notice that $\eta^m(|x|)\equiv 1$ and $\delta(|x|-m)\equiv 0$ for $|x|\leq m-1$.
For any fixed $r>0$, 
if ${\rm spt}(\phi)\subset [0,T]\times B_r(0)$, then $\delta(|x|-m)\phi\equiv 0$ for any $m\geq \bar{r}+1=2r+2(1+\bar{N})T+1$.
Hence,$\eqref{Phiphim2}$ implies that $\eqref{2.2a}$ holds for $u^m(t,x)$ with $m\geq \bar{r}+1$.
By combining with $\eqref{2.2am}$, 
it follows from Theorem $\ref{th:5.1}$ that all the functions $u^m(t,x)$ for $m\geq \bar{r}+1$ must coincide 
with each other almost everywhere in the cylinder $[0,T]\times B_r(0)$. 
Thus, the sequence $u^m(t,x)$ converges almost everywhere in $\Pi_{_{T}}$ to a bounded measurable function $u(t,x)$. 
Therefore, it follows from $u(t,x)\equiv u^{m_r}(t,x)$ in $[0,T]\times B_r(0)$ for $m_r:=2+[\bar{r}]$ that 
$u(t,x)$ 
is an entropy solution of the Cauchy problem $\eqref{1.1}$--$\eqref{ID}$. 
\end{proof}

\section{Well-Posedness of the Characteristic ODEs: Proof of Lemma $\ref{lem:3.1}$}
From $\eqref{4.1}$, to solve the characteristic ODEs of equation $\eqref{eq1.1}$, it suffices to solve ODE $\eqref{3.3}$. 
Denote 
$G(u,t,s)$ by
\begin{equation}\label{3.4}
G(u,t,s)=\int^{u}_{s}\frac{{\rm d}\eta}{g(\eta)}-t.
\end{equation}
Then it is direct to see that, if $g(u)\neq 0$ for $u\in(s,\bar{u}(t,s))$, then
\begin{equation}\label{3.5}
\bar{u}(t,s)\ {\rm solves\ ODE\ \eqref{3.3}}  
\qquad {\rm if\ and\ only\ if}\qquad G(\bar{u}(t,s),t,s)=0.
\end{equation}

\smallskip
\begin{proof}[Proof of {\rm Lemma} $\ref{lem:3.1}$]
According to $\eqref{3.7}$--$\eqref{3.7a}$, 
we use $\eqref{3.5}$ to construct the continuous solution of ODE $\eqref{3.3}$ case by case, 
and prove (i)--(iii). 
The proof is divided into six steps.

\smallskip
\noindent
{\bf 1.} We first consider the case that $s \in A$, {\it i.e.}, $g(s)=0$.

\smallskip
{\bf (a).} If $s \in {\rm int} A$, 
then $\bar{u}(t,s) \equiv s$ for $t \in \mathbb{R}^+$ solves ODE $\eqref{3.3}$.
Hence, $\bar{u}_s(t,s) \equiv 1$ on $\mathbb{R}^+ \times {\rm int} A$. 
Then $\bar{u}(t,s)$ solves ODE $\eqref{3.3}$ uniquely on $\mathbb{R}^+ \times {\rm int} A$ 
and $\eqref{3.9}$ holds.

\smallskip
{\bf (b).} If $s \in \partial A$, we claim: 
{\it 
For any $s_0 \in \partial A$, 
$\bar{u}(t,s_0)\equiv s_0$ for $t \geq t_0$ 
is the unique solution of the following problem{\rm :}
\begin{equation}\label{3.11}
 \frac{{\rm d}\bar{u}(t,s_0)}{{\rm d}t}=g(\bar{u}(t,s_0)),\qquad \bar{u}(t,s_0)|_{t=t_0}=s_0. 
\end{equation}
}

In fact, it is clear that 
$\bar{u}(t,s_0)\equiv s_0$ for $t \geq t_0$ is a solution of problem $\eqref{3.11}$.
For the uniqueness, 
on the contrary,
assume that there exists another solution $\bar{v}(t,s_0)$ for $t \geq t_0$ of problem $\eqref{3.11}$. 
Then there exists $\bar{t}>t_0$ such that $\bar{v}(\bar{t},s_0)\neq s_0$. 
Without loss of the generality, assume that $\bar{v}(\bar{t},s_0)>s_0$. 
Let
\begin{equation*}
\underline{t}=\inf \big\{t \in [t_0,\bar{t}\,]\, : \, \bar{v}(\tau,s_0) > s_0
\quad {\rm for\,\, any}\,\, \tau \in [t,\,\bar{t}\,]\big\}.
\end{equation*}
Since $\bar{v}(t,s_0)$ is continuous with respect to $t$, 
then $\underline{t}<\bar{t}$.
Set $w(t):=\bar{v}(t,s_0)-s_0$. 
From the definition of $\underline{t}$, we have
\begin{equation}\label{3.14}
w(t)>0 \quad\mbox{on $(\underline{t},\,\bar{t}\,]$},\qquad\,\,\, w(\underline{t})=0.
\end{equation}
Then it follows from $\eqref{3.11}$ that
\begin{equation}\label{3.15}
\frac{{\rm d}w(t)}{{\rm d}t}=\frac{g(\bar{v}(t,s_0))}{\bar{v}(t,s_0)-s_0}\, w(t)
\qquad\,\, {\rm for\  any}\  t\in (\underline{t},\,\bar{t}\,].
\end{equation}	
Since $g(u)$ is right-Lipschitz in $u$, it follows from $g(s_0)=0$ that
\begin{equation*}
\frac{g(\bar{v}(t,s_0))}{\bar{v}(t,s_0)-s_0}
=\frac{g(\bar{v}(t,s_0))-g(s_0)}{\bar{v}(t,s_0)-s_0}\leq L
\qquad\,\, {\rm for\ some}\ L>0.
\end{equation*}
Integrating $\eqref{3.15}$ over $[t_1,t_2]\subset(\underline{t},\,\bar{t}\,]$,  
we derive the contradiction from $\eqref{3.14}$ as follows: 
\begin{align*}
0<w(t_2)&=w(t_1) \exp \Big \{ \int_{t_1}^{t_2}\frac{g(\bar{v}(t,s_0))}{\bar{v}(t,s_0)-s_0}
\, {\rm d}t \Big\}
\nonumber\\[1mm]
&
\leq w(t_1) \exp\big\{L(t_2-t_1)\big\}
\longrightarrow 0
\qquad\,\,\, {\rm as}\ t_1 \rightarrow \underline{t}{+}.
\end{align*}

Therefore, for any $s \in \partial A$, 
$\bar{u}\equiv s$ for $t \in \mathbb{R}^+$ is the unique solution of ODE $\eqref{3.3}$.

\smallskip
\noindent
{\bf 2.} We now consider the case that $s\in I_i= (a_i, b_i)$ with $a_i$ and $b_i$ being finite numbers.

\smallskip
From $\eqref{3.7a}$, $g(u)\neq 0$ for $u\in (a_i,b_i)$ and $g(a_i)=g(b_i)=0$.
Since $a_i,b_i \in \partial A$, from {\it Step} 1, 
the value of solution $\bar{u}(t,s)$ with $s\in (a_i,b_i)$ must belong to $ [a_i,b_i]$. 

For each $s \in (a_i,b_i)$, we define
\begin{equation}\label{3.18}
A_i(s)=\int_s^{a_i} \frac{{\rm d}\eta}{g(\eta)},\qquad B_i(s)=\int_s^{b_i} \frac{{\rm d}\eta}{g(\eta)}.
\end{equation}

Due to the sign of $g(u)$ on $(a_i,b_i)$, there are two subcases.
We now prove the subcase that $g(u)>0$ on $(a_i,b_i)$ only, 
since the proof for the other subcase 
is similar.

\smallskip
\noindent
{\bf 2.1.} In this step, we construct the solution $\bar{u}=\bar{u}(t,s)$ of 
ODE $\eqref{3.3}$ on $\mathbb{R}^+ \times (a_i,b_i)$.

\smallskip
From $\eqref{3.4}$, $G_u(u,t,s)=\frac{1}{g(u)}>0$. 
Then $G(u,t,s)$ is strictly increasing with respect to $u \in (a_i,b_i)$. 
Since $G(s,t,s)=-t<0$, then, for any fixed $t>0$ and $s\in (a_i,b_i)$, 
the root $u$ of $G(u,t,s)=0$ must belong to $(s,b_i]$, if it exists. 
Thus, to solve $G(u,t,s)=0$ with respect to $u$, 
it suffices to determine the sign of $G(b_i,t,s)$.
	
Since $g(u)>0$ on $(a_i,b_i)$, then, from $\eqref{3.18}$, 
$B_i(s)$ can be a positive number or $\infty$. 
When $B_i(s)=\infty$, then, for any $t\in \mathbb{R}^+$,
\begin{equation*}
G(b_i,t,s)=B_i(s)-t>0,
\end{equation*}
which implies that, for any $t \in \mathbb{R}^+$, 
there exists a unique $\bar{u} \in (s,b_i)$ such that $G(\bar{u},t,s)=0$. 
Since $G_u(u,t,s)=\frac{1}{g(u)}>0$, applying the implicit function theorem, 
we can immediately integrate $\eqref{3.3}$ to obtain 
the unique solution $\bar{u}=\bar{u}(t,s)$ of ODE $\eqref{3.3}$ 
for $t \in \mathbb{R}^+$, where $\bar{u}=\bar{u}(t,s)$ is uniquely determined 
by $G(\bar{u},t,s)=0$.
	
When $B_i(s)<\infty$, 
if $t<B_i(s)$, then $G(b_i,t,s)=B_i(s)-t>0$. 
Since $G(s,t,s)=-t<0$ and $G_u(u,t,s)=\frac{1}{g(u)}>0$, 
then there exists a unique $\bar{u} \in (s,b_i)$ such that $G(\bar{u},t,s)=0$. 
Thus, applying the implicit function theorem, 
we can immediately integrate $\eqref{3.3}$ to obtain the unique solution 
$\bar{u}=\bar{u}(t,s)$ of $\eqref{3.3}$ on $[0,B_i(s))$.
If $t=B_i(s)$, then $G(b_i,t,s)=0$ so that $\bar{u}(t,s)=b_i$; 
and if $t>B_i(s)$, letting $t_0:=B_i(s)$, then 
$$
\bar{u}\mid_{t=t_0}=\bar{u}(t_0,s)=b_i \in \partial A,
$$
which, by $\eqref{3.11}$, implies that 
the unique solution $\bar{u}(t,s)$ on $[B_i(s),\infty)$ is given by
\begin{equation*}
\bar{u}(t,s)=b_i\qquad\,\,\, {\rm for}\,\, t\geq B_i(s).
\end{equation*}

Up to now,  
the solution $\bar{u}=\bar{u}(t,s)$ of ODE $\eqref{3.3}$ on $\mathbb{R}^+ \times (a_i,b_i)$ has been constructed.

\medskip
\noindent
{\bf 2.2.} In this step, we show the regularity of solution $\bar{u}=\bar{u}(t,s)$ on $\mathbb{R}^+\times (a_i,b_i)$.

\smallskip
{\bf (a).} If $B_i(s)=\infty$, then $\bar{u}=\bar{u}(t,s)$ on $\mathbb{R}^+ \times (a_i,b_i)$ 
is uniquely determined by $G(\bar{u},t,s)=0$. 
Since $g(u)>0$ on $(a_i,b_i)$, then, from $\eqref{3.4}$, 
$G(u,t,s)$ is $C^1$ on $ (a_i,b_i) \times \mathbb{R}^+ \times (a_i,b_i)$. 
By the implicit function theorem, $\bar{u}(t,s)$ is $C^1$ on $\mathbb{R}^+ \times (a_i,b_i)$. 
Furthermore, taking $\partial_s$ act on $G(\bar{u},t,s)=0$, we obtain
\begin{equation*}
\bar{u}_s(t,s)=\frac{g(\bar{u}(t,s))}{g(s)}>0.
\end{equation*}	

{\bf (b).} If $B_i(s)<\infty$, then, from $B'_i(s)=-\frac{1}{g(s)}<0$ on $(a_i,b_i)$, 
$B_i(s)$ is strictly decreasing on $(a_i,b_i)$.
Thus, there exists a curve $t_*(s)=B_i(s)$ defined on $(a_i,b_i)$. 
For $(t,s)\in \mathbb{R}^+ \times (a_i,b_i)$ with $t<t_*(s)$, 
$\bar{u}=\bar{u}(t,s)$ is uniquely determined by $G(\bar{u},t,s)=0$, 
and hence $\bar{u}(t,s)$ is $C^1$. 
For $(t,s)\in \mathbb{R}^+ \times (a_i,b_i)$ with $t>t_*(s)$, $\bar{u}(t,s) \equiv b_i$. 
Therefore, $\bar{u}(t,s)$ is $C^1$, and $\eqref{3.8}$ holds for 
$(t,s)\in \mathbb{R}^+ \times (a_i,b_i)$ with $t<t_*(s)$ or $t>t_*(s)$.

For $(t,s)\in \mathbb{R}^+ \times (a_i,b_i)$ with $t=t_*(s)$, {\it i.e.},
$(t,s)$ lying on the curve $t=t_*(s)$, we first have $\bar{u}(t,s)=b_i$. 
Then, for small $\triangle t$ and $\triangle s$, 
\begin{equation}\label{3.22}
\int_s^{b_i} \frac{{\rm d}\eta}{g(\eta)}=t,\qquad 
\int_{s+\triangle s}^{\bar{u}(t+\triangle t,s+\triangle s)} \frac{{\rm d}\eta}{g(\eta)}=t+\triangle t.
\end{equation}
Denote $\triangle \bar{u}:=\bar{u}(t+\triangle t,s+\triangle s)-\bar{u}(t,s)
=\bar{u}(t+\triangle t,s+\triangle s)-b_i$. 
If $\bar{u}(t+\triangle t,s+\triangle s)=b_i$, then $\triangle \bar{u}=0$.
If $\bar{u}(t+\triangle t,s+\triangle s) \neq b_i$, 
it follows from $\eqref{3.22}$ that 
\begin{equation}\label{3.23}
\triangle \bar{u}=\frac{g(\bar{u}(t,s)+\theta_1\triangle \bar{u})}{g(s+\theta_2 \triangle s)} \triangle s 
+g(\bar{u}(t,s)+\theta_1\triangle \bar{u}) \triangle t,
\end{equation}
where $\theta_i \in(0,1)$ for $i=1,\,2$.
	
Since $g(u)\in C(\mathbb{R})$ and $g(s+\theta_2\triangle s)>0$, 
then the coefficients of $\triangle s$ and $\triangle t$ are both bounded 
for $\triangle t$ and $\triangle s$ small. 
Thus,
$$
\triangle \bar{u} \longrightarrow 0  \qquad \mbox{as $\triangle t,\,\triangle s \rightarrow 0$},
$$ 
which means that $\bar{u}(t,s)$ is continuous on curve $t=t_*(s)$. 
Thus, as $\triangle t, \triangle s \rightarrow 0$,
$$
\frac{g(\bar{u}(t,s)+\theta_1\triangle \bar{u})}{g(s+\theta_2 \triangle s)}\longrightarrow 
\frac{g(\bar{u}(t,s))}{g(s)},
\qquad\,\,\, g(\bar{u}(t,s)+\theta_1\triangle \bar{u})\longrightarrow g(\bar{u}(t,s)).
$$
Then it follows from $\eqref{3.23}$ that
\begin{equation*}
{\rm d}\bar{u}=\frac{g(\bar{u}(t,s))}{g(s)}\,{\rm d}s +g(\bar{u}(t,s))\, {\rm d}t.
\end{equation*}	
This implies that $\bar{u}(t,s)$ is $C^1$ along the curve: $t=t_*(s)$, and satisfies 
$$
\bar{u}_s(t,s)=\frac{g(\bar{u}(t,s))}{g(s)}=\frac{g(b_i)}{g(s)}=0.
$$
	
Therefore, we have proved that $\bar{u}(t,s)\in C^1(\mathbb{R}^+ \times (a_i,b_i))$ 
satisfies $\eqref{3.3}$ and $\eqref{3.8}$.

\medskip
\noindent
{\bf 3.} We now consider the case that $s\in I_i= ({-}\infty,b_i)$ with $b_i$ being a finite number. 

From $\eqref{3.7a}$, $g(u)\neq 0$ for $u\in ({-}\infty,b_i)$ and $g(b_i)=0$.
Since $b_i\in \partial A$, from {\it Step} 1, 
the value of solution $\bar{u}(t,s)$ with $s\in ({-}\infty,b_i)$ must belong to $ ({-}\infty,b_i]$. 
Due to the sign of $g(u)$ on $({-}\infty,b_i)$, there exist two subcases:

\smallskip
\noindent
{\bf 3.1.} Subcase: $g(u)>0$ on $({-}\infty,b_i)$. 

From $\eqref{3.4}$, $G_u(u,t,s)=\frac{1}{g(u)}>0$, 
then $G(u,t,s)$ is strictly increasing with respect to $u \in ({-}\infty,b_i)$. 
Since $G(s,t,s)=-t<0$, then, for any fixed $t>0$ and $s\in ({-}\infty,b_i)$, 
the root $u$ of $G(u,t,s)=0$ must belong to $(s,b_i]$, if it exists. 
Thus, to solve $G(u,t,s)=0$ with respect to $u$, 
it suffices to determine the sign of $G(b_i,t,s)$. 
This follows via replacing $a_i$ by ${-}\infty$ in the proof 
of {\it Step} 2 without any modification.

\smallskip
\noindent
{\bf 3.2.} Subcase: $g(u)<0$ on $({-}\infty,b_i)$. 

From $\eqref{3.4}$, $G_u(u,t,s)=\frac{1}{g(u)}<0$, 
then $G(u,t,s)$ is strictly decreasing with respect to $u\in ({-}\infty,b_i)$. 
Since $G(s,t,s)=-t<0$, then, for any fixed $t>0$ and $s\in ({-}\infty,b_i)$, 
the root $u$ of $G(u,t,s)=0$ must belong to $({-}\infty,s)$, if it exists. 
Thus, to solve $G(u,t,s)=0$ with respect to $u$, 
it suffices to determine the sign of $G({-}\infty,t,s)$.

From $\eqref{conhg}$, 
\begin{equation*}
0\leq \mathop{\overline{\lim}}\limits_{\eta \rightarrow {-}\infty}
\frac{g(\eta)}{\eta}=\mathop{\overline{\lim}}\limits_{\eta \rightarrow {-}\infty}
\frac{|g(\eta)|}{|\eta|}=\lim_{\eta \rightarrow {-}\infty}
\sup_{v\in ({-}\infty,\eta]}\frac{|g(v)|}{|v|}<\infty,
\end{equation*}
which means that there exists $C_->0$ such that 
$0<\frac{|g(\eta)|}{|\eta|}\leq C_-$ for any $\eta<b_i$. 
Then, for any given $s\in ({-}\infty,b_i)$,
\begin{equation*}
\int^u_s \frac{{\rm d}\eta}{g(\eta)}=\int^s_u \frac{{\rm d}\eta}{|g(\eta)|}\geq
\int^{\min\{-1,s\}}_u \frac{{\rm d}\eta}{C_- |\eta|} \geq\frac{1}{C_-}
(\ln |u|-\ln |s|)\rightarrow \infty\qquad \mbox{as $u\rightarrow {-}\infty$},
\end{equation*}
which implies that, for any $t\in \mathbb{R}^+$ and $s\in ({-}\infty,b_i)$,
\begin{equation}\label{3.33}
G({-}\infty,t,s)=\lim_{u \rightarrow {-}\infty}\int^{u}_s \frac{{\rm d}\eta}{g(\eta)}-t=\infty.
\end{equation}
Since $G_u(u,t,s)=\frac{1}{g(u)}<0$ and $G(s,t,s)<0$, 
it follows from $\eqref{3.33}$ that, 
for any given $t\in \mathbb{R}^+$ and $s\in ({-}\infty,b_i)$, 
there exists a unique $\bar{u}\in ({-}\infty,s)$ such that $G(\bar{u},t,s)=0$. 
Thus, applying the implicit function theorem, 
we can immediately integrate $\eqref{3.3}$ to obtain the unique solution 
$\bar{u}=\bar{u}(t,s)$ of ODE $\eqref{3.3}$ for $t \in \mathbb{R}^+$, 
and $\bar{u}(t,s)$ is $C^1$ on $\mathbb{R}^+ \times ({-}\infty,b_i)$. 
Furthermore, taking $\partial_s$ on $G(\bar{u},t,s)=0$ lead to $\eqref{3.8}$.

Therefore, we have proved that the unique solution $\bar{u}(t,s)$ of ODE $\eqref{3.3}$ 
is $C^1$ on $\mathbb{R}^+ \times ({-}\infty,b_i)$ and satisfies $\eqref{3.8}$.

\medskip
\noindent
{\bf 4.} We now consider the case that $s\in I_i= (a_i, \infty)$ with $a_i$ being a finite number. 

\smallskip
From $\eqref{3.7a}$, $g(u)\neq 0$ for $u\in (a_i,\infty)$ and $g(a_i)=0$.
Since $a_i\in \partial A$, from {\it Step} 1, 
the value of solution $\bar{u}(t,s)$ with $s\in (a_i, \infty)$ must belong to $ [a_i,\infty)$. 
Due to the sign of $g(u)$ on $(a_i,\infty)$, there are two subcases:

\smallskip
\noindent
{\bf 4.1.} Subcase: $g(u)>0$ on $(a_i,\infty)$. 

From $\eqref{3.4}$, $G_u(u,t,s)=\frac{1}{g(u)}>0$, 
then $G(u,t,s)$ is strictly increasing with respect to $u \in (a_i,\infty)$. 
Since $G(s,t,s)=-t<0$, then, for any fixed $t>0$ and $s\in (a_i,\infty)$, 
the root $u$ of $G(u,t,s)=0$ must belong to $(s,\infty)$, if it exists. 
Thus, to solve $G(u,t,s)=0$ with respect to $u$, 
it suffices to determine the sign of $G(\infty,t,s)$.

From $\eqref{conhg}$, 
$$
0\leq \mathop{\overline{\lim}}\limits_{\eta \rightarrow \infty}
\frac{g(\eta)}{\eta}=\lim_{\eta \rightarrow \infty}
\sup_{v\in [\eta,\infty)}\frac{g(v)}{v}<\infty,
$$
which means that there exists $C_+>0$ such that, 
$0<\frac{g(\eta)}{\eta}\leq C_+$ for any $\eta>a_i$. 
Then, for any given $s\in (a_i,\infty)$,
\begin{equation*}
\int^u_s \frac{{\rm d}\eta}{g(\eta)}\geq
\int^u_{\max\{1,s\}} \frac{{\rm d}\eta}{C_+ \eta} \geq\frac{1}{C_+}
(\ln |u|-\ln |s|)\rightarrow \infty\qquad \mbox{\rm as $u\rightarrow \infty$},
\end{equation*}
which implies that, 
for any $t\in \mathbb{R}^+$ and $s\in (a_i,\infty)$,
\begin{equation}\label{3.30}
G(\infty,t,s)=\lim_{u \rightarrow \infty}\int^{u}_s \frac{{\rm d}\eta}{g(\eta)}-t=\infty.
\end{equation}
Since $G_u(u,t,s)=\frac{1}{g(u)}>0$ and $G(s,t,s)<0$, it follows from $\eqref{3.30}$ that, 
for any given $t \in \mathbb{R}^+$ and $s\in (a_i,\infty)$, 
there exists a unique $\bar{u}\in (s,\infty)$ such that $G(\bar{u},t,s)=0$.
Thus, applying the implicit function theorem, 
we can immediately integrate $\eqref{3.3}$ to obtain the unique solution 
$\bar{u}=\bar{u}(t,s)$ of ODE $\eqref{3.3}$ for $t \in \mathbb{R}^+$, 
and $\bar{u}(t,s)$ is $C^1$ on $\mathbb{R}^+ \times (a_i,\infty)$. 
Furthermore, by taking $\partial_s$ act on $G(\bar{u},t,s)=0$, 
$\eqref{3.8}$ follows.

\medskip
\noindent
{\bf 4.2.} Subcase: $g(u)<0$ on $(a_i,\infty)$. 

\smallskip
From $\eqref{3.4}$, $G_u(u,t,s)=\frac{1}{g(u)}<0$, 
then $G(u,t,s)$ is strictly decreasing with respect to $u \in (a_i,\infty)$. 
Since $G(s,t,s)=-t<0$, then, for any fixed $t>0$ and $s\in (a_i,\infty)$, 
the root $u$ of $G(u,t,s)=0$ must belong to $[a_i,s)$, if it exists. 
Thus, to solve $G(u,t,s)=0$ with respect to $u$, 
it suffices to determine the sign of $G(a_i,t,s)$. 
This follows via replacing $b_i$ by $\infty$ in {\it Step} 2.

Therefore, we have proved that the unique solution $\bar{u}(t,s)$ of ODE $\eqref{3.3}$ 
is $C^1$ on $\mathbb{R}^+ \times (a_i,\infty)$ and satisfies $\eqref{3.8}$.

\medskip
\noindent
{\bf 5.} We now consider the case that $s\in I_i= ({-}\infty,\infty)$.

In this case, $A=\emptyset$ and hence $g(u)$ keeps the sign for all $u \in \mathbb{R}$.
In fact, for any fixed $s\in\mathbb{R}$, 
if $g(u)>0$, then 
$$
G(\infty,t,s)>0,\qquad G(s,t,s)<0,\qquad G_u(u,t,s)>0;
$$ 
and, if $g(u)<0$, then 
$$
G({-}\infty,t,s)>0,\qquad G(s,t,s)<0,\qquad G_u(u,t,s)<0.
$$
By the same arguments in {\it Step} 3--4, it can be checked that 
the unique solution $\bar{u}(t,s)$ of ODE $\eqref{3.3}$ 
is $C^1$ on $\mathbb{R}^+ \times \mathbb{R}$ and satisfies $\eqref{3.8}$.

\medskip
\noindent
{\bf 6.} Based on {\it Steps} 1--5, we have proved that, for any $s \in \mathbb{R}$, 
there exists a unique solution $\bar{u}(t,s)$ of ODE $\eqref{3.3}$ 
for $t \in \mathbb{R}^+$ 
such that $\bar{u}(t,s)\in C^1(\mathbb{R}^+ \times (\mathbb{R}\setminus \partial A))$ and satisfies $\eqref{3.8}$ and $\eqref{3.9}$.
Thus, to complete the proof, it suffices to show $\eqref{3.10}$.

Suppose that $s_0 \in \partial A$. Then $g(s_0) = 0$ and $\bar{u}(t,s_0)\equiv s_0$. 
Denote $v(t,\triangle s)$ by
\begin{equation}\label{3.36a}
v(t,\triangle s):= \bar{u}(t,s_0+\triangle s)- \bar{u}(t,s_0)=\bar{u}(t,s_0+\triangle s)- s_0.
\end{equation}
Then it follows from $\eqref{3.11}$ in {\it Step} 1 that, for any fixed $t_0>0$, 
\begin{itemize}
\item[(a).] If $v(t_0,\triangle s)=0$, then $v(t,\triangle s)=0$ on $t \geq t_0$;

\smallskip
\item[(b).] If $v(t_0,\triangle s)\neq 0$, then $v(t,\triangle s)\neq 0$ on $t\leq t_0$. 
\end{itemize}
Thus, if $v(t_0,\triangle s)\neq 0$, it follows from $\eqref{3.36a}$ that, 
for $t\leq t_0$, 
\begin{equation*}
\frac{{\rm d}v(t,\triangle s)}{{\rm d}t}
=\frac{g(\bar{u}(t,s_0+\triangle s))}{\bar{u}(t,s_0+\triangle s)- s_0}\, v(t,\triangle s),
\qquad\,\,\,  v(0,\triangle s)=\triangle s.
\end{equation*}
This implies that 
\begin{equation}\label{3.38}
v(t_0,\triangle s)=\triangle s \,
\exp\Big\{\int_0^{t_0}\frac{g(\bar{u}(t,s_0+\triangle s))}
{\bar{u}(t,s_0+\triangle s)- s_0}\,{\rm d}t\Big\}.
\end{equation}
Since $\bar{u}(t,s)$ is locally bounded and $g(u)\in C(\mathbb{R})$ is right-Lipschitz in $u$, it follows from $\eqref{3.38}$ that 
\begin{equation*}
0 \leq \frac{v(t_0,\triangle s)}{\triangle s} \leq e^{Lt_0}
\qquad \mbox{\rm for\ some $L>0$}.
\end{equation*}
This implies that, for any $t_0 \in \mathbb{R}^+$, 
$\bar{u}(t_0,\cdot)$ is Lipschitz continuous in $s$ at $s_0 \in \partial A$.
	
Since $\bar{u}(t,s)$ solves ODE $\eqref{3.3}$ for any $s\in\mathbb{R}$, 
\begin{align*}
&\ |\bar{u}(t_0+\triangle t,s_0+\triangle s)-\bar{u}(t_0,s_0)| \nonumber\\[1mm]
&\leq |\bar{u}(t_0+\triangle t,s_0+\triangle s)-\bar{u}(t_0,s_0+\triangle s)|
+|\bar{u}(t_0,s_0+\triangle s)-\bar{u}(t_0,s_0)|\nonumber\\[1mm]
&\longrightarrow 0\qquad\,\,\, 
{\rm as}\,\, \triangle t,\, \triangle s\rightarrow 0.
\end{align*}
By the arbitrariness of $t_0 \in \mathbb{R}^+$ and  $s_0 \in \partial A$, this implies that
$\bar{u}(t,s)$ is continuous at any point $(t,s)\in\mathbb{R}^+ \times \partial A$. 
Thus, $\bar{u}(t,s)\in C(\mathbb{R}^+ \times \mathbb{R})$.

Furthermore, because $\bar{u}(t,s)\in C(\mathbb{R}^+ \times \mathbb{R})$, for any $s_0 \in \partial A$, 
$$\mathop{\lim}\limits_{\triangle s \rightarrow 0} \bar{u}(t,s_0+\triangle s)=\bar{u}(t,s_0)= s_0.
$$ 
By $\eqref{3.38}$, this implies that, for any $t>0$, 
$$
\mathop{\lim}\limits_{\triangle s \rightarrow 0}\frac{ v(t,\triangle s)}{\triangle s}\ {\rm exists}
\qquad {\rm if\ and\ only\ if}\qquad 
\mathop{\lim}\limits_{s \rightarrow s_0} \frac{g(s)}{s- s_0}\ {\rm exists\ or\ equals}\ {-}\infty,
$$
so that 
\begin{align*}
\bar{u}_s(t,s_0)=e^{g'(s_0)t}\qquad\,\,\, {\rm if}\ g'(s_0)\ {\rm exists\ or\ equals}\ {-}\infty.
\end{align*}
Therefore, for any $t>0$ and $s_0 \in \mathring {\partial} A$,
\begin{align*}
\triangle \bar{u}&=\frac{\bar{u}(t+\triangle t,s_0+\triangle s)-\bar{u}(t,s_0+\triangle s)}{\triangle t}\, \triangle t
 +\frac{\bar{u}(t,s_0+\triangle s)-\bar{u}(t,s_0)}{\triangle s} \, \triangle s\nonumber\\
&=\big(g(\bar{u}(t,s_0+\triangle s))+o(1)\big){\triangle t}+\big(e^{g'(s_0)t}+o(1)\big)\triangle s
\nonumber\\
&=\big(g(\bar{u}(t,s_0))+o(1)\big){\triangle t}+\big(e^{g'(s_0)t}+o(1)\big)\triangle s\nonumber\\[1mm]
&=e^{g'(s_0)t}\triangle s +o(1)(|\triangle t|+|\triangle s|).
\end{align*}	
Thus, $\bar{u}(t,s)$ is differentiable at $(t,s)\in\mathbb{R}^+ \times \{s_0\}$ for $s_0 \in \mathring {\partial} A$.

Therefore, $\bar{u}(t,s)\in C(\mathbb{R}^+ \times \mathbb{R})$ satisfies (i)--(iii) 
and is $C^1$ on $\mathbb{R}^+ \times (\mathbb{R}\setminus (\partial A-\mathring {\partial}A))$. 
This completes the proof of Lemma $\ref{lem:3.1}$.
\end{proof}

\section{Riemann Solutions with Non-Selfsimilar Structure: Proof of Theorem $\ref{th:4.4}$}

The proof of Theorem $\ref{th:4.4}$ is divided into three steps.

\medskip
\noindent
{\bf 1.} We first have the following lemma to determine the types of the basic waves, 
{\it i.e.}, shock waves and rarefaction waves, 
emitting from the initial discontinuity $M(x)=0$.

\begin{lemma}\label{lem:4.1}
Under the assumptions in {\rm Theorem} $\ref{th:4.4}$, the following properties hold{\rm :}
\begin{itemize}
\item [{\rm (i)}] If $u_->u_+$, the solution of the Riemann problem $\eqref{eq1.1}$--$\eqref{eq1.2}$ 
can not be continuous{\rm ;}

\item [{\rm (ii)}] If $u_-<u_+$, the characteristics of the Riemann problem $\eqref{eq1.1}$--$\eqref{eq1.2}$ 
do not intersect.
\end{itemize}
\end{lemma}

\begin{proof} If $(t,x)=(t,x_1,x_2,\cdots,x_n)$ for $t>0$ is a point 
on the characteristic emitting from the region $\pm M(x)>0$, 
then, by $\eqref{4.1}$, there exists a unique point $x^0=(x^0_1,x^0_2,\cdots,x^0_n)$ 
with $\pm M(x^0_1,x^0_2,\cdots,x^0_n)>0$ such that
$$
x_i(t)=x^0_i+\chi_i(t,u_\pm) \qquad {\rm for}\,\, i=1,2,\cdots,n.
$$
Thus,
\begin{equation*}
\pm M(x_1-\chi_1(t,u_\pm),x_2-\chi_2(t,u_\pm),\cdots,x_n-\chi_n(t,u_\pm))>0.
\end{equation*}
Then the characteristic families emitting from the regions
$M(x)<0$ and $M(x)>0$ cover the spaces $\Omega_-$ and $\Omega_+$, respectively,	
where $\Omega_-$ and $\Omega_+$ are given by
\begin{equation}\label{4.6}
\Omega_-=\big\{(t,x)\, :\, M(x_1-\chi_1(t,u_-),x_2-\chi_2(t,u_-),\cdots,
x_n-\chi_n(t,u_-))<0,\,\, t>0\big\},
\end{equation}
and
\begin{equation}\label{4.7}
\Omega_+=\big\{(t,x)\, :\, M(x_1-\chi_1(t,u_+),x_2-\chi_2(t,u_+),\cdots,
x_n-\chi_n(t,u_+))>0,\,\, t>0\big\}.
\end{equation}

For any fixed $(t,x)$ with $t>0$, we define $Q(s)$ by
\begin{equation}\label{4.8}
	Q(s):=M(x_1-\chi_1(t,s),x_2-\chi_2(t,s),\cdots,x_n-\chi_n(t,s)).
\end{equation}
We claim: 
{\it If $Q(s_0)=0$ for some $s_0$ lying between $u_-$ and $u_+$, then 
\begin{align}\label{Qs0}
(s-s_0)Q(s)<0 \qquad \mbox{ for small  $|s-s_0|>0$ with $s$ lying between $u_-$ and $u_+$}.
\end{align}
}

Without loss of generality, assume that $s-s_0>0$ is small. 
By Lemma $\ref{lem:3.1}$, $\bar {u}(t,\cdot)$ is a nondecreasing Lipschitz continuous function in $s$.
Then $\bar {u}(t,s_0)\leq \bar {u}(t,s)$ for $t\geq 0$. 
Thus, by the continuity of $\bar{u}(t,s)$ and $\bar{u}(0,s)=s$,  
there exists $\epsilon >0$ such that, for any $t \in (0,\epsilon)$,
\begin{equation}\label{4.9}
\bar {u}(t,s_0)<\frac{s_0+s}{2}<\bar {u}(t,s).
\end{equation}
Noticing that $Q(s_0)=0$, it follows from {\it Condition} $(\mathcal{H})$, $\eqref{4.1}$, 
and $\eqref{4.9}$ that 
\begin{align}\label{4.10}
Q(s)
&=-\big(M_{x_i}(x^0(s_0))+o(1)\big)\,\big(\chi_i(t,s)-\chi_i(t,s_0)\big)\nonumber\\
& =-\big(M_{x_i}(x^0(s_0))+o(1)\big)\,
\int^t_0 \big(f'_i(\bar{u}(\tau,s))-f'_i(\bar{u}(\tau,s_0))\big)\, {\rm d}\tau\nonumber\\
& =-(1+o(1))\int^t_0
M_{x_i}(x^0(s_0))\, f''_i(\xi(\tau))\big(\bar{u}(\tau,s))-\bar{u}(\tau,s_0)\big)\, {\rm d}\tau<0,
\end{align}
where $\xi(\tau)\in (a,b)$ and $x^0(s_0):=(\cdots,x_i-\chi_i(t,s_0),\cdots)$ satisfies $M(x^0(s_0))=Q(s_0)=0$.
Thus, the claim follows.

\smallskip
(i)\, For the case that $u_->u_+$, it suffices to show that 
$\Omega_-\cap \Omega_+\not=\emptyset$, {\it i.e.}, 
the characteristics emitting from regions $M(x)<0$ and $M(x)>0$, respectively, 
intersect each other.
	
On the contrary, assume that $\Omega_-\cap \Omega_+=\emptyset$. 
Since $\Omega_-$ and $\Omega_+$ are both open sets in $\mathbb{R}^+ \times \mathbb{R}^n$, 
there exists a point $(t,x)\in \mathbb{R}^+ \times \mathbb{R}^n$ such that 
$(t,x)\not\in \Omega_-\cup \Omega_+$, {\it i.e.}, 
\begin{equation*}
M(x_1-\chi_1(t,u_-),x_2-\chi_2(t,u_-),\cdots,x_n-\chi_n(t,u_-))\geq 0,
\end{equation*}
and
\begin{equation*}
M(x_1-\chi_1(t,u_+),x_2-\chi_2(t,u_+),\cdots,x_n-\chi_n(t,u_+))\leq 0.
\end{equation*}
By $\eqref{4.8}$, for such a point $(t,x)$, $Q(u_-) \geq 0$ and $Q(u_+) \leq 0$.  
This contradicts to $\eqref{Qs0}$. 
Therefore, $\Omega_-\cap \Omega_+ \neq \emptyset$ and then the solution is discontinuous.

 \smallskip
(ii)\, For the case that $u_-<u_+$, it suffices to show that $\Omega_-\cap \Omega_+ =\emptyset$.	

On the contrary, assume that $\Omega_-\cap \Omega_+ \not=\emptyset$. 
Then there exists a point $(t,x)\in \Omega_-\cap \Omega_+$ such that $Q(u_-)<0$ and $Q(u_+)>0$, 
which contradicts to $\eqref{Qs0}$. 
Therefore, $\Omega_-\cap \Omega_+=\emptyset$. 
Thus, the characteristics emitting from regions $M(x)<0$ and $M(x)>0$, respectively, 
do not intersect each other.
\end{proof}

\smallskip
\noindent
{\bf 2.} We now prove the case of shock waves in the following lemma:

\begin{lemma}\label{lem:4.2}
Under the assumptions in {\rm Theorem} $\ref{th:4.4}$,
if $u_->u_+$, the solution of the Riemann problem $\eqref{eq1.1}$--$\eqref{eq1.2}$ 
is a shock wave with non-selfsimilar structure given by $\eqref{4.13}$--$\eqref{4.14}$, 
{\it i.e.}, 
\begin{equation*}
u(t,x)=
\begin{cases}
 \bar{u}(t,u_-) \,\,\,&{\rm if}\ M(x_1-[\chi_1](t),x_2-[\chi_2](t),\cdots,x_n-[\chi_n](t))<0,
\\[1mm]
 \bar{u}(t,u_+) \,\,\,&{\rm if}\ M(x_1-[\chi_1](t),x_2-[\chi_2](t),\cdots,x_n-[\chi_n](t))>0,
\end{cases}
\end{equation*}
and the discontinuity $S(t,x)=0$, which is given by
\begin{equation*}
M(x_1-[\chi_1](t),x_2-[\chi_2](t),\cdots,x_n-[\chi_n](t))=0,
\end{equation*}
satisfies the Rankine-Hugoniot condition $\eqref{2.3}$ and the geometric entropy conditions $\eqref{2.4}$--$\eqref{2.5}$.
\end{lemma}

\begin{proof} Firstly, it is direct to check from $\eqref{4.14}$ that
\begin{equation*}
S_t=-M_{x_i}(x^0)\frac{[f_i]_\pm}{[u]_\pm}
(t),\quad
S_{x_i}=M_{x_i}(x^0)\qquad\,\, {\rm for}\,\,i=1,2,\cdots,n,
\end{equation*}
where $x^0=(x_1-[\chi_1](t),\cdots,x_n-[\chi_n](t))$.

Choose the unit normal vector $\vec{n}$ of the discontinuity $S(t,x)=0$ as
\begin{equation*}
\vec{n}= \alpha\, (S_t,S_{x_1},\cdots,S_{x_n}),
\end{equation*}
where $\alpha =-\big((S_t)^2+ \sum(S_{x_i})^2\big)^{-\frac{1}{2}}<0$. 
Then 
$$
\frac{\partial}{\partial \vec{n}}S(t,x)=\vec{ n} \cdot(S_t,S_{x_1},\cdots,S_{x_n})
=\frac{1}{\alpha}<0.
$$
Thus, the unit normal vector $\vec{n}$ points from region $S(t,x)>0$ to region $S(t,x)<0$, {\it i.e.},  
from the side of $\bar{u}(t,u_+)$ to the side of $\bar{u}(t,u_-)$.
Then it is direct to see that $u(t,x)$ given by $\eqref{4.13}$ satisfies 
the Rankine-Hugoniot condition $\eqref{2.3}$ along the discontinuity $S(t,x)=0$ given by $\eqref{4.14}$ 
when $u_l>u_r$ keeps, 
where $u_l$ and $u_r$ are given by
$$
u_l=\bar{u}(t,u_-),\qquad u_r=\bar{u}(t,u_+).
$$
Moreover, $S(0,x)=0$ coincides with the initial discontinuity $M(x)=0$.

The remaining proof is to check that $u(t,x)$ given by $\eqref{4.13}$ 
satisfies the geometric entropy conditions $\eqref{2.4}$--$\eqref{2.5}$ 
along discontinuity $S(t,x)=0$ given by $\eqref{4.14}$ when $u_l>u_r$ keeps. 
Let $\xi_0\in [\bar{u}(t,u_+),\bar{u}(t,u_-)]$ be determined by
$$
M_{x_i}(x^0)f'_i(\xi_0)
=\dfrac{ M_{x_i}(x^0) f_i(\bar{u}(t,u_-))- M_{x_i}(x^0)f_i(\bar{u}(t,u_+))}
{\bar{u}(t,u_-)-\bar{u}(t,u_+)}
= M_{x_i}(x^0)\frac{[f_i]_\pm}{[u]_\pm}
(t).
$$

(i)\, If $k\in [\bar{u}(t,u_+),\bar{u}(t,u_-)]$ with $k\geq\xi_0$, then 
\begin{align}\label{nkul}
&\ \vec{n}\cdot (k-u_l,f_1(k)-f_1(u_l),\cdots,f_n(k)-f_n(u_l))\nonumber\\
&=\alpha \,(- M_{x_i}(x^0)\frac{[f_i]_\pm}{[u]_\pm}
(t), M_{x_1}(x^0),\cdots,M_{x_n}(x^0))\nonumber\\
&\quad\,\,\,\,\, \cdot (k-\bar{u}(t,u_-),f_1(k)-f_1(\bar{u}(t,u_-)),\cdots, f_n(k)-f_n(\bar{u}(t,u_-)))\nonumber\\
&=\alpha\, (k-\bar{u}(t,u_-)) M_{x_i}(x^0)\,
\Big(\frac{f_i(k)-f_i(\bar{u}(t,u_-))}{k-\bar{u}(t,u_-)}-\frac{[f_i]_\pm}{[u]_\pm}
(t)\Big)\nonumber\\
&=(-\alpha)\, (\bar{u}(t,u_-)-k) M_{x_i}(x^0)\,\big(f'_i(\xi_1)-f'_i(\xi_0)\big)\nonumber\\
&=|\alpha|\, (\bar{u}(t,u_-)-k)\int_0^1 M_{x_i}(x^0)f''_i(\xi_0+\theta(\xi_1-\xi_0))\,{\rm d}\theta\,
(\xi_1-\xi_0)
\geq 0,
\end{align}
where $\xi_1$ with $\xi_0\leq k\leq \xi_1\leq \bar{u}(t,u_-)$ is determined by
$$
M_{x_i}(x^0)f'_i(\xi_1)
=\dfrac{ M_{x_i}(x^0) f_i(k)- M_{x_i}(x^0)f_i(\bar{u}(t,u_-))}{k-\bar{u}(t,u_-)}
=M_{x_i}(x^0)\frac{ f_i(k)-f_i(\bar{u}(t,u_-))}{k-\bar{u}(t,u_-)}.
$$
	
(ii)\, If $k\in [\bar{u}(t,u_+),\bar{u}(t,u_-)]$ with $k\leq \xi_0$, then
\begin{align}\label{nkur}
&\ \vec{n}\cdot (k-u_l,f_1(k)-f_1(u_l),\cdots,f_n(k)-f_n(u_l))\nonumber\\[1mm]
&=\vec{n}\cdot (k-u_r,f_1(k)-f_1(u_r),\cdots,f_n(k)-f_n(u_r))\nonumber\\[1mm]
&=\vec{n}\cdot (k-\bar{u}(t,u_+),f_1(k)-f_1(\bar{u}(t,u_+)),\cdots,f_n(k)-f_n(\bar{u}(t,u_+)))\nonumber\\
&= \alpha\, (k-\bar{u}(t,u_+))  M_{x_i}(x^0)\,
\Big(\frac{f_i(k)-f_i(\bar{u}(t,u_+))}{k-\bar{u}(t,u_+)}-\frac{[f_i]_\pm}{[u]_\pm}
(t)\Big)\nonumber\\
&= \alpha\, (k-\bar{u}(t,u_+)) M_{x_i}(x^0)\,\big(f'_i(\xi_2)-f'_i(\xi_0)\big)\nonumber\\
&=(-\alpha)\, (k-\bar{u}(t,u_+)) \int_0^1 M_{x_i}(x^0)f''_i(\xi_2+\theta(\xi_0-\xi_2))\,{\rm d}\theta\,
(\xi_0-\xi_2)
\geq  0,
\end{align}
where $\xi_2$ with $\bar{u}(t,u_+)\leq\xi_2\leq k\leq\xi_0$ is determined by
$$
M_{x_i}(x^0) f'_i(\xi_2)
=\frac{  M_{x_i}(x^0)f_i(k)- M_{x_i}(x^0)f_i(\bar{u}(t,u_+))}{k-\bar{u}(t,u_+)}
=M_{x_i}(x^0) \dfrac{ f_i(k)-f_i(\bar{u}(t,u_+))}{k-\bar{u}(t,u_+)}.
$$	

This completes the proof of Lemma $\ref{lem:4.2}$.
\end{proof}

\smallskip
\noindent
{\bf 3.} We now prove the case of rarefaction waves in the following lemma:

\begin{lemma}\label{lem:4.3}
Under the assumptions in {\rm Theorem} $\ref{th:4.4}$,
if $u_-<u_+$, the solution of the Riemann problem $\eqref{eq1.1}$--$\eqref{eq1.2}$ 
is a rarefaction wave with non-selfsimilar structure given by $\eqref{4.17}$--$\eqref{4.18}$, 
{\it i.e.},
\begin{equation*}
u(t,x)=
\begin{cases}
\bar{u}(t,u_-) \,\,&{\rm if}\ M(x_1-\chi_1(t,u_-),\cdots,x_n-\chi_n(t,u_-))<0,\,\,t\geq 0;\\
\bar{u}(t,c(t,x)) \,\,&{\rm if}\ M(x_1-\chi_1(t,u_-),\cdots,x_n-\chi_n(t,u_-))\geq 0,\\
&\quad M(x_1-\chi_1(t,u_+),\cdots,x_n-\chi_n(t,u_+))\leq 0,\,\, t>0;\\
\bar{u}(t,u_+) \,\,&{\rm if}\ M(x_1-\chi_1(t,u_+),\cdots,x_n-\chi_n(t,u_+))>0,\,\, t\geq 0,
\end{cases}
\end{equation*}
where $c(t,x)$ is the implicit function uniquely determined by 
\begin{equation*}
F(t,x,c):=M(x_1-\chi_1(t,c),x_2-\chi_2(t,c),\cdots,x_n-\chi_n(t,c))=0.
\end{equation*}
\end{lemma}

\begin{proof} By Lemma $\ref{lem:4.1}$, $\Omega_-\cap \Omega_+=\emptyset$. Then
\begin{equation*}
u(t,x)=
\begin{cases}
\bar{u}(t,u_-)\quad & {\rm if}\ (t,x)\in \Omega_-,\\
\bar{u}(t,u_+)\quad & {\rm if}\ (t,x)\in \Omega_+,
\end{cases}
\end{equation*}
where $\Omega_-$ and $\Omega_+$ are given by $\eqref{4.6}$ and $\eqref{4.7}$, respectively. 

Let $\Omega$ be the region between $\Omega_-$ and $\Omega_+$. Then
$$
\Omega=\{(t,x)\,:\, F(t,x,u_-)\geq 0 \,\,\mbox{and}\,\, F(t,x,u_+)\leq 0\,\, \mbox{for } t>0\}.
$$

We now show that $\Omega$ can be fulfilled by the rarefaction wave determined by $\eqref{4.18}$.
In fact, for any point $(t_0,x_0)\in \Omega$, 
$$
F(t_0,x_0,u_-)\geq 0,\qquad F(t_0,x_0,u_+)\leq 0.
$$
Then there exists $c=c(t_0,x_0)\in [u_-,u_+]$ such that
\begin{equation*}
{F}(t_0,x_0,c(t_0,x_0))=0.
\end{equation*}	
It follows from $\eqref{Qs0}$
that $c=c(t_0,x_0)$ is the unique root of $F(t_0,x_0,c)=0$.
	
By the arbitrariness of $(t_0,x_0)\in \Omega$, $c=c(t,x)$ is the implicit function 
uniquely determined by $F(t,x,c)=0$ on the whole region $\Omega$. 
Clearly, $u_-\leq c(t,x)\leq u_+$.
	
Noticing that $M(x)\in C^1(\mathbb{R})$ and $f_i(u)\in C^2(\mathbb{R})$, 
it follows from $\bar{u}(t,s)\in C(\mathbb{R}^+ \times \mathbb{R})$ and 
$\bar{u}(t,s)\in C^1(\mathbb{R}^+\times \mathbb{R}\setminus (\partial A-\mathring {\partial}A))$ that
$$
F(t,x,c)\in C(\mathbb{R}^+ \times \mathbb{R}^{n}\times \mathbb{R})\cap 
C^1(\mathbb{R}^+\times \mathbb{R}^n \times (\mathbb{R}\setminus (\partial A-\mathring {\partial}A))).
$$
Then, by the implicit function theorem, we conclude that
$c(t,x)\in C(\Omega)$, and $c(t,x)$ is differentiable at $(t,x)\in \Omega$ if $c(t,x) \notin \partial A-\mathring {\partial}A$. 
Therefore, $\bar{u}(t,c(t,x))$ belongs to $C(\Omega)$, and 
$\bar{u}(t,c(t,x))$ is differentiable at $(t,x)\in\Omega$ if $c(t,x) \notin \partial A-\mathring {\partial}A.$
	
For any point $(t,x)\in {\rm int}\, \Omega$, if $c=c(t,x) \notin \partial A-\mathring {\partial}A$, then 
$$F_t=- M_{x_i}(x^0)\, f'_i(\bar{u}(t,c)),\qquad\,\,\, F_{x_i}=M_{x_i}(x^0), 
$$ 
and 
\begin{equation}\label{4.21}
F_c=-\int_0^t M_{x_i}(x^0)\, f''_i(\bar{u}(\tau,c))\,\bar{u}_s(\tau,c)\, {\rm d}\tau<0.
\end{equation}
This implies
\begin{equation}\label{4.23}
c_t=-\frac{F_t}{F_c}=\frac{ M_{x_i}(x^0)\, f'_i(\bar{u}(t,c))}{F_c},\qquad\,\,\, 
c_{x_i}=-\frac{F_{x_i}}{F_c}=-\frac{M_{x_i}(x^0)}{F_c}.
\end{equation}
Then, by the chain rule, 
\begin{align*}
&\partial_t\big(\bar{u}(t,c(t,x))\big)=\frac{{\rm d}\bar{u}}{{\rm d}t}(t,c(t,x))+\bar{u}_s(t,c(t,x))\, c_t(t,x),\\
&\partial_{x_i}\big(\bar{u}(t,c(t,x))\big)=\bar{u}_s(t,c(t,x))\, c_{x_i}(t,x).
\end{align*}
Thus, by $\eqref{4.23}$, 
\begin{align*}
\frac{\partial\bar{u} }{\partial t}+\frac{\partial }{\partial x_i}f_i(\bar{u})
&=g(\bar{u})+\bar{u}_s c_t+ f'_i(\bar{u})\bar{u}_s c_{x_i}\nonumber\\
&=g(\bar{u})+\bar{u}_s \big(c_t+ f'_i(\bar{u})c_{x_i}\big)\nonumber\\
&=g(\bar{u})+\frac{\bar{u}_s}{F_c} \big( M_{x_i}(x^0)f'_i(\bar{u})- f'_i(\bar{u})M_{x_i}(x^0)\big)\nonumber\\
&=g(\bar{u}).
\end{align*}

Finally, if $(t,x)$ lies on the surface $F(t,x,u_-)=0$, 
then $(t,x,u_-)$ satisfies $\eqref{4.18}$ with $c=u_-$ so that	$u(t,x)=\bar{u}(t,u_-)$; 
and if $(t,x)$ lies on the surface $F(t,x,u_+)=0$, 
then $(t,x,u_+)$ satisfies $\eqref{4.18}$ with $c=u_+$ so that $u(t,x)=\bar{u}(t,u_+)$. 
If $c=c(t,x) \in \partial A-\mathring {\partial}A$, 
then $u(t,x)=\bar{u}(t,c(t,x))$ for $(t,x)$ lying on the surface 
$M(x_1-\chi_1(t,c(t,x)),x_2-\chi_2(t,c(t,x)),\cdots, x_n-\chi_n(t,c(t,x)))=0$. 
This means that the rarefaction wave defined by $\eqref{4.17}$--$\eqref{4.18}$ is continuous 
on $\partial \Omega$ and the surfaces determined by 
$c(t,x)=c$ with $c\in [u_-,u_+]\cap(\partial A-\mathring {\partial}A)$.
 
This completes the proof of Lemma $\ref{lem:4.3}$.
\end{proof}

Combining Lemma $\ref{lem:4.2}$ with Lemma $\ref{lem:4.3}$, we have proved Theorem $\ref{th:4.4}$.

\smallskip
Finally, based on the proof above, we have the following remarks: 

\begin{remark}\label{rem:6.1}
If $u_-$ and $u_+$ belong to the same interval $I_i=(a_i,b_i)$ 
with $u_-> u_+$ and finite $A_i(s)$ $(${\it resp.,} $B_i(s)$$)$ in $\eqref{3.18}$, 
then $\bar{u}(t,u_-)=\bar{u}(t,u_+)=a_i$ $(${\it resp.,} $b_i$$)$ for sufficiently large $t$ 
so that the shock wave \emph{disappears in a finite time}, 
which is essentially different from the case of the Lipschitz source term. 
\end{remark}

\begin{remark}
If $u_-$ and $u_+$ belong to the same interval $I_i=(a_i,b_i)$ 
with $u_-< u_+$ and finite $A_i(s)$ $(${\it resp.,} $B_i(s)$$)$ in $\eqref{3.18}$, 
then $\bar{u}(t,u_-)=\bar{u}(t,u_+)=\bar{u}(t,c(t,x))=a_i$ $(${\it resp.,} $b_i$$)$ 
for sufficiently large $t$ 
so that the rarefaction wave \emph{disappears in a finite time}, 
which is also essentially different from the case of the Lipschitz source term. 
\end{remark}

\begin{remark}\label{rem:6.3}
If $u_-,u_+$ do not belong to the same interval $I_i$ or $J_j$, 
even though the solution $\bar{u}(t,c(t,x))$ is always continuous, 
$\bar{u}(t,c(t,x))$ may not be $C^1$ on the surfaces determined by 
$c(t,x)=c$ with $c\in [u_-,u_+]\cap(\partial A-\mathring {\partial}A)$. 
It is caused by the singularity of the source term $g(u)$ on $\partial A-\mathring {\partial}A$. 
In fact, by {\rm Lemma} $\ref{lem:3.1}$, 
if $s_0 \in \partial A-\mathring {\partial}A$ or $g'(s_0)=-\infty$, 
then $\bar{u}(t,s)$ is not differentiable in $s$ at the same point $s_0$. 
For such cases, \emph{the rarefaction wave consists of more than one smooth rarefaction wave pieces}, 
which are connected by weak discontinuities. 
This phenomenon never occurs in the case of the Lipschitz source term.
\end{remark}

\begin{remark}\label{rem:11}
In {\rm Theorem} $\ref{th:4.4}$, {\it Condition} $(\mathcal{H})$ can be extended to the case that, 
\begin{equation}
\begin{array}{l}
\mbox{$H(x,u)\geq 0$ for $x\in\mathbb{R}^n$ with $M(x)=0$ and $u\in (a,b)$ such that}\\ 
\mbox{$H(x,u)=0$ holds for at most countable points of $(x,u)$.}
\end{array}
\end{equation}
In fact, it follows from $\eqref{4.10}$ that $\eqref{Qs0}$ holds, 
the geometric entropy conditions $\eqref{nkul}$--$\eqref{nkur}$ clearly hold, 
and the implicit function $c=c(t,x)$ for $(t,x)\in \Omega$ can be uniquely determined by $F(t,x,c)=0$.
For $\eqref{4.21}$, it always holds that $F_c\leq 0$, 
where $F_c=0$ if and only if $g(c)=0$ and $H(x^0,c)=0$, 
in which case the solution $u(t,x)=\bar{u}(t,c(t,x))$ is continuous but not differentiable, 
so that the rarefaction wave has a weak discontinuity. 
\end{remark}

\section{Two Examples}
In this section, we provide two examples to demonstrate the uniqueness of basic waves 
when the source term is non-Lipschitz from left (necessarily right-Lipschitz), 
and the non-uniqueness of basic waves when the source term 
is non-Lipschitz from right, respectively.

\subsection{Riemann solutions with a right-Lipschitz source term} 
Consider the $2$-D Riemann problem: 
\begin{equation}\label{6.1}
\begin{cases}
 u_t+(\frac{u^2}{2})_x+(\frac{u^4}{4})_y=-u^{\frac{1}{3}},\\[2mm]
u(0,x,y)=
\begin{cases}
u_-\qquad{\rm if}\,\, x^3+y<0,
\\[2mm]
u_+\qquad{\rm if}\,\, x^3+y>0.
\end{cases}
\end{cases}
\end{equation}

Since $g(u)=-u^{\frac{1}{3}}$ satisfies that $g'(0)={-}\infty$, 
the source term $-u^{\frac{1}{3}}$ of the equation in $\eqref{6.1}$ is right-Lipschitz but not left-Lipschitz. 
By applying Theorem $\ref{th:4.4}$ to the Riemann problem $\eqref{6.1}$, 
the global Riemann solutions with non-selfsimilar structures, 
including shock waves and rarefaction waves, 
are obtained for various cases. 
For this example, we see that the Riemann solution is actually unique 
when the source term is non-Lipschitz from left (but necessarily right-Lipschitz).
Moreover, we can clearly see the new phenomenon that the shock waves and rarefaction waves 
\emph{disappear in a finite time} since the solutions decay to zero in a finite time.

\smallskip
Solving the characteristic ODEs of the equation in $\eqref{6.1}$, we obtain
\begin{equation*}
\bar{u}(t,s)=
\begin{cases}
{\rm sgn}(s) \big(|s|^{\frac{2}{3}}-\frac{2}{3}t\big)^{\frac{3}{2}}
\quad&{\rm if}\ t< \frac{3}{2}|s|^{\frac{2}{3}}, 
\\[2mm]
0\quad&{\rm if}\ t \geq \frac{3}{2}|s|^{\frac{2}{3}}.
\end{cases}
\end{equation*}
For $\eqref{4.1}$, it is direct to check that
\begin{equation*}
\chi_1(t,s)=\begin{cases}
\frac{3}{5}{\rm sgn}(s)\big(|s|^{\frac{5}{3}}-(|s|^{\frac{2}{3}}-\frac{2}{3}t)^{\frac{5}{2}}\big)
\quad& {\rm if}\ t< \frac{3}{2}|s|^{\frac{2}{3}},\\[2mm]
\frac{3}{5}{\rm sgn}(s)|s|^{\frac{5}{3}}
\quad&{\rm if}\ t \geq \frac{3}{2}|s|^{\frac{2}{3}},
\end{cases}
\end{equation*}
\begin{equation*}
\chi_2(t,s)=\begin{cases}
\frac{3}{11}{\rm sgn}(s)
\big(|s|^{\frac{11}{3}}-(|s|^{\frac{2}{3}}-\frac{2}{3}t)^{\frac{11}{2}}\big)
& {\rm if}\ t< \frac{3}{2}|s|^{\frac{2}{3}},
\\[2mm]
\frac{3}{11}{\rm sgn}(s)|s|^{\frac{11}{3}}
&{\rm if}\ t \geq \frac{3}{2}|s|^{\frac{2}{3}}.
\end{cases}
\end{equation*}
Moreover, for $\eqref{4.3}$--$\eqref{4.4}$, we have
\begin{equation*}
\begin{cases}
 [\chi_1](t)=\frac{1}{2}\big(\chi_1(t,u_-)+\chi_1(t,u_+)\big),\\[2mm]
 [\chi_2](t)=\frac{1}{4}\big(\chi_2(t,u_-)+\chi_2(t,u_+)+P(t)\big),
\end{cases}
\end{equation*}
where $P(t)$ is given by
\begin{equation*}
P(t):=\int_0^t \big(\bar{u}^2(t,u_-)\bar{u}(t,u_+)+\bar{u}(t,u_-)\bar{u}^2(t,u_+)\big)\, {\rm d}\tau.
\end{equation*}

For the Riemann problem $\eqref{6.1}$, {\it Condition} $(\mathcal{H})$ in $\eqref{convexh}$ 
\begin{equation}\label{6.7a}
H(x,y,u)=3x^2+3u^2>0
\end{equation}
holds, except for $(x,u)=(0,0)$.
By the arguments in Remark $\ref{rem:11}$, 
we can use Theorem $\ref{th:4.4}$ to solve the Riemann problem $\eqref{6.1}$ case by case.

\smallskip
\noindent
{\bf 1.} We first consider the case that $u_-> u_+$. 
By Theorem $\ref{th:4.4}$, the solution is a shock wave. 
Based on the relation between $|u_-|$ and $|u_+|$, there are three subcases.

\smallskip
\noindent
{\bf 1.1.} For the case that $|u_-|> |u_+|$, the solution is given by (see Fig. $\ref{fig1}$):

\begin{figure}
	\begin{center}
		{\includegraphics[width=0.6\columnwidth]{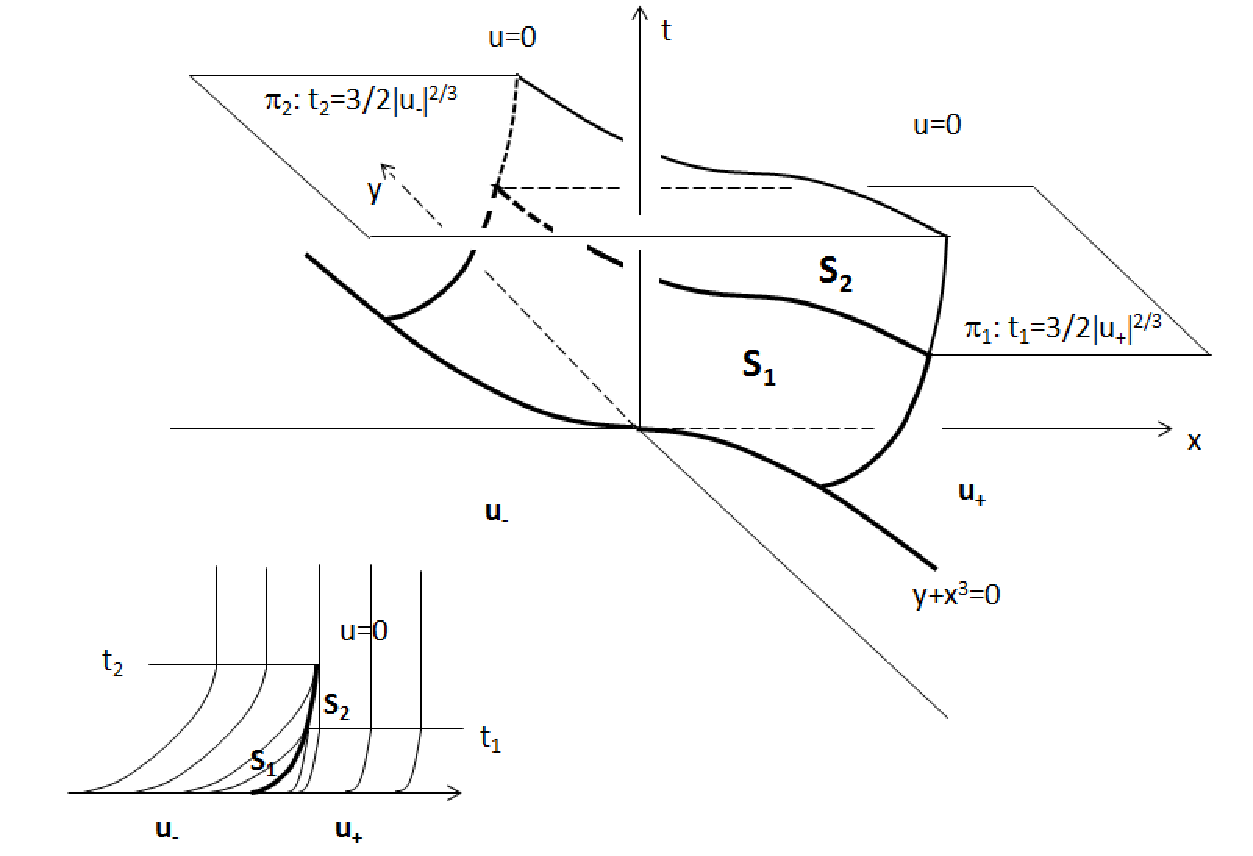}}
  \caption{$u_->u_+\,\, {\rm with}\,\,|u_-|>|u_+|.$
		}\label{fig1}
	\end{center}
\end{figure}

(i)\, When $0 \leq t \leq \frac{3}{2}|u_+|^{\frac{2}{3}}$, 
\begin{equation*}
u(t,x,y)=
\begin{cases}
{\rm sgn}(u_-)\,\big(|u_-|^{\frac{2}{3}}-\frac{2}{3}t\big)^{\frac{3}{2}}
\quad&{\rm if}\ S_1(t,x,y)<0,\\[2mm]
 {\rm sgn}(u_+)\, \big(|u_+|^{\frac{2}{3}}-\frac{2}{3}t\big)^{\frac{3}{2}}
 \quad&{\rm if}\ S_1(t,x,y)>0,
\end{cases}
\end{equation*}
and the shock surface $S_1(t,x,y)=0$ is given by
$$\Big\{x-\frac{1}{2}\big(\chi_1(t,u_-)+\chi_1(t,u_+)\big)\Big\}^3
+y-\frac{1}{4}\big(\chi_2(t,u_-)+\chi_2(t,u_+)+P(t)\big)=0.$$

(ii)\, When $\frac{3}{2}|u_+|^{\frac{2}{3}} \leq t \leq \frac{3}{2}|u_-|^{\frac{2}{3}}$, 
\begin{equation*}
u(t,x,y)=
\begin{cases}
 {\rm sgn}(u_-)\,\big(|u_-|^{\frac{2}{3}}-\frac{2}{3}t\big)^{\frac{3}{2}}
 \quad&{\rm if}\ S_2(t,x,y)<0,\\[2mm]
 0\quad&{\rm if}\ S_2(t,x,y)>0,
\end{cases}
\end{equation*}
and the shock surface $S_2(t,x,y)=0$ is given by
$$\Big\{x-\frac{1}{2}\big(\chi_1(t,u_-)+\chi_1(\tfrac{3}{2}|u_+|^{\frac{2}{3}},u_+)\big)\Big\}^3
+y-\frac{1}{4}\big(\chi_2(t,u_-)+
\chi_2(\tfrac{3}{2}|u_+|^{\frac{2}{3}},u_+)+
P(\tfrac{3}{2}|u_+|^{\frac{2}{3}})\big)=0.$$

(iii)\, When $t \geq \frac{3}{2}|u_-|^{\frac{2}{3}}$, the solution $u(t,x,y)\equiv 0$.

\begin{remark}
For the case that $u_->u_+$ with $|u_-|>|u_+|$, 
the shock surface is given by $S_1(t,x,y)=0$ when $0 \leq t \leq \frac{3}{2}|u_+|^{\frac{2}{3}}$, 
and by $S_2(t,x,y)=0$ when $\frac{3}{2}|u_+|^{\frac{2}{3}} \leq t \leq \frac{3}{2}|u_-|^{\frac{2}{3}}$. 
These two shock surfaces connect each other at $t=\frac{3}{2}|u_+|^{\frac{2}{3}} $ continuously. 
When $t \geq \frac{3}{2}|u_-|^{\frac{2}{3}}$, the solution $u(t,x,y)\equiv 0$ 
so that the shock wave disappears.
\end{remark}

\begin{figure}
	\begin{center}
		{\includegraphics[width=0.6\columnwidth]{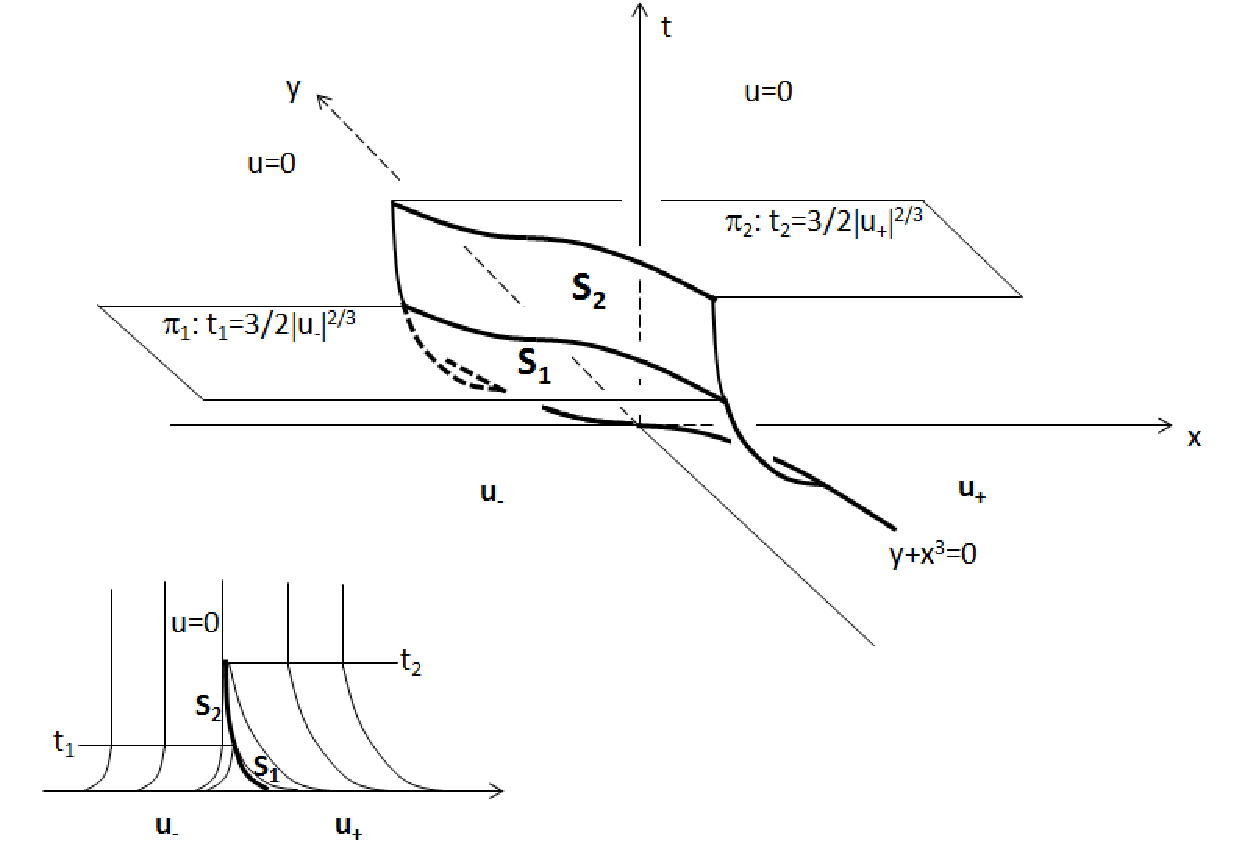}}
  \caption{$u_->u_+\,\, {\rm with}\,\,|u_-|<|u_+|.$
		}\label{fig2}
	\end{center}
\end{figure}

\smallskip
\noindent
{\bf 1.2.} For the case that $|u_-|<|u_+|$, the solution is given by (see Fig. $\ref{fig2}$):

(i)\, When $0 \leq t \leq \frac{3}{2}|u_-|^{\frac{2}{3}}$, 
\begin{equation*}
u(t,x,y)=
\begin{cases}
 {\rm sgn}(u_-)\, \big(|u_-|^{\frac{2}{3}}-\frac{2}{3}t\big)^{\frac{3}{2}}
 \quad &{\rm if}\ S_1(t,x,y)<0,\\[2mm]
 {\rm sgn}(u_+)\, \big(|u_+|^{\frac{2}{3}}-\frac{2}{3}t\big)^{\frac{3}{2}}
 \quad &{\rm if}\ S_1(t,x,y)>0,
\end{cases}
\end{equation*}
and the shock surface $S_1(t,x,y)=0$ is given by
$$\Big\{x-\frac{1}{2}\big(\chi_1(t,u_-)+\chi_1(t,u_+)\big)\Big\}^3
+y-\frac{1}{4}\big(\chi_2(t,u_-)+\chi_2(t,u_+)+P(t)\big)=0.$$

(ii)\, When $\frac{3}{2}|u_-|^{\frac{2}{3}} \leq t \leq \frac{3}{2}|u_+|^{\frac{2}{3}}$, 
\begin{equation*}
u(t,x,y)=
\begin{cases}
0\quad&{\rm if}\ S_2(t,x,y)<0,\\[1mm]
 {\rm sgn}(u_+) \big(|u_+|^{\frac{2}{3}}-\frac{2}{3}t\big)^{\frac{3}{2}}
 \quad&{\rm if}\ S_2(t,x,y)>0,
\end{cases}
\end{equation*}
and the shock surface $S_2(t,x,y)=0$ is given by
$$\Big\{x-\frac{1}{2}\big(\chi_1(\tfrac{3}{2}|u_-|^{\frac{2}{3}},u_-)+\chi_1(t,u_+)\big)\Big\}^3
+y-\frac{1}{4}\big(\chi_2(\tfrac{3}{2}|u_-|^{\frac{2}{3}},u_-)+
\chi_2(t,u_+)+P(\tfrac{3}{2}|u_-|^{\frac{2}{3}})\big)=0.$$

(iii)\, When $t \geq \frac{3}{2}|u_+|^{\frac{2}{3}}$, the solution $u(t,x,y)\equiv 0$.

\begin{remark}
For the case that $u_->u_+$ with $|u_-|<|u_+|$, 
the shock surface is given by $S_1(t,x,y)=0$ when $0 \leq t \leq \frac{3}{2}|u_-|^{\frac{2}{3}}$, 
and by $S_2(t,x,y)=0$ when $\frac{3}{2}|u_-|^{\frac{2}{3}} \leq t \leq \frac{3}{2}|u_+|^{\frac{2}{3}}$. 
These two shock surfaces connect each other at $t=\frac{3}{2}|u_-|^{\frac{2}{3}}$ continuously. 
When $t \geq \frac{3}{2}|u_+|^{\frac{2}{3}}$, the solution $u(t,x,y)\equiv 0$, 
so that the shock wave disappears.
\end{remark}

\smallskip
\noindent
{\bf 1.3.} For the case that $|u_-|=|u_+|$, it follows from $u_->u_+$ that $u_-=-u_+>0$. 
Then $[\chi_1](t)\equiv 0$ and $[\chi_2](t)\equiv 0$. 
Therefore, the shock surface $x^3+y=0$ stays still until $t = \frac{3}{2}|u_+|^{\frac{2}{3}}$, after which $u(t,x,y) \equiv 0$ 
so that the shock wave disappears.

\medskip
\noindent
{\bf 2.} We now consider the case that $u_-< u_+$. 
By Theorem $\ref{th:4.4}$, the solution is a rarefaction wave. 
Since the structure of the rarefaction wave for the case that $ 0<u_-< u_+$ or $u_-< u_+<0$ 
is just a part of the case that $u_-<0< u_+$, to make our illustration short and without loss of generality,
we only give the solution for the case that $u_-<0< u_+$.

By Theorem $\ref{th:4.4}$, the solution is given by (see Fig. $\ref{fig3}$):
\begin{equation*}
u(t,x,y)=
\begin{cases}
-\big(|u_-|^{\frac{2}{3}}-\frac{2}{3}t\big)^{\frac{3}{2}} 
\quad&{\rm if}\ \big\{x+\frac{3}{5}\big(|u_-|^{\frac{5}{3}}
-\big(|u_-|^{\frac{2}{3}}-\frac{2}{3}t\big)^{\frac{5}{2}}\big)\big\}^3
+y\\[2mm]
&\quad\, +\frac{3}{11}\big(|u_-|^{\frac{11}{3}}
-\big(|u_-|^{\frac{2}{3}}-\frac{2}{3}t\big)^{\frac{11}{2}}\big)< 0,\,  
0 \leq t\leq \frac{3}{2}|u_-|^{\frac{2}{3}},
\\[2mm]
0 &{\rm if}\ \big(x+\frac{3}{5}|u_-|^{\frac{5}{3}}\big)^3
+y+\frac{3}{11}|u_-|^{\frac{11}{3}}< 0,\, t\geq \frac{3}{2}|u_-|^{\frac{2}{3}},
\\[2mm]
u_-(t,x,y)&{\rm if}\ (t,x,y)\in \Omega^-_R,
\\[2mm]
0&{\rm if}\ (t,x,y)\in \Omega^0_R,
\\[2mm]
u_+(t,x,y)&{\rm if}\ (t,x,y)\in \Omega^+_R,
\\[2mm]
\big(|u_+|^{\frac{2}{3}}-\frac{2}{3}t\big)^{\frac{3}{2}} 
&{\rm if}\ \big\{x-\frac{3}{5}\big(|u_+|^{\frac{5}{3}}
-\big(|u_+|^{\frac{2}{3}}-\frac{2}{3}t\big)^{\frac{5}{2}}\big)\big\}^3+y
\\[2mm]
 & \quad\, -\frac{3}{11}\big(|u_+|^{\frac{11}{3}}
 -\big(|u_+|^{\frac{2}{3}}-\frac{2}{3}t\big)^{\frac{11}{2}}\big)> 0,\, 
  0 \leq t\leq \frac{3}{2}|u_+|^{\frac{2}{3}},
\\[2mm]
0 &{\rm if}\ \big(x-\frac{3}{5}|u_+|^{\frac{5}{3}}\big)^3
+y-\frac{3}{11}|u_+|^{\frac{11}{3}}> 0,\, t \geq \frac{3}{2}|u_+|^{\frac{2}{3}},
\end{cases}
\end{equation*}
where $u_-(t,x,y)=-\big(|c_-(t,x,y)|^{\frac{2}{3}}-\frac{2}{3}t\big)^{\frac{3}{2}}$ in $\Omega^-_R$, 
and $c_-=c_-(t,x,y)<0$ is the unique root of the following equation
\begin{equation*}
\Big\{x+\frac{3}{5}\big({-}c^{\frac{5}{3}}-\big(c^{\frac{2}{3}}-\tfrac{2}{3}t\big)^{\frac{5}{2}}\big)\Big\}^3
+y+\frac{3}{11}\Big({-}c^{\frac{11}{3}}-\big(c^{\frac{2}{3}}-\tfrac{2}{3}t\big)^{\frac{11}{2}}\Big)=0;
\end{equation*}
and $u_+(t,x,y)=\big(|c_+(t,x,y)|^{\frac{2}{3}}-\frac{2}{3}t\big)^{\frac{3}{2}}$ in $\Omega^+_R$, 
and $c_+=c_+(t,x,y)>0$ is the unique root of the following equation:
\begin{equation*}
\Big\{x-\frac{3}{5}\big(c^{\frac{5}{3}}-\big(c^{\frac{2}{3}}-\tfrac{2}{3}t\big)^{\frac{5}{2}}\big)\Big\}^3
+y-\frac{3}{11}\Big(c^{\frac{11}{3}}-\big(c^{\frac{2}{3}}-\tfrac{2}{3}t\big)^{\frac{11}{2}}\Big)=0.
\end{equation*}
Here, the regions $\Omega^\pm_R$ and $\Omega^0_R$ are given by
$$
\Omega^\pm_R=\left\{(t,x,y)\,\left|
\,\begin{array}{l}
\mp\Big(\big\{x\mp\frac{3}{5}\big(|u_\pm|^{\frac{5}{3}}-\big(|u_\pm|^{\frac{2}{3}}-\frac{2}{3}t\big)^{\frac{5}{2}}\big)\big\}^3+y
\\[2mm]
\qquad\mp\frac{3}{11}\big(|u_\pm|^{\frac{11}{3}}-\big(|u_\pm|^{\frac{2}{3}}-\frac{2}{3}t\big)^{\frac{11}{2}}\big)\Big)\geq 0,
\\[2mm]
\pm\Big(\big(x\mp\frac{3}{5}\big(\frac{2}{3}t\big)^{\frac{5}{2}}\big)^3+y\mp\frac{3}{11}\big(\frac{2}{3}t\big)^{\frac{11}{2}}\Big)> 0, 
 \\[2mm]
0 < t< \frac{3}{2}|u_\pm|^{\frac{2}{3}}.
\end{array}\right.
\right\},
$$
and

\begin{figure}
	\begin{center}
		{\includegraphics[width=0.6\columnwidth]{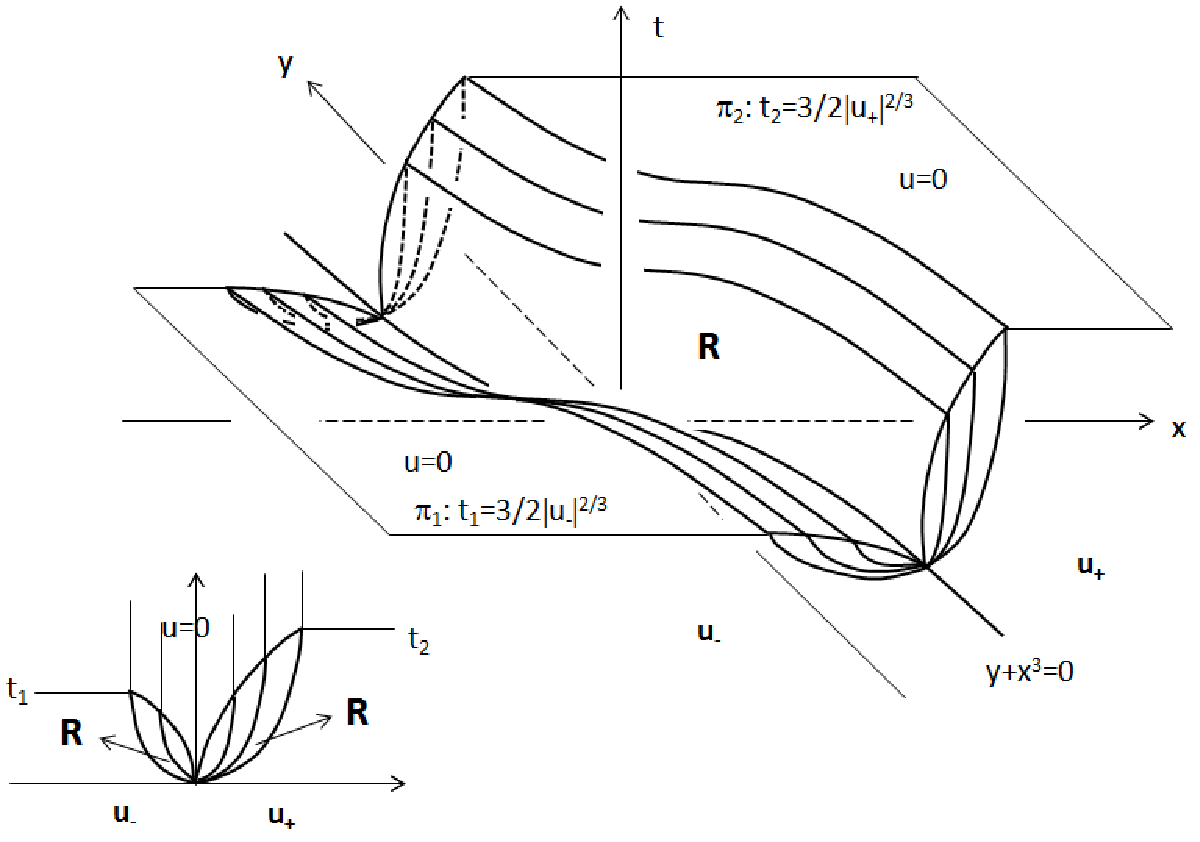}}
  \caption{$u_-<0<u_+.$
		}\label{fig3}
	\end{center}
\end{figure}

$$
\Omega^0_R=\left\{(t,x,y)\,\left|\,\begin{array}{l}
\big(x+\frac{3}{5}\big(\frac{2}{3}t\big)^{\frac{5}{2}}\big)^3
+y+\frac{3}{11}\big(\frac{2}{3}t\big)^{\frac{11}{2}}\geq 0,
\\[2mm]
\big (x-\frac{3}{5}\big(\frac{2}{3}t\big)^{\frac{5}{2}}\big)^3
+y-\frac{3}{11}\big(\frac{2}{3}t\big)^{\frac{11}{2}}\leq 0,
\\[2mm]
\big(x+\frac{3}{5}|u_-|^{\frac{5}{3}}\big)^3+y+\frac{3}{11}|u_-|^{\frac{11}{3}}\geq 0,
\\[2mm]
\big(x-\frac{3}{5}|u_+|^{\frac{5}{3}}\big)^3+y-\frac{3}{11}|u_+|^{\frac{11}{3}}\leq 0,
\end{array}\right.
\right\}.
$$

\begin{remark}
For the case that $u_-< u_+$, 
the solution is a rarefaction wave in $\Omega^-_R\cup \Omega^+_R$, 
where the characteristics emitting from regions $\{x^3+y>0\}$ and $\{x^3+y<0\}$ can not reach.
It is interesting to see that the rarefaction wave \emph{decays to zero} at the interfaces of $\partial \Omega^-_R\cap \partial \Omega^0_R$ and $\partial \Omega^+_R \cap \partial \Omega^0_R$, 
and remains zero after that.
\end{remark}

\smallskip
\subsection{Non-uniqueness of Riemann solutions with a left-Lipschitz source term} 
Consider the $2$-D Riemann problem: 
\begin{equation}\label{6.13}
\begin{cases}
 u_t+(\frac{u^2}{2})_x+(\frac{u^4}{4})_y=u^{\frac{1}{3}},\\[2mm]
u(0,x,y)=
\begin{cases}
u_-\qquad{\rm if}\ x^3+y<0,\\[1mm]
u_+\qquad{\rm if}\ x^3+y>0.
\end{cases}
\end{cases}
\end{equation}

Since $g(u)=u^{\frac{1}{3}}$ satisfies that $g'(0)=\infty$, 
the source term $u^{\frac{1}{3}}$ of the equation in $\eqref{6.13}$ is not right-Lipschitz.
So it does not satisfy the assumptions in Theorem $\ref{th:4.4}$. 
From this example, we see that the Riemann solution is \emph{not unique} when the source term is not right-Lipschitz. 
In fact, since the source term $u^{\frac{1}{3}}$ is not right-Lipschitz at $u=0$, 
the solution of ODE $\eqref{3.3}$ with initial data $s=0$ is not unique, 
and so does the characteristic of the solution of ODE $\eqref{3.1}$. 

Based on such observation on the non-uniqueness of the solutions of characteristic ODEs, 
by using the method of characteristic analysis, 
we can construct several different solutions for the same Riemann initial data, 
all of which satisfy the Rankine-Hugoniot condition and the geometric entropy conditions.

\smallskip
Solving the characteristic ODEs of the equation in $\eqref{6.13}$, we obtain 
\begin{equation}\label{6.16.1}
\bar{u}(t,s) ={\rm sgn}(s)\, \big(|s|^{\frac{2}{3}}+\frac{2}{3}t\big)^{\frac{3}{2}}
\qquad{\rm if}\ s\neq 0;
\end{equation}
and if $s=0$, there exist at least three different solutions of 
the characteristic ODE $\eqref{3.3}$, 
\begin{equation}\label{6.16.2}
 \bar{u}(t,0) =\big(\frac{2}{3}t\big)^{\frac{3}{2}},\quad 
 -\big(\frac{2}{3}t\big)^{\frac{3}{2}},\quad 
 {\rm or}
 \,\,\,\, 0.
\end{equation}
For $\eqref{4.1}$, if $s\neq0$, 
\begin{equation*}
\begin{cases}
 \chi_1(t,s) =\frac{3}{5}{\rm sgn}(s)\,
 \big(\big(|s|^{\frac{2}{3}}+\frac{2}{3}t\big)^{\frac{5}{2}}-|s|^{\frac{5}{3}}\big), \\[2mm]
 \chi_2(t,s) =\frac{3}{11}{\rm sgn}(s)\,
 \big(\big(|s|^{\frac{2}{3}}+\frac{2}{3}t\big)^{\frac{11}{2}}-|s|^{\frac{11}{3}}\big),
\end{cases}
\end{equation*}
and if $s=0$,  corresponding to \eqref{6.16.2}, there exist at least three different solutions of 
the characteristic ODE $\eqref{3.1}$, 
\begin{equation*}
\begin{cases}
\chi_1(t,0) =\frac{3}{5}\big(\frac{2}{3}t\big)^{\frac{5}{2}},\quad\,\,\,\; 
-\frac{3}{5}\big(\frac{2}{3}t\big)^{\frac{5}{2}}, \quad\,\,\,\, 
\mbox{or}
\,\,\,\, 0;\\[2mm]
\chi_2(t,0) =\frac{3}{11}\big(\frac{2}{3}t\big)^{\frac{11}{2}},\quad 
-\frac{3}{11}\big(\frac{2}{3}t\big)^{\frac{11}{2}}, \quad 
\mbox{or}
\,\,\,\, 0.
\end{cases}
\end{equation*}
Moreover, for $\eqref{4.3}$--$\eqref{4.4}$, we have 
\begin{equation*}
\begin{cases}
 [\chi_1](t)=\frac{1}{2}\big(\chi_1(t,u_-)+\chi_1(t,u_+)\big),\\[2mm]
 [\chi_2](t)=\frac{1}{4}\big(\chi_2(t,u_-)+\chi_2(t,u_+)+\tilde{P}(t)\big),
\end{cases}
\end{equation*}
where $\tilde{P}(t)$ is given by
\begin{equation*}
\tilde{P}(t):=
\int_0^t \big(\bar{u}^2(\tau,u_-)\bar{u}(\tau,u_+)+\bar{u}(\tau,u_-)\bar{u}^2(\tau,u_+)\big) \, {\rm d}\tau.
\end{equation*}

As shown in $\eqref{6.7a}$, 
the Riemann problem $\eqref{6.13}$ satisfies {\it Condition} $(\mathcal{H})$.
Noticing that the source term $u^{\frac{1}{3}}$ is Lipschitz 
on any interval $[a,\infty)$ with $a>0$ and $({-}\infty, b]$ with $b<0$, respectively. 
Then, by $\eqref{6.16.1}$, the solution is unique 
if the Riemann initial data in $\eqref{6.13}$ satisfy 
$$
\min \{u_-, u_+\}>0\qquad {\rm or}\qquad \max \{u_-, u_+\}<0.
$$
However, the solution is not unique 
if $0$ belongs to the closed interval bounded by $u_-$ and $u_+$.
Based on such observation, as an example, we construct different shock solutions 
for the case that $u_->u_+=0$, 
and different rarefaction waves for the case that $u_-< u_+=0$.

\smallskip
\noindent
{\bf 1.} We first consider the case that $u_->u_+=0$.
Similar to the proof of Lemma $\ref{lem:4.1}$, 
it is easy to see that the solution must contain a discontinuity. 
Based on the three choices of $\bar{u}(t,0)$ in $\eqref{6.16.2}$, 
we can construct three shock waves as follows (see Fig. $\ref{fig4}$):

\smallskip
{\it Solution} 1: By choosing $u_l=\bar{u}(t,u_-)>0$ and $u_r=\bar{u}(t,0)=0$, 
the solution is 
\begin{equation*}
u(t,x,y)=
\begin{cases}
 \big(|u_-|^{\frac{2}{3}}+\frac{2}{3}t\big)^{\frac{3}{2}}\quad&{\rm if}\ S_1(t,x,y)<0,
\\[2mm]
 0\quad&{\rm if}\ S_1(t,x,y)>0,
\end{cases}
\end{equation*}
and the shock surface $S_1(t,x,y)=0$ is given by
$$
\big(x-\frac{1}{2}\chi_1(t,u_-)\big)^3+y-\frac{1}{4}\chi_2(t,u_-)=0.
$$

\smallskip
{\it Solution} 2: By choosing $u_l=\bar{u}(t,u_-)>0$ 
and $u_r=\bar{u}(t,0)=-(\frac{2}{3}t)^{\frac{3}{2}}$, 
the solution is
\begin{equation*}
u(t,x,y)=
\begin{cases}
 \big(|u_-|^{\frac{2}{3}}+\frac{2}{3}t\big)^{\frac{3}{2}}\quad&{\rm if}\ S_2(t,x,y)<0,
\\[2mm]
 -\big(\frac{2}{3}t\big)^{\frac{3}{2}}\quad&{\rm if}\ S_2(t,x,y)>0,
\end{cases}
\end{equation*}
and the shock wave $S_2(t,x,y)=0$ is given by
$$
\Big\{x-\frac{1}{2}\big(\chi_1(t,u_-)-\tfrac{3}{5}\big(\tfrac{2}{3}t\big)^{\frac{5}{2}}\big)\Big\}^3
+y-\frac{1}{4}\Big(\chi_2(t,u_-)-\tfrac{3}{11}\big(\tfrac{2}{3}t\big)^{\frac{11}{2}}+\tilde{P}(t)\Big)=0,
$$
where 
$\tilde{P}(t)$ is given by
$$
\tilde{P}(t)=\int_0^t \Big({-}\big(\tfrac{2}{3}\tau\big)^{\frac{3}{2}}\big(u_-^{\frac{2}{3}}+\tfrac{2}{3}\tau\big)^3
+\big(\tfrac{2}{3}\tau\big)^3\big(u_-^{\frac{2}{3}}+\tfrac{2}{3}\tau\big)^{\frac{3}{2}}\Big)
\, {\rm d}\tau.
$$

\begin{figure}
	\begin{center}
		{\includegraphics[width=0.7\columnwidth]{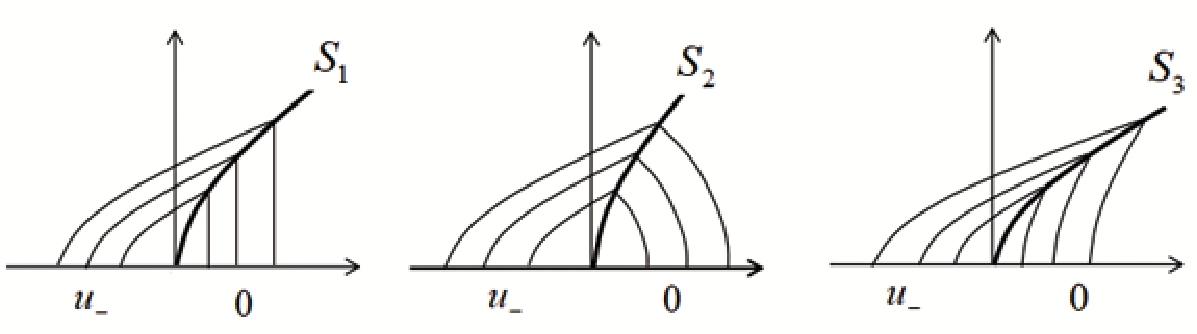}}
  \caption{$u_->0\,\, {\rm with}\,\, u_+=0.$
		}\label{fig4}
	\end{center}
\end{figure}

\smallskip
{\it Solution} 3: By choosing $u_l=\bar{u}(t,u_-)>0$ 
and $u_r=\bar{u}(t,0)=\big(\frac{2}{3}t\big)^{\frac{3}{2}}$, 
the solution is
\begin{equation*}
u(t,x,y)=
\begin{cases}
 \big(|u_-|^{\frac{2}{3}}+\frac{2}{3}t\big)^{\frac{3}{2}}\quad&{\rm if}\ S_3(t,x,y)<0,
\\[2mm]
 \big(\frac{2}{3}t\big)^{\frac{3}{2}}\quad&{\rm if}\ S_3(t,x,y)>0,
\end{cases}
\end{equation*}
and the shock wave $S_3(t,x,y)=0$ is given by
$$
\Big\{x-\frac{1}{2}\big(\chi_1(t,u_-)+\tfrac{3}{5}\big(\tfrac{2}{3}t\big)^{\frac{5}{2}}\big)\Big\}^3
+y-\frac{1}{4}\Big(\chi_2(t,u_-)+\tfrac{3}{11}\big(\tfrac{2}{3}t\big)^{\frac{11}{2}}+\tilde{P}(t)\Big)=0,
$$
where 
$\tilde{P}(t)$ is given by
$$
\tilde{P}(t)=\int_0^t 
\Big(\big(\tfrac{2}{3}\tau\big)^{\frac{3}{2}}\big(u_-^{\frac{2}{3}}+\tfrac{2}{3}\tau\big)^3
+\big(\tfrac{2}{3}\tau\big)^3\big(u_-^{\frac{2}{3}}+\tfrac{2}{3}\tau\big)^{\frac{3}{2}}\Big) 
\, {\rm d}\tau.
$$

\smallskip
\noindent
{\bf 2.} We now consider the case that $u_-<u_+=0$.
Similar to the proof of Lemma $\ref{lem:4.1}$, 
it is easy to see that the characteristics 
emitting from regions $M(x)<0$ and $M(x)>0$, respectively, do not intersect.
Based on the three choices of $\bar{u}(t,0)$ in $\eqref{6.16.2}$, 
we can construct three rarefaction waves as follows (see Fig. $\ref{fig5}$):

\smallskip
{\it Solution} 1: By choosing $u_l=\bar{u}(t,u_-)<0$ and $u_r=\bar{u}(t,0)=0$, 
the solution is
\begin{equation*}
u(t,x,y)=\left\{\begin{array}{ll}
-\big(|u_-|^{\frac{2}{3}}+\frac{2}{3}t\big)^{\frac{3}{2}} 
\quad&{\rm if}\ \big\{x+\frac{3}{5}\big(\big(|u_-|^{\frac{2}{3}}+\frac{2}{3}t\big)^{\frac{5}{2}}
-|u_-|^{\frac{5}{3}}\big)\big\}^3+y\\[2mm]
 &\quad\,\, +\frac{3}{11}\big({-}|u_-|^{\frac{11}{3}}+\big(|u_-|^{\frac{2}{3}}
 +\frac{2}{3}t\big)^{\frac{11}{2}}\big)< 0,\, t>0,
\\[2mm]
-\big(|c|^{\frac{2}{3}}+\frac{2}{3}t\big)^{\frac{3}{2}}&{\rm if}\ (t,x,y)\in \Omega_R,
\\[2mm]
-\big(\frac{2}{3}(t-t_0)\big)^{\frac{3}{2}}&
{\rm if}\ \big\{x+\frac{3}{5}
\big(\frac{2}{3}(t-t_0)\big)^{\frac{5}{2}}\big\}^3
 +y+\frac{3}{11}\big(\frac{2}{3}(t-t_0)\big)^{\frac{11}{2}}=0\\[1mm]
&\quad\,\,\mbox{for some } t_0\geq 0, \,\mbox{and}\,\,\, (t,x,y)\in \Omega^0_-,
\\[2mm]
0 &{\rm if}\ x^3+y>0,
\end{array}\right.
\end{equation*}
where $c=c(t,x,y)<0$ on $\Omega_R$ is the unique root of the following equation:
\begin{equation}\label{phi}
\Phi(t,x,y,c):=\Big\{x+\frac{3}{5}\big(\big(c^{\frac{2}{3}}+\tfrac{2}{3}t\big)^{\frac{5}{2}}
+c^{\frac{5}{3}}\big)\Big\}^3
+y+\frac{3}{11}\Big(\big(c^{\frac{2}{3}}
+\tfrac{2}{3}t\big)^{\frac{11}{2}}+c^{\frac{11}{3}}\Big)=0.
\end{equation}

\begin{figure}
	\begin{center}
		{\includegraphics[width=0.7\columnwidth]{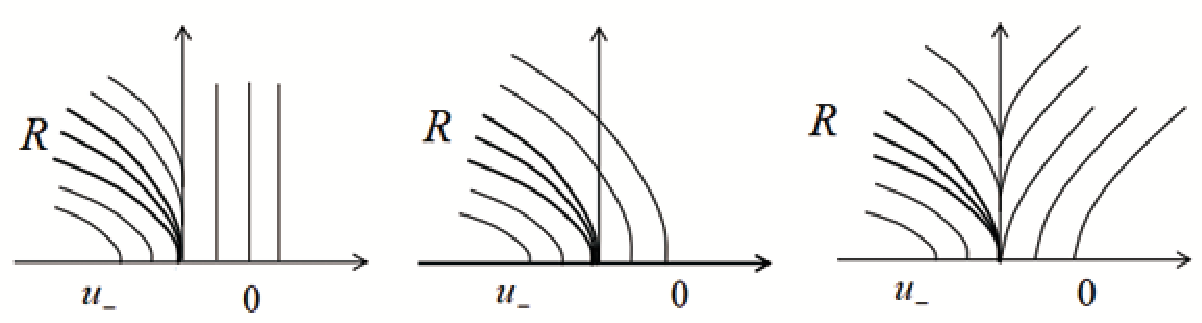}}
  \caption{$u_-<0, u_+=0.$
		}\label{fig5}
	\end{center}
\end{figure}

\smallskip
{\it Solution} 2: By choosing $u_l=\bar{u}(t,u_-)<0$ 
and $u_r=\bar{u}(t,0)=-\big(\frac{2}{3}t\big)^{\frac{3}{2}}$,  
the solution is
\begin{equation*}
u(t,x,y)=\left\{\begin{array}{ll}
-\big(|u_-|^{\frac{2}{3}}+\frac{2}{3}t\big)^{\frac{3}{2}} 
\quad &{\rm if}\ \big\{x+\frac{3}{5}\big(\big(|u_-|^{\frac{2}{3}}+\frac{2}{3}t\big)^{\frac{5}{2}}
-|u_-|^{\frac{5}{3}}\big)\big\}^3+y
\\[2mm]
& \quad\,\, +\frac{3}{11}\big({-}|u_-|^{\frac{11}{3}}+\big(|u_-|^{\frac{2}{3}}
+\frac{2}{3}t\big)^{\frac{11}{2}}\big)< 0, \, t>0,
\\[2mm]
-\big(|c|^{\frac{2}{3}}+\frac{2}{3}t\big)^{\frac{3}{2}} &{\rm if}\ (t,x,y)\in \Omega_R,
\\[2mm]
-\big(\frac{2}{3}t\big)^{\frac{3}{2}} 
&{\rm if}\ \big(x+\frac{3}{5}\big(\frac{2}{3}t\big)^{\frac{5}{2}}\big)^3
+y+\frac{3}{11}\big(\frac{2}{3}t\big)^{\frac{11}{2}} > 0,
\\[2mm]
\end{array}\right.
\end{equation*}
where $c=c(t,x,y)<0$ on $\Omega_R$ is the unique root of $\Phi(t,x,y,c)=0$ in $\eqref{phi}$.

\smallskip
{\it Solution} 3: By choose $u_l=\bar{u}(t,u_-)<0$ 
and $u_r=\bar{u}(t,0)=\big(\frac{2}{3}t\big)^{\frac{3}{2}}$, 
the solution is
\begin{equation*}
u(t,x,y)=\left\{\begin{array}{ll}
-\big(|u_-|^{\frac{2}{3}}+\frac{2}{3}t\big)^{\frac{3}{2}} 
\quad &{\rm if}\ \big\{x+\frac{3}{5}\big(\big(|u_-|^{\frac{2}{3}}+
\frac{2}{3}t\big)^{\frac{5}{2}}-|u_-|^{\frac{5}{3}}\big)\big\}^3+y\\[2mm]
 & \quad\,\,+\frac{3}{11}\big({-}|u_-|^{\frac{11}{3}}
 +\big(|u_-|^{\frac{2}{3}}+\frac{2}{3}t\big)^{\frac{11}{2}}\big)< 0,\,\, t>0,
\\[2mm]
-\big(|c|^{\frac{2}{3}}+\frac{2}{3}t\big)^{\frac{3}{2}}&{\rm if}\ (t,x,y)\in \Omega_R,
\\[2mm]
-\big(\frac{2}{3}(t-t_0)\big)^{\frac{3}{2}}
&{\rm if}\ 
\big\{x+\frac{3}{5}\big(\frac{2}{3}(t-t_0)\big)^{\frac{5}{2}}\big\}^3
 +y+\frac{3}{11}\big(\frac{2}{3}(t-t_0)\big)^{\frac{11}{2}}= 0
 \\[2mm]
 &\quad\,\, \mbox{for some }  t_0\geq 0, \ \mbox{and } (t,x,y)\in \Omega^0_-,
\\[2mm]
\big(\frac{2}{3}(t-t_0)\big)^{\frac{3}{2}}
&{\rm if}
\big\{x-\frac{3}{5}\big(\frac{2}{3}(t-t_0)\big)^{\frac{5}{2}}\big\}^3
  +y-\frac{3}{11}\big(\frac{2}{3}(t-t_0)\big)^{\frac{11}{2}}= 0 \\[2mm]
 & \quad\,\, \mbox{for some }  t_0\geq 0,\, \, \mbox{and } (t,x,y)\in \Omega^0_+,\\[2mm]
\big(\frac{2}{3}t\big)^{\frac{3}{2}} 
&{\rm if}\ \big(x-\frac{3}{5}\big(\frac{2}{3}t\big)^{\frac{5}{2}}\big)^3
+y-\frac{3}{11}\big(\frac{2}{3}t\big)^{\frac{11}{2}} > 0,
\end{array}\right.
\end{equation*}
where $c=c(t,x,y)<0$ on $\Omega_R$ is the unique root of $\Phi(t,x,y,c)=0$ in $\eqref{phi}$.

All above, regions $\Omega_R$ and $\Omega^0_\pm$ are given by
\begin{equation*}
 \Omega_R=\left\{(t,x,y)\,\left|\,\begin{array}{l}
 \big\{x+\frac{3}{5}\big(\big(|u_-|^{\frac{2}{3}}
 +\frac{2}{3}t\big)^{\frac{5}{2}}-|u_-|^{\frac{5}{3}}\big)\big\}^3+y\\[2mm]
 \quad +\frac{3}{11}\big({-}|u_-|^{\frac{11}{3}}
 +\big(|u_-|^{\frac{2}{3}}+\frac{2}{3}t\big)^{\frac{11}{2}}\big)\geq 0,
\\[2mm]
 \big(x+\frac{3}{5}\big(\frac{2}{3}t\big)^{\frac{5}{2}}\big)^3
 +y+\frac{3}{11}\big(\frac{2}{3}t\big)^{\frac{11}{2}} \leq 0,\, t>0
\end{array}\right.
\right\},  
\end{equation*}

\begin{equation*}
\Omega^0_-=\left\{(t,x,y)\,\Big|\,
\big(x+\frac{3}{5}\big(\tfrac{2}{3}t\big)^{\frac{5}{2}}\big)^3
+y+\frac{3}{11}\big(\tfrac{2}{3}t\big)^{\frac{11}{2}} \geq 0,\, x^3+y \leq 0
\right\},
\end{equation*}

\begin{equation*}
\Omega^0_+=\left\{(t,x,y)\,\Big|\,
 \big(x-\frac{3}{5}\big(\tfrac{2}{3}t\big)^{\frac{5}{2}}\big)^3+y-
\frac{3}{11}\big(\tfrac{2}{3}t\big)^{\frac{11}{2}} \leq 0,\, x^3+y \geq 0
\right\}.
\end{equation*}

\begin{remark}
As shown above, the non-uniqueness of the solutions of the Riemann problem $\eqref{6.13}$ inherits from 
the non-uniqueness of the solutions of the corresponding characteristic ODEs 
with initial data $\bar{u}(0,s)=s=0$. 
Furthermore, since $u^{\frac{1}{3}}$ is not right-Lipschitz at $u=0$, for any $T >0$,
we can choose $\bar{u}(t,0)=0$ if $t\leq T$ 
and $\bar{u}(t,0)=(\frac{2}{3}(t-T))^{\frac{3}{2}}$
or $-(\frac{2}{3}(t-T))^{\frac{3}{2}}$ if $t\geq T$. 
Therefore, there exist \emph{infinitely many} different shock waves and rarefaction waves.
\end{remark}

\medskip
~\\ \textbf{Acknowledgements}.
The research of Gaowei Cao was supported in part by the National Natural Science Foundation
of China No.11701551 and the China Scholarship Council No.20200491
0200.
The research of Gui-Qiang G. Chen was supported in part by 
the UK Engineering and Physical Sciences Research Council Awards 
EP/V008854 and EP/V051121/1.
The research of Wei Xiang was supported in part by the
Research Grants Council of the HKSAR, China (Project No.  CityU 11300021, CityU 11311722, 
CityU 11305523, and CityU 11305625), and HK PolyU Internal project P0045335.
The research of Xiaozhou Yang was supported in part by the NSFC (Grant 11471332).

\end{document}